\documentclass[a4paper,10pt,leqno]{amsart}
\usepackage{hyperref}
\usepackage{enumerate,amssymb}
\usepackage[arrow,curve,matrix,tips,2cell]{xy}
  \SelectTips{eu}{10} \UseTips
  \UseAllTwocells
\usepackage{pdfsync}

\DeclareMathAlphabet\EuR{U}{eur}{m}{n}
\SetMathAlphabet\EuR{bold}{U}{eur}{b}{n}

\DeclareMathOperator{\aut}{Aut}

\DeclareMathOperator*{\colim}{colim}
\DeclareMathOperator{\cok}{cok}
\DeclareMathOperator{\con}{con}

\DeclareMathOperator{\Endo}{end}
\DeclareMathOperator{\ev}{ev}
\DeclareMathOperator{\Ext}{Ext}

\DeclareMathOperator{\id}{id}

\DeclareMathOperator{\im}{im}
\DeclareMathOperator{\ind}{ind}

\DeclareMathOperator{\odd}{odd}

\DeclareMathOperator{\pr}{pr}

\DeclareMathOperator{\rk}{rk}

\DeclareMathOperator{\tors}{torsion}

\DeclareMathOperator{\trf}{trf}

  \newcommand{\IC}{\mathbb{C}}

  \newcommand{\IH}{\mathbb{H}}
  \newcommand{\II}{\mathbb{I}}

  \newcommand{\IQ}{\mathbb{Q}}
   \newcommand{\Q}{\mathbb{Q}}
  \newcommand{\IR}{\mathbb{R}}

  \newcommand{\IZ}{\mathbb{Z}}
  \newcommand{\C}{\mathbb{C}}
  \newcommand{\R}{\mathbb{R}}
  \newcommand{\Z}{\mathbb{Z}}

  \newcommand{\calh}{\mathcal{H}}

  \newcommand{\calp}{\mathcal{P}}

  \newcommand{\bfko}{{\mathbf k}{\mathbf o}}  
   \newcommand{\bfKO}{{\mathbf K}{\mathbf O}}

  \newcommand{\bfp}{{\mathbf p}}

\newcommand{\eub}[1]{\underline{E}#1}

\newcommand{\bub}[1]{\underline{B}#1}

\DeclareMathOperator{\pt}{\ast}
\newcommand{\ptors}{p\text{-}\tors}

\newcommand{\xycomsquare}[8]                      
{\xymatrix{#1 \ar[r]^{#2} \ar[d]^{#4} &
#3 \ar[d]^{#5}  \\
#6\ar[r]^{#7} &
#8
}
}

\newcommand{\xycomsquareminus}[8]                      
{\xymatrix{#1 \ar[r]^-{#2} \ar[d]^-{#4} &
#3 \ar[d]^-{#5}  \\
#6\ar[r]^-{#7} &
#8
}
}

\newcounter{commentcounter}

\newcommand{\version}[1]                       
{\begin{center} last edited on #1 (or later)\\
last compiled on \today\\
name of tex-file: \jobname
\end{center}
}

\theoremstyle{plain}
\newtheorem{theorem}{Theorem}[section]

\newtheorem{lemma}[theorem]{Lemma}
\newtheorem{corollary}[theorem]{Corollary}
\newtheorem{proposition}[theorem]{Proposition}

\theoremstyle{definition}

\newtheorem{condition}[theorem]{Condition}
\newtheorem{example}[theorem]{Example}

\newtheorem{remark}[theorem]{Remark}
\newtheorem{notation}[theorem]{Notation}

\theoremstyle{remark}

\makeatletter\let\c@equation=\c@theorem\makeatother

\title[The topological $K$-theory of certain crystallographic groups]
{The topological $K$-theory of certain crystallographic groups}
              \author{James F. Davis}
              \email{jfdavis@indiana.edu}
              \urladdr{http://www.indiana.edu/~jfdavis/}
              \address{Department of Mathematics\\
              Indiana University\\
              Rawles Hall\\
              831 East 3rd St\\
              Bloomington, IN 47405\\
              U.S.A.}
              \author{Wolfgang L\"uck}
              \email{wolfgang.lueck@him.uni-bonn.de}
                \urladdr{http://www.him.uni-bonn.de/lueck}
              \address{Rheinische Wilhelms-Universit\"at Bonn\\
               Mathematisches Institut\\
               Endenicher Allee 62, 53115 Bonn, Germany}
               \date{March  2011}

\begin{document}

\begin{abstract}
  Let $\Gamma$ be a semidirect product of the form $\IZ^n \rtimes_{\rho} \IZ/p$
  where $p$ is prime and the $\IZ/p$-action $\rho$ on $\IZ^n$ is free away from
  the origin.  We will compute the topological $K$-theory of the real and
  complex group $C^*$-algebra of $\Gamma$ and show that $\Gamma$ satisfies the
  unstable Gromov-Lawson-Rosenberg Conjecture.  On the way we will analyze the
  (co-)homology and the topological $K$-theory of the classifying spaces
  $B\Gamma$ and $\bub{\Gamma}$.  The latter is the quotient of the induced
  $\IZ/p$-action on the torus $T^n$.
\end{abstract}

\maketitle


\typeout{--------------------   Section 0: Introduction --------------------------}

\setcounter{section}{-1}
\section{Introduction}
\label{sec:Introduction}

Let $p$ be a prime.  Let $\rho \colon \IZ/p \to \aut (\IZ^n)
=\mathrm{GL}(n,\IZ)$ be a group homomorphism. Throughout this paper we will
assume:

\begin{condition}[Free conjugation action]
  \label{cond:fre_action}
  The induced action of $\IZ/p$ on $\IZ^n$ is free when restricted to $\IZ^n -
  0$.
\end{condition}

Denote by
\begin{eqnarray}
  \Gamma &  = & \IZ^n \rtimes_{\rho} \IZ/p
  \label{Gamma_is_Zn_rtimes_Z/p}
\end{eqnarray}
the associated semidirect product. Since $\Gamma$ has a finitely generated, free
abelian subgroup which is normal, maximal abelian, and has finite index,
$\Gamma$ is isomorphic to a crystallographic group.  An example of such group
$\Gamma$ is given by $\IZ^{p-1} \rtimes_{\rho} \IZ/p$ where the action $\rho$ is
given by the regular representation $\Z[\Z/p]$ modulo the ideal generated by the
norm element.  When $n =1$ and $p =2$, $\Gamma$ is the infinite dihedral group.

Let $B\Gamma := \Gamma \backslash E\Gamma$ be the classifying space of $\Gamma$.
Denote by $\eub{\Gamma}$ be the classifying space for proper group actions of
$\Gamma$. Let $\bub{\Gamma} = \Gamma \backslash \eub{\Gamma}$.  The space
$\bub{\Gamma}$ is the quotient of the torus $T^n$ under the $\IZ/p$-action
associated to $\rho$.  It is not a manifold, but an orbifold quotient.

To compute the $K$-theory of the $C^*$-algebra, we will use the Baum-Connes
Conjecture which predicts for a group $G$ that the complex and real assembly
maps
\begin{eqnarray*}
  K^{G}_n(\eub{G}) & \xrightarrow{\cong} & K_n(C^*_r(G));
  \\
  KO^{G}_n(\eub{G}) & \xrightarrow{\cong} & KO_n(C^*_r(G;\R)),
\end{eqnarray*}
are bijective for $n \in \IZ$. The point of the Baum-Connes Conjecture is that it identifies the very
hard to compute topological $K$-theory of the group $C^*$-algebra of $G$ to the  better accessible
evaluation at  $\eub{G}$ of the equivariant homology theory given by equivariant
topological $K$-theory. The Baum-Connes Conjecture has been proved for a large class of groups
which includes crystallographic groups (and many more)
in~\cite{Higson-Kasparov(2001)}.  We will later use the composite maps,
where in each case the second map is induction with the projection $\Gamma \to \{1\}$.
\begin{eqnarray*}
  K_m(C^*_r(\Gamma)) 
  \xleftarrow{\cong} & K^{\Gamma}_n(\eub{\Gamma})   \to  & K_m(\bub{\Gamma});  \\
  KO_m(C^*_r(\Gamma;\R)) 
  \xleftarrow{\cong} & KO^{\Gamma}_m(\eub{\Gamma})   \to  & KO_m(\bub{\Gamma}).  
\end{eqnarray*}

Next we describe the main results of this paper. We will show in
Lemma~\ref{lem:preliminaries_about_Gamma_and_Zn_rho}~%
\ref{lem:preliminaries_about_Gamma_and_Zn_rho:ideals} that $k = n/(p-1)$ is an
integer.  Let $\calp$ be the set of conjugacy classes $\{(P)\}$ of finite
non-trivial subgroups of $\Gamma$.

\begin{theorem}[Topological $K$-theory of the complex group $C^*$-algebra]
  \label{the:Topological_K-theory_of_the_complex_group_Cast-algebra} Let 
$\Gamma  = \Z^n \rtimes_\rho \Z/p$ be a group satisfying
  Condition~\ref{cond:fre_action}.
  \begin{enumerate}

  \item \label{the:Topological_K-theory_of_the_complex_group_Cast-algebra:explicite}
    If $ p = 2$
$$K_m(C^*_r(\Gamma)) \cong
\begin{cases}
  \IZ^{3 \cdot 2^{n-1}} & m \; \text{even};
  \\
  0 & m \; \text{odd}.
\end{cases}
$$
If $p$ is odd
$$K_m(C^*_r(\Gamma)) \cong
\begin{cases}
  \IZ^{d_{\ev}} & m \; \text{even};
  \\
  \IZ^{d_{\odd}} & m \; \text{odd};
\end{cases}
$$
where
\begin{eqnarray*}
  d_{\ev}
  & = &
  \frac{2^{(p-1)k} + p -1}{2p} + \frac{(p-1) \cdot p^{k-1}}{2} + (p-1) \cdot p^k;
  \\
  d_{\odd} 
  & = & 
  \frac{2^{(p-1)k} + p -1}{2p} -  \frac{(p-1) \cdot p^{k-1}}{2}.
\end{eqnarray*}
In particular $K_m(C^*_r(\Gamma))$ is always a finitely generated free abelian
group.

\item \label{the:Topological_K-theory_of_the_complex_group_Cast-algebra:K0_abstract}
  There is an exact sequence
$$0 \to \bigoplus_{(P) \in \calp} \widetilde{R}_{\IC}(P) \to 
K_0(C^*_r(\Gamma)) \to K_0(\bub{\Gamma}) \to 0,$$ where $\widetilde{R}_{\IC}(P)$
is the kernel of the map $R_{\IC}(P) \to \IZ$ sending the class $[V]$ of a
complex $P$-representation $V$ to $\dim_{\IC}(\IC \otimes_{\IC P} V)$.

\item \label{the:Topological_K-theory_of_the_complex_group_Cast-algebra:K1_abstract}
  The map
$$K_1(C^*_r(\Gamma)) \xrightarrow{\cong} K_1(\bub{\Gamma})$$
is an isomorphism.  Restricting to the subgroup $\IZ^n$ of $\Gamma$ induces an
isomorphism
$$K_1(C^*_r(\Gamma)) \xrightarrow{\cong} K_1(C^*_r(\IZ^n_\rho))^{\IZ/p}.$$

\end{enumerate}
\end{theorem}

\begin{remark}[Twisted group algebras]\label{rem:relation_to_elpw}
  The computation of
  Theorem~\ref{the:Topological_K-theory_of_the_complex_group_Cast-algebra} has
  already been carried out in the case $p = 2$ and in the case $n = 2$ and 
  $p =3$ in~\cite[Theorem~0.4,
  Example~3.7]{Echterhof-Lueck-Philipps-Walters(2010)}. In view
  of~\cite[Theorem~0.3]{Echterhof-Lueck-Philipps-Walters(2010)} the computation
  presented in this paper yields also computations for the topological
  $K$-theory $K_*(C^*_r(\Gamma,\omega))$ of twisted group algebras for
  appropriate cocycles $\omega$. One may investigate whether the whole program
  of~\cite{Echterhof-Lueck-Philipps-Walters(2010)} can be carried over to the
  more general situation considered in this paper.
\end{remark}

\begin{remark}[Computations by Cuntz and Li]
  \label{rem:Cuntz-Lie}
  Cuntz and Li~\cite{Cuntz-Li(2009integers)} compute the $K$-theory of
  $C^*$-algebras that are associated with rings of integers in number
  fields. They have to make the assumption that the algebraic number field
  contains only $\{\pm 1\}$ as roots of unity. This is related to our
  computation in the case $p = 2$.  Our results, in particular, if we could
  handle instead of a prime $p$ any natural number, may be useful to extend
  their program to the arbitrary case. However, the complexity we already
  encounter in the case of a prime $p$ shows that this is a difficult task.
\end{remark}

We are also interested in the slightly more difficult real case because of
applications to the question whether a closed smooth spin manifold carries a
Riemannian metric with positive scalar curvature (see
Theorem~\ref{the:The_(unstable)_Gromov-Lawson-Rosenberg_Conjecture_holds_for_Gamma}).
The numbers $r_l$ appearing in the next theorem will be defined in~\eqref{r_m}
and analyzed in Subsection~\ref{subsec:On_the_numbers_r_m}.

\begin{theorem}[Topological $K$-theory of the real group $C^*$-algebra]
  \label{the:Topological_K-theory_of_the_real_group_Cast-algebra}
  Let $p$ be an odd prime. Let $\Gamma = \Z^n \rtimes \Z/p$ be a group
  satisfying Condition~\ref{cond:fre_action}. Then for all $m \in \IZ$ :

  \begin{enumerate}

  \item \label{the:Topological_K-theory_of_the_real_group_Cast-algebra:explicite}
$$KO_m(C^*_r(\Gamma;\IR)) \cong
\begin{cases}
  \IZ^{p^k(p-1)/2} \oplus \left(\bigoplus_{l=0}^n KO_{m-l}(\pt)^{r_{l}}\right) &
  m \; \text{even};
  \\
  \bigoplus_{l=0}^n KO_{m-l}(\pt)^{r_{l}} & m \; \text{odd}.
\end{cases}
$$

\item \label{the:Topological_K-theory_of_the_real_group_Cast-algebra:K0_abstract}
  There is an exact sequence
$$0 \to \bigoplus_{(P) \in \calp} \widetilde{KO}_{2m}^{\IZ/p}(\pt) \to 
KO_{2m}(C^*_r(\Gamma;\IR)) \to KO_{2m}(\bub{\Gamma}) \to 0,$$ where
$\widetilde{KO}_m^{\IZ/p}(\pt) = \ker \left( KO_m^{\IZ/p}(\pt) \to
  KO_m(\pt)\right) \cong \IZ^{(p-1)/2}$.  The exact sequence is split after
inverting $p$.

\item \label{the:Topological_K-theory_of_the_real_group_Cast-algebra:K1_abstract}
  The map
$$KO_{2m+1}(C^*_r(\Gamma;\IR)) \xrightarrow{\cong} KO_{2m+1}(\bub{\Gamma})$$
is an isomorphism.  Restricting to the subgroup $\IZ^n$ of $\Gamma$ induces an
isomorphism
$$KO_{2m+1}(C^*_r(\Gamma;\IR)) \xrightarrow{\cong} KO_{2m+1}(C^*_r(\IZ^n_\rho;\IR))^{\IZ/p}.$$

\end{enumerate}
\end{theorem}

If $M$ is a closed spin manifold of dimension $m$ with fundamental group $G$,
one can define an invariant $\alpha(M) \in KO_m(C^*_r(G;\R))$ as the index of a
Dirac operator.  If $M$ admits a metric of positive scalar curvature, then
$\alpha(M)= 0$.  This theory and connections with the Gromov-Lawson-Rosenberg
Conjecture will be reviewed in
Subsection~\ref{subsec:The_Gromov-Lawson-Rosenberg_Conjecture}.

\begin{theorem}[(Unstable) Gromov-Lawson-Rosenberg Conjecture]
  \label{the:The_(unstable)_Gromov-Lawson-Rosenberg_Conjecture_holds_for_Gamma}
  Let $p$ be an odd prime.  Let $M$ be a closed spin manifold of dimension $m
  \geq 5$ and fundamental group $\Gamma$ as defined
  in~\eqref{Gamma_is_Zn_rtimes_Z/p}.  Then $M$ admits a metric of positive
  scalar curvature if and only if $\alpha(M)$ is zero.  Moreover if $m$ is odd,
  then $M$ admits a metric of positive scalar curvature if and only if the
  $p$-sheeted covering associated to the projection $\Gamma \to \IZ/p$ does.
\end{theorem}

\begin{example}\label{exa:pos_scal}
  Here is an example where the last sentence of
  Theorem~\ref{the:The_(unstable)_Gromov-Lawson-Rosenberg_Conjecture_holds_for_Gamma}
  applies.  Choose an odd integer $k > 1$.  Let $M$ be a balanced product $S^k
  \times_\Gamma \R^n$ where $\Gamma$ acts on the sphere via the projection
  $\Gamma \to \Z/p$ and a free action of $\Z/p$ on the sphere and $\Gamma$ acts
  on $\R^n$ via its crystallographic action.  Then its $p$-fold cover $S^k
  \times T^n$ admits a metric of positive scalar curvature since it is a spin
  boundary or since it is a product of a closed manifold with a closed Riemannian manifold with positive scalar curvature, 
and hence $M$ admits a metric of positive scalar curvature.
\end{example}

\begin{remark} 
  Notice that
  Theorem~\ref{the:The_(unstable)_Gromov-Lawson-Rosenberg_Conjecture_holds_for_Gamma}
  is not true for $\IZ^4 \times \IZ/3$ (see Schick~\cite{Schick(1998e)}),
  whereas it is true for $\IZ^4\rtimes_{\rho} \IZ/3$ for appropriate $\rho$ by
  Theorem~\ref{the:The_(unstable)_Gromov-Lawson-Rosenberg_Conjecture_holds_for_Gamma}.
\end{remark}

The computation of the topological $K$-theory of the reduced complex group
$C^*$-algebra $C^*_r(\Gamma)$ and of the reduced real group $C^*$-algebra
$C^*_r(\Gamma;\IR)$ will be done in a sequence of steps, passing in each step to
a more difficult situation.

We will first compute the (co-)homology of $B\Gamma$ and $\bub{\Gamma}$. A
complete answer is given in
Theorem~\ref{the:Cohomology_of_BGamma_and_bub(Gamma)} and
Theorem~\ref{the:Homology_of_BGamma_and_bub(Gamma)}.

Then we will analyze the complex and real topological $K$-cohomology and
$K$-homology of $B\Gamma$ and $\bub{\Gamma}$.  A complete answer is given in
Theorem~\ref{the:K-cohomology_of_BGamma_and_bub(Gamma)},
Theorem~\ref{the:K-homology_of_BGamma_and_bub(Gamma)},
Theorem~\ref{the:KO-cohomology_of_BGamma_and_bub(Gamma)} and
Theorem~\ref{the:KO-homology_of_BGamma_and_bub(Gamma)} except for the exact
structure of the $p$-torsion in $K^{2m+1}(\bub{\Gamma})$,
$KO^{2m+1}(\bub{\Gamma})$, $K_{2m}(\bub{\Gamma})$, and $KO_{2m}(\bub{\Gamma})$.

In the third step we will compute the equivariant complex and real topological
$K$-theory of $\eub{\Gamma}$, and hence the $K$-theory of the complex and real
$C^*$-algebras of $\Gamma$.  A complete answer is given in
Theorem~\ref{the:Topological_K-theory_of_the_complex_group_Cast-algebra} and
Theorem~\ref{the:Topological_K-theory_of_the_real_group_Cast-algebra}.  It is
rather surprising that we can give a complete answer although we do not know the
full answer for $\bub{\Gamma}$.

Finally we use the Baum-Connes Conjecture to prove
Theorem~\ref{the:Topological_K-theory_of_the_complex_group_Cast-algebra} and
Theorem~\ref{the:Topological_K-theory_of_the_real_group_Cast-algebra} in
Sections~\ref{sec:Topological_K-theory_of_the_group_Cast-algebra}.

The proof of
Theorem~\ref{the:The_(unstable)_Gromov-Lawson-Rosenberg_Conjecture_holds_for_Gamma}
will be presented in
Section~\ref{sec:The_(unstable)_Gromov-Lawson-Rosenberg_Conjecture_holds_for_Gamma}.

Although we are interested in the homological versions, it is important in each
step to deal first with the cohomological versions as well since we will make
use of the multiplicative structure and the Atiyah-Segal Completion Theorem.

This paper was financially supported by the Hausdorff Institute for Mathematics,
the Max-Planck-Institut f\"ur Mathematik, the Sonderforschungsbereich 478 \---
Geometrische Strukturen in der Mathematik \---, the NSF-grant of the first
author, and the Max-Planck-Forschungspreis and the Leibniz-Preis of the second
author. We thank the referee for his detailed report.

The paper is organized as follows:
\tableofcontents


\typeout{------------   Section 1: Group cohomology ------------}

\section{Group cohomology}
\label{sec:Group_cohomology}

In this section we compute the cohomology of $B\Gamma$ and $\eub{\Gamma}$ for
the group $\Gamma$ defined in~\eqref{Gamma_is_Zn_rtimes_Z/p}.  It fits into a
split exact sequence
\begin{eqnarray}
  &
  1 \to \IZ^n \xrightarrow{\iota}  \Gamma \xrightarrow{\pi} \IZ/p \to 1 
  &
  \label{Gamma_as_extension}
\end{eqnarray} 
We write the group operation in $\IZ/p$ and $\Gamma$ multiplicatively and in
$\IZ^n$ additively.  We fix a generator $t \in \Z/p$ and denote the value of
$\rho(t)$ by $\rho \colon \IZ^n \to \IZ^n$.  When wish to emphasize that $\IZ^n$
is a $\IZ[\IZ/p]$-module, we denote it by $\IZ^n_\rho$.


\subsection{Statement of the computation of the cohomology}

\begin{notation}[$\eub{G}$ and $\bub{G}$]
  \label{not:eub(G)_andbub(G)}
  For a discrete group $G$ we let $\eub{G}$ denote the \emph{classifying space
    for proper $G$-actions}.  Let $\bub{G}$ be the quotient space
  $G\backslash\eub{G}$.
\end{notation}

Recall that a model for the classifying space for proper $G$-actions is a
$G$-$CW$-complex $\eub{G}$ such that $\eub{G}^H$ is contractible if $H \subset
G$ is finite and empty otherwise.  Two models are $G$-homotopy equivalent.
There is a $G$-map $EG \to \eub G$ which is unique up to $G$-homotopy.  Hence
there is a map $BG \to \bub G$, unique up to homotopy.  If $G$ is torsion-free,
then $EG = \eub{G}$ and $BG = \bub{G}$. For more information about $\eub{G}$ we
refer for instance to the survey article~\cite{Lueck(2005s)}.

We will write $H^m(G)$ and $H_m(G)$ instead of $H^m(BG)$ and $H_m(BG)$.

\begin{example}[$\eub{\Gamma}$ and $\bub{\Gamma}$]
  Since the group $\Gamma$ is crystallographic and hence acts properly on
  $\IR^n$ by smooth isometries, a model for $\eub{\Gamma}$ is given by $\IR^n$
  with this $\Gamma$-action. In particular $\bub{\Gamma}$ is a quotient of the
  $n$-torus $T^n$ by a $\IZ/p$-action.
\end{example}

The main result of this section is the computation of the group cohomology of
$B\Gamma$ and $\bub{\Gamma}$.  Most of the calculation for $H^*(B\Gamma)$ has
already been carried out by Adem~\cite{Adem(1987)} and later, with different
methods, by Adem-Ge-Pan-Petrosyan~\cite{Adem-Ge-Pan-Petrosyan(2008)}.  The
computation of $H^*(\bub\Gamma)$ has recently and independently obtained by
different methods by Adem-Duman-Gomez~\cite{Adem-Duman-Gomez(2010)}.  We include
a complete proof since the techniques will be needed later when we compute
topological $K$-theory.

Let
$$N = t^0 + t + \dots + t^{p-1} \in \IZ[\IZ/p]$$ 
be the \emph{norm element}. Denote by $I(\IZ/p)$ the \emph{augmentation ideal},
i.e., the kernel of the augmentation homomorphism $\IZ[\IZ/p] \to \IZ$. Let
$\zeta = e^{2\pi i / p} \in \IC$ be a primitive $p$-th root of unity. We have
isomorphisms of $\IZ[\IZ/p]$-modules
$$
\IZ[\IZ/p]/N \cong \IZ[\zeta] \cong I(\IZ/p).
$$

Define natural numbers for $m,j,k \in \IZ_{\ge 0}$.
\begin{eqnarray}
  r_m  & := & \rk_{\IZ}\left(\left(\Lambda^m(\IZ[\zeta]^k\right)^{\IZ/p}\right);
  \label{r_m}
  \\
  a_j &: = & \left|\{(\ell_1, \dots, \ell_k) \in \IZ^k \mid \ell_1+ \dots + \ell_k = j, 0
    \leq \ell_i \leq p-1\}\right|;
  \label{a_j}
  \\ 
  s_m & := & \sum_{j=0}^{m-1} a_j,
  \label{s_m}
\end{eqnarray}  
where here and in the sequel $\Lambda^m$ means the $m$-th exterior power of a
$\IZ$-module.  Notice that these numbers $r_m$, $a_j$ and $s_m$ depend on $k$
but we omit this from the notation since $k$ will be determined by the equation
$n = k(p-1)$ (see Lemma~\ref{lem:preliminaries_about_Gamma_and_Zn_rho}%
~\ref{lem:preliminaries_about_Gamma_and_Zn_rho:ideals}) and hence by $\Gamma$.
Note that $r_0 = 1, r_1 = 0, a_0 = 1, a_1=k, s_0 = 0, s_1 = 1$, and $s_2 = k+1$.
We will give more information about these numbers in
Subsection~\ref{subsec:On_the_numbers_r_m}.

\begin{theorem}[Cohomology of $B\Gamma$ and $\bub{\Gamma}$]\
  \label{the:Cohomology_of_BGamma_and_bub(Gamma)}
  \begin{enumerate}

  \item \label{the:Cohomology_of_BGamma_and_bub(Gamma):Hm(BGamma)} For $m \ge 0$
$$H^m(\Gamma) \cong 
\begin{cases}
  \IZ^{r_m} \oplus (\IZ/p)^{s_{m}} & m \;\text{even;}
  \\
  \IZ^{r_m} & m \;\text{odd.}
\end{cases}
$$

\item \label{the:Cohomology_of_BGamma_and_bub(Gamma):Gamma_to_Zn} For $m \ge 0$
  the restriction map
$$ H^m(\Gamma) \to H^m(\IZ^n_\rho)^{\IZ/p}$$
is split surjective.  The kernel is isomorphic to $(\IZ/p)^{s_m}$ if $m$ is even
and 0 if $m$ is odd.

\item \label{the:Cohomology_of_BGamma_and_bub(Gamma):restriction_to-p-subgr} The map induced
by the various inclusions
$$
 \varphi^{m} \colon {H}^{m}(\Gamma)  \to \bigoplus_{(P) \in \calp} H^{m}(P)
$$
is bijective for $m > n$.

\item
  \label{the:Cohomology_of_BGamma_and_bub(Gamma):bub(Gamma)}
  For $m \ge 0$
$$H^m(\bub{\Gamma}) \cong 
\begin{cases}
  \IZ^{r_m} & m \;\text{even;}
  \\
  \IZ^{r_m} \oplus (\IZ/p)^{p^k-s_m} & m \;\text{odd}, m \ge 3;
  \\
  0 & m = 1.
\end{cases}
$$

\end{enumerate}

\end{theorem}

\begin{remark}[Multiplicative structure]
\label{rem:multiplicative_structure}
A transfer argument shows that the kernel of the restriction map 
$H^m(\Gamma) \to H^m(\IZ^n)$ is $p$-torsion.
Theorem~\ref{the:Cohomology_of_BGamma_and_bub(Gamma)} together with
the exact sequence~\eqref{long_exact_cohomology_sequences_for_bub(Gamma)_BGamma)} implies
that the map induced by the restrictions to the various subgroups
$$H^m(\Gamma) \to H^m(\IZ^n) \oplus  \bigoplus_{(P) \in \calp} H^{m}(P)$$
is injective. The multiplicative structure of the target is obvious. This allows in principle
to detect the multiplicative structure on $H^*(\Gamma)$.
\end{remark}

\subsection{Proof of Theorem~\ref{the:Cohomology_of_BGamma_and_bub(Gamma)}}
\label{subsec:Proof_of_Theorem_ref(the:Cohomology_of_BGamma_and_bub(Gamma))}

The proof of Theorem~\ref{the:Cohomology_of_BGamma_and_bub(Gamma)} needs some
preparation.

\begin{lemma} \label{lem:preliminaries_about_Gamma_and_Zn_rho}
  \begin{enumerate}
  \item \label{lem:preliminaries_about_Gamma_and_Zn_rho:ideals} We have an
    isomorphism of $\Z[\Z/p]$-modules,
$$\IZ^n_\rho \cong I_1 \oplus \dots \oplus I_k,$$
where the $I_j$ are non-zero ideals of $\IZ[\zeta]$.

We have
\begin{eqnarray*}
  \IZ^n_\rho \otimes \IQ & \cong & \IQ(\zeta)^k;
  \\
  n  & = & k(p-1).
\end{eqnarray*}

\item \label{lem:preliminaries_about_Gamma_and_Zn_rho:finite_subgroups} Each
  non-trivial finite subgroup $P$ of $\Gamma$ is isomorphic to $\IZ/p$ and its
  Weyl group $W_{\Gamma}\!P := N_{\Gamma}P/P$ is trivial.

\item \label{lem:preliminaries_about_Gamma_and_Zn_rho:list_of_finite_subgroups}
  There are isomorphisms
$$
H^1(\IZ/p;\IZ^n_{\rho}) \xrightarrow{\cong} \cok(\rho -\id \colon \IZ^n \to
\IZ^n) \cong (\IZ/p)^k;
$$
and a bijection
$$
\cok\bigl(\rho -\id \colon \IZ^n \to \IZ^n\bigr) \xrightarrow{\cong} \calp :=
\{(P)\mid P \subset \Gamma, 1 < |P| < \infty \}.
$$
If we fix an element $s \in \Gamma$ of order $p$, the bijection sends the
element $\overline{u} \in \IZ^n_\rho/(1-\rho)\IZ^n_\rho$ to the subgroup of order
$p$ generated by $us$.

\item \label{lem:preliminaries_about_Gamma_and_Zn_rho:order_of_calp} We have
  $|\calp| = p^k$.

\item \label{lem:preliminaries_about_Gamma_and_Zn_rho:fixed_set} There is a
  bijection from the $\IZ/p$-fixed set of the $\IZ/p$-space $T^n_{\rho} :=
  \R^n_\rho/\Z^n_\rho$ with $H^1(\Z/p; \Z^n_\rho)$.  In particular
  $(T^n_\rho)^{\IZ/p}$ consists of $p^k$ points.

\item \label{lem:preliminaries_about_Gamma_and_Zn_rho:commutator}
  $[\Gamma,\Gamma] = \im\left(\rho - \id \colon \IZ^n \to \IZ^n\right)$.

\item \label{lem:preliminaries_about_Gamma_and_Zn_rho:abelianization}
  $\Gamma/[\Gamma,\Gamma] \cong \cok(\rho -\id \colon \IZ^n \to \IZ^n) \oplus
  \IZ/p = (\IZ/p)^{k+1}$.

\end{enumerate}

\end{lemma}
\begin{proof}~\ref{lem:preliminaries_about_Gamma_and_Zn_rho:ideals} Let $u \in \IZ^n_\rho$.
  Then $N \cdot u $ is fixed by the action of $t \in \IZ/p$ and hence is zero by
  assumption.  Thus $\IZ^n_\rho$ is a finitely generated module over the
  Dedekind domain $\IZ[\IZ/p]/N = \IZ[\zeta]$. Any finitely generated
  torsion-free module over a Dedekind domain is isomorphic to a direct sum of
  non-zero ideals (see~\cite[page~11]{Milnor(1971)}).  Since $I_j \otimes \IQ
  \cong \IQ (\zeta)$, we see $\rk_{\IZ}(I_j) = p-1$.  
\\[2mm]~\ref{lem:preliminaries_about_Gamma_and_Zn_rho:finite_subgroups} This is
  obvious.  
\\[2mm]~\ref{lem:preliminaries_about_Gamma_and_Zn_rho:list_of_finite_subgroups} 
Since  the norm element $N$ acts trivially on $\IZ^n_{\rho}$, we get
$$
\cok\bigl(\rho - \id \colon \IZ^n \to \IZ^n\bigr) = H^1(\IZ/p;\IZ^n_{\rho}).
$$
We will show
$$H^1(\IZ/p;\IZ^n_\rho) \cong \widehat{H}^0(\IZ/p;H^1(\IZ^n_\rho)) \cong (\IZ/p)^{k}$$
in Lemma~\ref{lem:Hochschild-Serre_ss_forKast(BGamma)}~%
\ref{lem:Hochschild-Serre_ss_forKast(BGamma):Tate_cohomology}. One easily checks
that the map $\cok\bigl(\rho -\id \colon \IZ^n \to \IZ^n\bigr) \to \calp$ is
bijective.  \\[2mm]%
\ref{lem:preliminaries_about_Gamma_and_Zn_rho:order_of_calp} This follows from
assertion~\ref{lem:preliminaries_about_Gamma_and_Zn_rho:list_of_finite_subgroups}.
\\[2mm]%
\ref{lem:preliminaries_about_Gamma_and_Zn_rho:fixed_set} Consider the short
exact sequence of $\Z[\Z/p]$-modules
$$
0 \to \Z^n_\rho \to \R^n_\rho \to T^n_\rho \to 0
$$
Then the long exact cohomology sequence
$$
(\Z^n_\rho)^{\Z/p} \to (\R^n_\rho)^{\Z/p} \to (T^n_\rho)^{\Z/p} \to H^1(\Z/p;
\Z^n_\rho) \to H^1(\Z/p; \R^n_\rho)
$$
is isomorphic to
$$
0 \to 0 \to (T^n_\rho)^{\Z/p} \to (\Z/p)^k \to 0~.
$$
\\[2mm]%
\ref{lem:preliminaries_about_Gamma_and_Zn_rho:commutator} For $(i,p) = 1$ we
have $(\zeta^i - 1 )/ (\zeta - 1) \in \IZ[\zeta]^\times$ and hence we get $\ker
(\rho - \id) = \ker (\rho^i - \id) = 0$ and $\im (\rho - \id) = \im (\rho^i -
\id)$. This implies
$$[\Gamma,\Gamma] = \im\left(\rho - \id  \colon \IZ^n \to
  \IZ^n\right).$$ \\[2mm]%
\ref{lem:preliminaries_about_Gamma_and_Zn_rho:abelianization} The isomorphism
$$\cok\bigl(\rho -\id \colon \IZ^n \to \IZ^n\bigr) \oplus \IZ/p 
\xrightarrow{\cong} \Gamma/[\Gamma,\Gamma]
$$ sends $(\overline{u},\overline{i}) \mapsto \overline{us^i}$. 
\end{proof}

Next will analyze the \emph{Hochschild-Serre Spectral sequence}
(see~\cite[page~171]{Brown(1982)})
$$E^{i,j}_2 = H^i(\IZ/p; H^j(\IZ^n_\rho)) \Rightarrow H^{i+j}(\Gamma)$$
of the extension~\eqref{Gamma_as_extension}. We say that a spectral sequence
\emph{collapses} if all differentials $d^{i,j}_r$ are trivial for $r \ge 2$ and
all extension problems are trivial.  The basic properties of the Tate cohomology
$\widehat{H}^i(G;M)$ of a finite group $G$ with coefficients in a $\Z
[G]$-module $M$ are reviewed in Appendix~\ref{tate_and_transfer}.

\begin{lemma}
  \label{lem:Hochschild-Serre_ss_forKast(BGamma)}\
  \begin{enumerate}
  \item \label{lem:Hochschild-Serre_ss_forKast(BGamma):Tate_cohomology}
$$\widehat{H}^i(\IZ/p;H^j(\IZ^n_{\rho})) \cong
\bigoplus_{\substack{\ell_1 + \dots +\ell_k = j\\0 \leq \ell_q\leq p-1}}
\widehat{H}^{i+j}(\IZ/p;\Z) =
\begin{cases} (\IZ/p)^{a_j} & i+j \;\text{even}; \\ 0 & i +j \;\text{odd}.
\end{cases}
$$

\item \label{lem:Hochschild-Serre_ss_forKast(BGamma):collapse} The
  Hochschild-Serre spectral sequence associated to the
  extension~\eqref{Gamma_as_extension} collapses.

\end{enumerate}
\end{lemma}
\begin{proof}~\ref{lem:Hochschild-Serre_ss_forKast(BGamma):Tate_cohomology} 
There is a sequence of $\IZ[\IZ/p]$-isomorphisms
$$H^1(\IZ^n_{\rho}) \cong \hom_{\IZ}(H_1(\IZ^n_{\rho}),\IZ) \cong 
\hom_{\IZ}(\IZ^n_{\rho},\IZ) \cong \IZ^n_{\rho^*},$$ where $\rho(t)^* \colon
\IZ^n \to \IZ^n$ for $t \in \IZ/p$ is given by the transpose of the matrix
describing $\rho(t) \colon \IZ^n \to \IZ^n$.  The natural map given by the
product in cohomology
$$\Lambda^jH^1(\IZ^n)  \xrightarrow{\cong} H^j(\IZ^n)$$
is bijective and hence is a $\IZ[\IZ/p]$-isomorphism by naturality.  Thus we
obtain a $\IZ[\IZ/p]$-isomorphism
\begin{eqnarray*}
  H^j(\IZ^n_{\rho}) &\cong & \Lambda^j\IZ^n_{\rho^{\ast}}.
\end{eqnarray*}  
Given a non-zero ideal $I \subset \IZ[\zeta]$, There exists an isomorphism of $\IZ_{(p)}[\zeta]$-modules
$$I \otimes \IZ_{(p)} \xrightarrow{\cong} 
\IZ[\zeta] \otimes _{\IZ}\IZ_{(p)} = \IZ_{(p)}[\zeta].$$ This is true since
$\IZ_{(p)}[\zeta]$ is a discrete valuation ring, hence all ideals are principal.
Since $\Z^n_{\rho^*}$ is isomorphic to a direct sum of ideals of $\Z[\zeta]$, we
obtain for an appropriate natural number $k$ isomorphisms of $\IZ[\zeta] \otimes
_{\IZ}\IZ_{(p)} = \IZ_{(p)}[\zeta]$-modules
\begin{eqnarray*}
  H^j(\IZ_{\rho}^n) \otimes _{\IZ}\IZ_{(p)} \cong
  \Lambda^j\IZ^n_{\rho^*} \otimes _{\IZ}\IZ_{(p)} \cong 
  \Lambda^j\IZ[\zeta]^k \otimes _{\IZ}\IZ_{(p)}. 
\end{eqnarray*}
For every $\IZ[\IZ/p]$-module $M$ the obvious map
$$\widehat{H}^i(\IZ/p;M) \to \widehat{H}^i(\IZ/p;M \otimes _{\IZ}\IZ_{(p)})$$
is bijective. Hence we obtain an isomorphism
$$\widehat{H}^i(\IZ/p;H^j(\IZ^n_{\rho})) \cong
\widehat{H}^i(\IZ/p;\Lambda^j\IZ[\zeta]^k).$$ Since
$$
\Lambda^*\bigl(\bigoplus_k\IZ[\zeta])\bigr) = \bigotimes_k \Lambda^*
(\IZ[\zeta])
$$ 
and $\Lambda^l(\IZ[\zeta]) = 0$ for $l \ge p$, we get
$$
\Lambda^j(\IZ[\zeta]^k)) = \bigoplus_{\substack{\ell_1 + \dots +\ell_k = j\\0
    \leq \ell_q\leq p-1}} \Lambda^{\ell_1} \IZ[\zeta] \otimes \dots \otimes
\Lambda^{\ell_k} \IZ[\zeta].
$$
Therefore we obtain an isomorphism
$$\widehat{H}^i(\IZ/p;H^j(\IZ^n_{\rho})) \cong
\bigoplus_{\substack{\ell_1 + \dots +\ell_k = j\\0 \leq \ell_q\leq p-1}}
\widehat{H}^i\left(\IZ/p;\Lambda^{\ell_1} \IZ[\zeta] \otimes \dots \otimes
  \Lambda^{\ell_k} \IZ[\zeta]\right)).
$$
Hence it suffices to show for $l_1, \ldots ,l_k$ in $\{0,1,\ldots, p-1\}$
$$\widehat{H}^i\left(\IZ/p;\Lambda^{\ell_1} \IZ[\zeta] \otimes \dots \otimes
  \Lambda^{\ell_k} \IZ[\zeta]\right) \cong \widehat{H}^{i+\sum_{a = 1}^k
  l_a}(\IZ/p;\Z).$$ This will be done by induction over $j = \sum_{a = 1}^k
l_a$.  The induction beginning $j = 0$ is trivial, the induction step from $j-1$
to $j \ge 1$ done as follows. We can assume without loss of generality that
$1\le l_1 \le p-1$ otherwise permute the factors.  There is an exact sequence of
$\IZ[\IZ/p]$-modules
\begin{eqnarray*}
  0 \to \IZ \to  \IZ[\IZ/p] \to \IZ[\zeta] \to 0.
\end{eqnarray*}
where $1 \in \Z$ maps to the norm element $N \in \Z[\Z/p]$.  Since this exact
sequence splits as an exact sequence of $\Z$-modules, it induces an exact
sequence of $\IZ[\IZ/p]$-modules
\begin{eqnarray}
  1 \to \Lambda^{l_1-1}\IZ[\zeta] \to \Lambda^{l_1}\IZ[\IZ/p]
  \to \Lambda^{l_1}\IZ[\zeta]  \to 1,
  \label{lem:Hochschild-Serre_ss_forKast(BGamma):Lambda_sequence}
\end{eqnarray}
where the second map is induced by the epimorphism $\IZ[\IZ/p] \to \IZ[\zeta]$
and the first sends $u_1 \wedge u_2 \wedge \ldots \wedge u_{l_1-1}$ to $u_1'
\wedge u_2' \wedge \ldots \wedge u_{l_1-1}' \wedge N$, where $u_b' \in
\IZ[\IZ/p]$ is any element whose image under the projection $\IZ[\IZ/p] \to
\IZ[\zeta]$ is $u_b$. This is independent of the choice of the $u_b'$-s since
two such choices differ by a multiple of the norm element $N \in \IZ[\IZ/p]$.

We next show that the middle term
of~\eqref{lem:Hochschild-Serre_ss_forKast(BGamma):Lambda_sequence} is a free
$\Z[\Z/p]$-module when $1\le l_1 \le p-1$.  Since $\IZ/p = \{t^0,t^1, \ldots,
t^{p-1}\}$ is a $\IZ$-basis for $\IZ[\IZ/p]$, we obtain a $\IZ$-basis for
$\Lambda^{l_1}\IZ[\IZ/p]$ by
$$\{t^I \mid I \subset\IZ/p, |I| = l_1\},$$
where $t^I= t^{i_1} \wedge t^{i_2} \wedge \ldots \wedge t^{i_{l_1}}$ for $I =
\{i_1,i_2,\ldots , i_{l_1}\}$ with $1 \le i_1 < i_2 < \ldots < i_{l_1} \le p-1$.
An element $s \in \IZ/p$ acts on $\Lambda^{l_1}\IZ[\IZ/p]$ by sending the basis
element $t^I$ to $\pm t^{s+I}$.  The $\IZ/p$ action on $\{I \subset \IZ/p, |I| =
l_1\}$ which sends $I$ to $s+I$ for $s \in \IZ/p$, is free.  Indeed, for $s \in
\Z/p -\{0\}$, the permutation of the $p$-element set $\Z/p$ given by $a \mapsto
s + a$ cannot have any proper invariant sets since the permutation has order $p$
and $p$ is prime.  This implies that the $\IZ[\IZ/p]$-module
$\Lambda^{l_1}\IZ[\IZ/p]$ is free.  

We obtain from the exact
sequence~\eqref{lem:Hochschild-Serre_ss_forKast(BGamma):Lambda_sequence} an
exact sequence of $\IZ[\IZ/p]$-modules with a free $\IZ[\IZ/p]$-module in the
middle
\begin{multline*}
  1 \to \Lambda^{l_1-1}\IZ[\zeta] \otimes \Lambda^{\ell_2} \IZ[\zeta] \otimes
  \dots \otimes \Lambda^{\ell_k} \IZ[\zeta] \to \Lambda^{l_1}\IZ[\IZ/p] \otimes
  \Lambda^{\ell_2} \IZ[\zeta] \otimes \dots \otimes \Lambda^{\ell_k} \IZ[\zeta]
  \\
  \to \Lambda^{l_1}\IZ[\zeta] \otimes \Lambda^{\ell_2} \IZ[\zeta] \otimes \dots
  \otimes \Lambda^{\ell_k} \IZ[\zeta] \to 1.
\end{multline*}
Hence we obtain for $i \in \IZ$ an isomorphism
$$ \widehat{H}^{i}\left(\IZ/p;\Lambda^{\ell_1} \IZ[\zeta] \otimes \dots \otimes
  \Lambda^{\ell_k} \IZ[\zeta]\right)\cong
\widehat{H}^{i+1}\left(\IZ/p;\Lambda^{\ell_1-1} \IZ[\zeta] \otimes \dots \otimes
  \Lambda^{\ell_k} \IZ[\zeta]\right).$$ Now apply the induction hypothesis. This
finishes the proof of
assertion~\ref{lem:Hochschild-Serre_ss_forKast(BGamma):Tate_cohomology}.
\\[1mm]%
\ref{lem:Hochschild-Serre_ss_forKast(BGamma):collapse} Next we want to show that
the differentials $d^{i,j}_r$ are zero for all $r \ge 2$ and $i,j$. By the
checkerboard pattern of the $E_2$-term it suffices to show for $r \ge 2$ and
that the differentials $d_r^{0,j}$ are trivial for $r \ge 2$ and all odd $j \ge
1$. This is equivalent to show that for every odd $j \ge 1$ the edge
homomorphism (see Proposition~\ref{prop:edge_homomorphisms}) 
$$\iota^j \colon H^j(\Gamma) \to H^j(\IZ^n_{\rho})^{\IZ/p} = E^{0,j}_2$$ 
is surjective.  But
$\widehat{H}^0(\IZ/p,H^j(\IZ^n_{\rho}))=0$ by
assertion~\ref{lem:Hochschild-Serre_ss_forKast(BGamma):Tate_cohomology}, so the
norm map $N = \iota^j \circ \trf^j\colon H^j(\Z^n_\rho)_{\Z/p} \to
H^j(\Z^n_\rho)^{\Z/p}$ is surjective (see Theorem\ref{norm_is_composite}), so
$\iota^j$ is surjective.

It remains to show that all extensions are trivial.  Since the composite
$$H^{i+j}(\Gamma)   \xrightarrow{\iota^{i+j}} H^{i+j}(\IZ^n_\rho)\xrightarrow{\trf^{i+j}}
H^{i+j}(\Gamma)$$ is multiplication with $p$, the torsion in $H^{i+j}(\Gamma)$
has exponent $p$.  Since $p \cdot E^{i,j}_{\infty} = p \cdot E^{i,j}_2 = 0$ for
$i > 0$, all extensions are trivial and
$$H^m\Gamma \cong \bigoplus_{i+j = m} E^{i,j}_{\infty} =  \bigoplus_{i+j = m} E^{i,j}_2.$$
\end{proof}

\begin{proof}[Proof of 
assertions~\ref{the:Cohomology_of_BGamma_and_bub(Gamma):Hm(BGamma)} 
and~\ref{the:Cohomology_of_BGamma_and_bub(Gamma):Gamma_to_Zn}
of Theorem~\ref{the:Cohomology_of_BGamma_and_bub(Gamma)}]
These are direct consequences of \linebreak
Lemma~\ref{lem:Hochschild-Serre_ss_forKast(BGamma)}.
\end{proof}

\begin{proof}[Proof of assertion~\ref{the:Cohomology_of_BGamma_and_bub(Gamma):restriction_to-p-subgr}
of Theorem~\ref{the:Cohomology_of_BGamma_and_bub(Gamma)}]

We obtain from~\cite[Corollary~2.11]{Lueck-Weiermann(2007)} together with
Lemma~\ref{lem:preliminaries_about_Gamma_and_Zn_rho}~%
\ref{lem:preliminaries_about_Gamma_and_Zn_rho:finite_subgroups} a cellular
$\Gamma$-pushout
\begin{eqnarray}
  &
  \xycomsquareminus{\coprod_{(P) \in \calp} \Gamma \times _P EP}{i_0}{E\Gamma}
  {\coprod_{(P) \in \calp} \pr_P }{f}
  {\coprod_{(P) \in \calp} \Gamma/P}{i_1}{\eub{\Gamma}}
  &
  \label{G-pushoutfor_EGamma_to_eunderbar_Gamma}
\end{eqnarray}
where $i_0$ and $i_1$ are inclusions of $\Gamma$-$CW$-complexes, $\pr_P$ is the
obvious $\Gamma$-equivariant projection and $\calp$ is the set of conjugacy
classes of subgroups of $\Gamma$ of order $p$.  Taking the quotient with respect
to the $\Gamma$-action we obtain
from~\eqref{G-pushoutfor_EGamma_to_eunderbar_Gamma} the cellular pushout
\begin{eqnarray}
  &
  \xycomsquareminus{\coprod_{(P) \in \calp} BP}{j_0}{B\Gamma}
  {\coprod_{(P) \in \calp} \overline{\pr}_P }{\overline{f}}
  {\coprod_{(P) \in \calp} \pt}{j_1}{\bub{\Gamma}}
  &\label{pushout_for_BGamma_to_bunderbar(Gamma)}
\end{eqnarray}
where $j_0$ and $j_1$ are inclusions of $CW$-complexes, $\overline{\pr}_P$ is the
obvious projection.  It yields the following long exact sequence for $m \ge 0$
\begin{multline}
  0 \to {H}^{2m}(\bub{\Gamma}) \xrightarrow{\overline{f}^*} {H}^{2m}(\Gamma)
  \xrightarrow{\varphi^{2m}} \bigoplus_{(P) \in \calp} \widetilde{H}^{2m}(P)
  \\
  \xrightarrow{\delta^{2m}} {H}^{2m+1}(\bub{\Gamma})
  \xrightarrow{\overline{f}^*} {H}^{2m+1}(\Gamma) \to 0
  \label{long_exact_cohomology_sequences_for_bub(Gamma)_BGamma)}
\end{multline}
where $\varphi^*$ is the map induced by the various inclusions $P \subset
\Gamma$ for $(P) \in \calp$.

Now assertion~\ref{the:Cohomology_of_BGamma_and_bub(Gamma):restriction_to-p-subgr}  follows
from~\eqref{long_exact_cohomology_sequences_for_bub(Gamma)_BGamma)} since there is a $n$-dimensional model
for  $\bub{\Gamma}$.
\end{proof}

We still need to prove assertion~\ref{the:Cohomology_of_BGamma_and_bub(Gamma):bub(Gamma)} of
Theorem~\ref{the:Cohomology_of_BGamma_and_bub(Gamma)}.

In order to compute $H^*(\bub \Gamma)$, we need to compute the kernel and image
of $\varphi^{2m}$.

\begin{lemma}
  \label{lem:computing_kernel_and_image_of_phi2m}
  Let $m \geq 1$.
  \begin{enumerate}
  \item
    \label{lem:computing_kernel_and_image_of_phi2m:kernel}
    Let $K^{2m}$ be the kernel of $ \varphi^{2m} $.  There is a short exact
    sequence
$$
0 \to K^{2m} \to H^{2m}(\Z^n_\rho)^{\Z/p} \to \widehat H^0(\Z/p;
H^{2m}(\Z^n_\rho)) \to 0
$$
where the first non-trivial map is the restriction of $\iota^* \colon H^{2m}(\Gamma)
\to H^{2m}(\Z^n_\rho)^{\Z/p}$ to $K^{2m}$ and the second non-trivial map is given by the
quotient map appearing in the definition of Tate cohomology.  It follows that
$K^{2m} \cong \Z^{r_m}$.
\item
  \label{lem:computing_kernel_and_image_of_phi2m:image}
  The image of $\varphi^{2m}$ is isomorphic to
$$\ker\left(H^{2m}(\Gamma) \to H^{2m}(\IZ^n_{\rho})^{\IZ/p}\right) \oplus
\widehat{H}^0(\IZ/p;H^{2m}(\IZ^n_{\rho})) \cong (\IZ/p)^{s_{2m+1}}.$$
\end{enumerate}
\end{lemma}

\begin{proof}~\ref{lem:computing_kernel_and_image_of_phi2m:kernel} Let $\beta
  \in H^2(\IZ/p) \cong \IZ/p$ be a generator.  Let $L^{2m}$ be the kernel of
$$-\cup \pi^{*}(\beta)^n \colon H^{2m}(\Gamma) \to H^{2m+2n}(\Gamma).$$
We first claim that $K^{2m} = L^{2m}$.  Indeed, the following diagram commutes
$$\xycomsquareminus
{H^{2m}(\Gamma)} {\varphi^{2m}} {\bigoplus_{(P) \in \calp} H^{2m}(P)} {- \cup
  \pi^{*}(\beta)^n}{- \cup \beta^n} {H^{2m+2n}(\Gamma)} {\varphi^{2m+2n}}
{\bigoplus_{(P) \in \calp} H^{2m+2n}(P)}
$$
Since $\dim(\bub{\Gamma}) \le n$, we have $H^{i+2n}(\bub{\Gamma}) = 0$ for $i
\ge 1$.  Hence the lower horizontal arrow is bijective
by~\eqref{long_exact_cohomology_sequences_for_bub(Gamma)_BGamma)}.  The right
vertical arrow is bijective. Thus $K^{2m} = L^{2m}$.

Recall that we have an descending filtration
$$H^{2m}(\Gamma) = F^{0,2m} \supset F^{1,2m-1} \supset  \cdots \supset F^{2m,0} \supset
F^{2m+1,-1} =0$$ such that $F^{r,2m-1}/F^{r+1,2m-r-1} \cong
E^{r,2m-r}_{\infty}$.  Recall that $E^{2,0}_2 = H^2(\IZ/p;H^0(\IZ^n_{\rho})) =
H^2(\IZ/p)$ so that we can think of $\beta$ as an element in $E^{2,0}_2$.
Recall that $E^{i,j}_2 = E^{i,j}_{\infty}$ by
Lemma~\ref{lem:Hochschild-Serre_ss_forKast(BGamma)}~\ref{lem:Hochschild-Serre_ss_forKast(BGamma):collapse}.
From the multiplicative structure of the spectral sequence we see that the image
of the map
$$- \cup \pi^*(\beta)^{n} \colon H^{2m}(\Gamma)  \to H^{2m+2n}(\Gamma)$$
lies in $F^{2n,2m}$ and the following diagram commutes
\begin{equation}
  \label{filtration_commute}
  \xymatrix@!C=8em{0 \ar[d] 
    &
    0 \ar[d] 
    \\
    F^{1,2m-1} \ar[r]_{-\cup \pi^{*}(\beta)^{n}}^{\cong} \ar[d]
    & 
    F^{2n+1,2m-1} \ar[d]
    \\        
    H^{2m}(\Gamma) \ar[r]_{-\cup \pi^{*}(\beta)^{n}} \ar[d]
    & 
    F^{2n,2m} \ar[d]
    \\
    E_{\infty}^{0,2m} \ar[r]_{-\cup \beta^{n}} \ar[d]
    & 
    E_{\infty}^{2n,2m} \ar[d]
    \\
    0 
    &
    0
  }
\end{equation}
where the columns are exact. The upper horizontal arrow is bijective.  Namely,
one shows by induction over $r = -1, 0, 1, \ldots, 2m-1$ that the map
$$-\cup \pi^{*}(\beta)^{n} \colon F^{2m-r,r} \to F^{2m-r+2n,r}$$
is bijective. The induction beginning $r = -1$ is trivial since then both the
source and the target are trivial, and the induction step from $r-1$ to $r$
follows from the five lemma and the fact that the map
$$- \cup \beta^n \colon E^{2m-r,r}_{\infty} = H^{2m-r}(\IZ/p;H^r(\IZ^n_{\rho})) \to 
E^{2m-r+2n,r}_{\infty} = H^{2m-r+2n}(\IZ/p;H^r(\IZ^n_{\rho}))$$ is bijective.

The bottom horizontal map in diagram \eqref{filtration_commute} can be
identified with the composition of the canonical quotient map
$$H^0(\IZ/p;H^{2m}(\IZ^n_{\rho})) \to \widehat{H}^0(\IZ/p;H^{2m}(\IZ^n_{\rho})).$$
with the isomorphism
$$- \cup \beta^n \colon \widehat{H}^0(\IZ/p;H^{2m}(\IZ^n_{\rho})) \xrightarrow{\cong}
\widehat{H}^{2n}(\IZ/p;H^{2m}(\IZ^n_{\rho})).$$

So what do we know about diagram \eqref{filtration_commute}?  The top horizontal
map is an isomorphism, the kernel of middle horizontal map is $L^{2m}$, and the
bottom horizontal map is onto.  We conclude from the snake lemma that the middle
map is an epimorphism and that we have a short exact sequence
$$
0 \to L^{2m} \to E_{\infty}^{0,2m} \to E_{\infty}^{2n,2m} \to 0.
$$
The first non-trivial map is the composite of the inclusion $K^{2m} =
L^{2m}\subset H^{2m}(\Gamma)$ with the epimorphism
$$
H^{2m}(\Gamma) \to E^{0,2m}_{\infty} = H^{2m}(\Z^n_\rho)^{\Z/p}
$$
induced by the inclusion $\iota \colon \IZ^n \to \Gamma$.  We have already
identified the second non-trivial map (up to isomorphism) with the quotient map
as desired.  Hence the sequence in
assertion~\ref{lem:computing_kernel_and_image_of_phi2m:kernel} is exact.  Since
the middle term is isomorphic to $\Z^{r_m}$ and the right term is finite,
$K^{2m}$ is also isomorphic to $\Z^{r_m}$.  \\[1mm]%
\ref{lem:computing_kernel_and_image_of_phi2m:image} The exact sequence
$$
0 \to \ker\left(H^{2m}(\Gamma) \to H^{2m}(\IZ^n_{\rho})^{\IZ/p}\right) \to
H^{2m}(\Gamma) \xrightarrow{\iota^{2m}} H^{2m}(\IZ^n_{\rho})^{\IZ/p} \to 0
$$
has the property that $\iota^{2m}$ restricted to $K^{2m}$ is injective.  Thus we
can quotient by $K^{2m}$ and $\iota^{2m}(K^{2m})$ in the middle and right hand
term respectively and maintain exactness.  Hence we have the exact sequence
\begin{multline}
  \label{lem:computing_kernel_and_image_of_phi2m:eq0}
  0 \to \ker\left(H^{2m}(\Gamma) \to H^{2m}(\IZ^n_{\rho})^{\IZ/p}\right) \to
  H^{2m}(\Gamma)/K^{2m}
  \\
  \to \widehat{H}^0(\IZ/p;H^{2m}(\IZ^n_{\rho})) \to 0.
\end{multline}
where we used assertion~\ref{lem:computing_kernel_and_image_of_phi2m:kernel}  to
compute the right hand term.  We conclude from
Lemma~\ref{lem:Hochschild-Serre_ss_forKast(BGamma)}
\begin{eqnarray}
  \widehat{H}^0(\IZ/p;H^{2m}(\IZ^n_{\rho})) & \cong & (\IZ/p)^{a_{2m}};
  \label{lem:computing_kernel_of_phi2m:eq1}
  \\
  \ker\left(H^{2m}(\Gamma) \to H^{2m}(\IZ^n_{\rho})^{\IZ/p}\right) 
  & \cong & \bigoplus_{i=1}^{2m} E^{i,2m-i} \cong \bigoplus_{j=0}^{2m-1}
  (\IZ/p)^{a_j}
  \label{lem:computing_kernel_of_phi2m:eq2}
\end{eqnarray}
Since $H^{2m}(\Gamma)/K^{2m}$ is isomorphic to a subgroup of $\bigoplus_{(P) \in
  \calp} \widetilde{H}^{2m}(P)$ by the long exact cohomology
sequence~\eqref{long_exact_cohomology_sequences_for_bub(Gamma)_BGamma)} it is
annihilated by multiplication with $p$. Hence the short exact
sequence~\eqref{lem:computing_kernel_and_image_of_phi2m:eq0} splits and we
conclude from~\eqref{lem:computing_kernel_of_phi2m:eq1}
and~\eqref{lem:computing_kernel_of_phi2m:eq2}
$$H^{2m}(\Gamma)/K^{2m} \cong \bigoplus_{j=0}^{2m} (\IZ/p)^{a_j} \cong
(\IZ/p)^{s_{2m+1}}.$$ This finishes the proof of
Lemma~\ref{lem:computing_kernel_and_image_of_phi2m}.

\end{proof}

We conclude from the exact sequence
\eqref{long_exact_cohomology_sequences_for_bub(Gamma)_BGamma)},
Theorem~\ref{the:Cohomology_of_BGamma_and_bub(Gamma)}%
~\ref{the:Cohomology_of_BGamma_and_bub(Gamma):Hm(BGamma)},
Lemma~\ref{lem:preliminaries_about_Gamma_and_Zn_rho}~%
\ref{lem:preliminaries_about_Gamma_and_Zn_rho:order_of_calp}, and
Lemma~\ref{lem:computing_kernel_and_image_of_phi2m}

\begin{corollary}\label{cor:identifying_long_exact_sequence_for_H(Gamma)}
  For $m \ge 1$ the long exact
  sequence~\eqref{long_exact_cohomology_sequences_for_bub(Gamma)_BGamma)} can be
  identified with
$$0 \to \IZ^{r_{2m}} \to \IZ^{r_{2m}} \oplus (\IZ/p)^{s_{2m}} \to  (\IZ/p)^{p^k} 
\to \IZ^{r_{2m+1}} \oplus (\IZ/p)^{p^k-s_{2m+1}} \to \IZ^{r_{2m+1}} \to 0,$$
\end{corollary}

\begin{proof}[Proof of assertion~\ref{the:Cohomology_of_BGamma_and_bub(Gamma):bub(Gamma)}
of Theorem~\ref{the:Cohomology_of_BGamma_and_bub(Gamma)}]
Obviously $H^0(\bub{\Gamma}) \cong \IZ$.  Since $(\IZ^n)^{\IZ/p} = 0$ by
assumption, we get $H^1(\Gamma) = 0$ from
assertion~\ref{the:Cohomology_of_BGamma_and_bub(Gamma):Gamma_to_Zn} of
Theorem~\ref{the:Cohomology_of_BGamma_and_bub(Gamma)}.  We conclude
$H^1(\bub{\Gamma}) \cong 0$ from the long exact
sequence~\eqref{long_exact_cohomology_sequences_for_bub(Gamma)_BGamma)}.  The
values of $H^m(\bub{\Gamma})$ for $m \ge 2$ have already been determined in
Corollary~\ref{cor:identifying_long_exact_sequence_for_H(Gamma)}.  Hence
assertion~\ref{the:Cohomology_of_BGamma_and_bub(Gamma):bub(Gamma)} of
Theorem~\ref{the:Cohomology_of_BGamma_and_bub(Gamma)} follows. This finishes the
proof of Theorem~\ref{the:Cohomology_of_BGamma_and_bub(Gamma)}.
\end{proof}


\subsection{On the numbers \texorpdfstring{$r_m$}{rm}}
\label{subsec:On_the_numbers_r_m}

In this subsection we collect some basic information about the numbers $r_m$,
$a_j$ and $s_m$ introduced in~\eqref{\texorpdfstring{r_m}{rm}},\eqref{a_j},
and~\eqref{s_m}.

Since $\IZ^n$ acts freely on $\eub{\Gamma}= \R^n$, we conclude from
Lemma~\ref{lem:preliminaries_about_Gamma_and_Zn_rho}~%
\ref{lem:preliminaries_about_Gamma_and_Zn_rho:ideals} and
Proposition~\ref{prop:quotient_iso}
\begin{eqnarray*}
  r_m  
  & = &  
  \rk_{\IQ}\left(\Lambda^m_{\IQ}\bigl(\IQ(\zeta)^k)^{\IZ/p}\bigr)\right)
  \\
  & = & 
  \rk_{\IQ}\left(H^m\bigl(B\IZ^n_{\rho};\IQ\bigr)^{\IZ/p}\right)
  \\
  & = & 
  \rk_{\IQ}\left(H^m\bigl(\bub{\Gamma};\IQ)\right)
  \\
  & = & 
  \rk_{\IQ} \left(H^m(\Gamma;\IQ)\right)
  \\
\end{eqnarray*}
Since Tate cohomology is rationally trivial, the norm map is a rational
isomorphism, hence also
\begin{eqnarray}
  r_m & = & 
  \rk_{\IQ}\left( \Lambda^m_{\IQ}(\IQ[\zeta]^k) \otimes_{\IQ[\IZ/p]} \IQ\right).
  \label{r_m_in_terms_of_orbifold}
\end{eqnarray}

\begin{lemma} \label{lem:LambdajZ(zeta)_in_R_Q(Z/p)}
  \begin{enumerate}

  \item \label{lem:LambdajZ(zeta)_in_R_Q(Z/p):special_values} We have $r_0 = 1$,
    $r_1 = 0$, $a_0 = 1$, $a_1 = k$, $s_0= 0$, $s_1 = 1$, and $s_2 = k+1$. We
    get $r_m = 0$ for $m \ge n+1$ and $s_m = p^k$ for $m \ge n$.

  \item \label{lem:LambdajZ(zeta)_in_R_Q(Z/p):r_odd_and_r_ev} If $p$ is odd, we
    get
    \begin{eqnarray*}
      \sum_{\substack{m \ge 0\\m \; \text{even}}} r_m
      & = &
      \frac{2^{(p-1)k} + p -1}{2p} + \frac{(p-1) \cdot p^{k-1}}{2};
      \\
      \sum_{\substack{m \ge 0\\m \; \text{odd }}} r_m 
      & = & 
      \frac{2^{(p-1)k} + p -1}{2p} -  \frac{(p-1) \cdot p^{k-1}}{2}.
    \end{eqnarray*}
    If $ p = 2$, we get
    \begin{eqnarray*}
      \sum_{\substack{m \ge 0\\m \; \text{even}}} r_m
      & = & 2^{n-1};
      \\
      \sum_{\substack{m \ge 0\\m \; \text{odd }}} r_m 
      & = & 
      0.
    \end{eqnarray*}

  \item \label{lem:LambdajZ(zeta)_in_R_Q(Z/p):k_is_1} Suppose that $k = 1$. Then
    \begin{eqnarray*}
      r_m 
      & = & 
      \frac{1}{p} \cdot \left(\binom{p-1}{m} + (-1)^m \cdot (p-1)\right)  \quad \text{for } 0 \le m \le (p-1);
      \\
      r_m & = & 0 \quad \text{for } m \ge p;
      \\
      a_m & = & 1 \quad \text{for } 0 \le m \le p-1;
      \\
      a_m & = & 0 \quad \text{for } p \le m;
      \\
      s_m & = & m \quad  \text{for } 0 \le m \le p-1;
      \\
      s_m & = & p \quad \text{for } m \ge p.
    \end{eqnarray*}

  \end{enumerate}

\end{lemma}

\begin{proof}
  In the proof below we write $\Lambda^l V$ instead of $\Lambda^l_\Q V$ for a
  $\Q$-vector space $V$.  
  \\[1mm]~\ref{lem:LambdajZ(zeta)_in_R_Q(Z/p):special_values} This follows directly from
  the definitions.  
  \\[1mm]~\ref{lem:LambdajZ(zeta)_in_R_Q(Z/p):r_odd_and_r_ev} Suppose $1 \le l \le p-1$.
  By rationalizing the exact
  sequence~\eqref{lem:Hochschild-Serre_ss_forKast(BGamma):Lambda_sequence} we
  have the short exact sequence of $\IQ[\IZ/p]$-modules
  \begin{eqnarray*}
    0 \to \Lambda^{l-1}\IQ[\zeta] \to \Lambda^{l}\IQ[\IZ/p]
    \to \Lambda^{l}\IQ[\zeta] \to 0.
  \end{eqnarray*}
  Since $\Lambda^{l}\IZ[\IZ/p]$ is finitely generated free as
  $\IZ[\IZ/p]$-module (see proof
  of~Lemma~\ref{lem:Hochschild-Serre_ss_forKast(BGamma)}~%
~\ref{lem:Hochschild-Serre_ss_forKast(BGamma):Tate_cohomology}), the following
  equation holds in the rational representation ring $R_{\IQ}(\IZ/p)$
  \begin{eqnarray}
    \bigr[\Lambda^{l}\IQ[\zeta]\bigr] + \bigl[\Lambda^{l-1}\IQ[\zeta]\bigr] 
    & = & 
    \frac{1}{p} \cdot \binom{p}{l} \cdot \bigl[\IQ[\IZ/p]\bigr].
    \label{relating_the_classes_of_consecutive_Lambdal}
  \end{eqnarray}
  One shows by induction over $l$ for $0 \le l \le p-1$
  \begin{eqnarray}
    \quad \bigl[\Lambda^{l}(\IQ[\zeta])\bigr]
    & = &
    (-1)^{l} \cdot [\IQ] + 
    \frac{1}{p} \left(\binom{p-1}{l} - (-1)^l\right) \cdot
    \bigl[\IQ[\IZ/p]\bigr].
    \label{LambdalQ(zeta)}
  \end{eqnarray}

  Since $\sum_{l=0}^{p-1} \binom{p-1}{l} = 2^{p-1}$, we get
  \begin{eqnarray}
    \sum_{l=0}^{p-1} \bigr[\Lambda^{l}\IQ[\zeta]\bigr]
    & = &
    \begin{cases}
      [\IQ] + \frac{2^{p-1} -1}{p} \cdot \bigl[\IQ[\IZ/p]\bigr] & \text{if}\; p\; \text{is odd};
      \\
      [\IQ[\IZ/2]] & \text{if}\;  p = 2.
    \end{cases}
    \label{sum_all_l_Lambdal}
  \end{eqnarray}
  Since
$$
\Lambda^*\bigl(\bigoplus_k\IQ[\zeta]\bigr) = \bigotimes_k \Lambda^* (\IQ[\zeta])
$$ 
and $\Lambda^l(\IQ[\zeta]) = 0$ for $l \ge p$, we get
\begin{eqnarray}
  \bigl[\Lambda^j(\IQ[\zeta]^k)\bigr] 
  & = &\sum_{\substack{\ell_1 + \dots +\ell_k = j\\0
      \leq \ell_i\leq p-1}} \; \prod_{i=1}^k \bigl[\Lambda^{\ell_i}(\IQ[\zeta])\bigr].
  \label{class_of_LambdajQ(zeta)k)_abstract}
\end{eqnarray}

We conclude from~\eqref{sum_all_l_Lambdal}
and~\eqref{class_of_LambdajQ(zeta)k)_abstract}
\begin{eqnarray*}
  \sum_{j \ge 0} \bigl[\Lambda^j(\IQ[\zeta]^k)\bigr] 
  & = & 
  \sum_{j \ge 0} \left(\sum_{\substack{\ell_1 + \dots +\ell_k = j\\0
        \leq \ell_i\leq p-1}} \; \prod_{i=1}^k \bigl[\Lambda^{\ell_i}(\IQ[\zeta])\bigr]\right)
  \\
  & = & 
  \sum_{\substack{l_1,l_2, \ldots , l_k\\0 \leq \ell_q\leq p-1}} 
  \; \prod_{i=1}^k \bigl[\Lambda^{\ell_i}(\IQ[\zeta])\bigr]
  \\
  & = & 
  \prod_{i=1}^k\; \sum_{0 \leq \ell_i \leq p-1} \;  \bigl[\Lambda^{\ell_i}(\IQ[\zeta])\bigr]
  \\
  & = & 
  \begin{cases}
    \left([\IQ] + \frac{2^{p-1} -1}{p} \cdot \bigl[\IQ[\IZ/p]\bigr]\right)^k &
    \text{if} \; p \; \text{is odd};
    \\
    [\IQ[\IZ/2]]^k & \text{if}\; p = 2.
  \end{cases}
\end{eqnarray*}
Since $[\IQ]$ is the multiplicative unit in $R_{\IQ}(\IZ/p)$, and
$\bigl[\IQ[\IZ/p]\bigr]^i = p^{i-1} \cdot \bigl[\IQ[\IZ/p]\bigr]$, we obtain the
following equality in $R_{\IQ}(\IZ/p)$ if $p$ is odd:
\begin{eqnarray}
  \sum_{j \ge 0} \bigl[\Lambda^j(\IQ[\zeta]^k)\bigr] 
  & = & 
  \sum_{i=0}^k \binom{k}{i} \cdot \frac{(2^{p-1} - 1)^i}{p^i} \cdot 
  \bigl[\IQ[\IZ/p]\bigr]^i \cdot [\IQ]^{k-i}
  \nonumber 
  \\
  & = & 
  [\IQ] + \frac{1}{p} \cdot \left(-1 + \sum_{i=0}^{k}  \binom{k}{i} (2^{p-1} - 1)^i\right) 
  \cdot \bigl[\IQ[\IZ/p]\bigr]
  \nonumber 
  \\
  & = & 
  [\IQ] + \frac{1}{p} \cdot \left(-1 + 2^{(p-1)k}\right) 
  \cdot \bigl[\IQ[\IZ/p]\bigr]
  \nonumber 
  \\
  & = & 
  [\IQ] + \frac{2^{(p-1)k}-1}{p} \cdot \bigl[\IQ[\IZ/p]\bigr].
  \label{sum_j_ge_0_Lambda[zeta]k_p_is_odd}
\end{eqnarray}
If $ p = 2$, we obtain
\begin{eqnarray*}
  \sum_{j \ge 0} \bigl[\Lambda^j(\IQ[\zeta]^k)\bigr] 
  & = & 
  2^{k-1} \cdot [\IQ[\IZ/2]].
  \label{sum_j_ge_0_Lambda[zeta]k_p_is_2}
\end{eqnarray*}

There is a homomorphism of abelian groups
$$\Phi\colon R_{\IQ}(\IZ/p) \to \IZ, \quad [V] \mapsto \rk_{\IQ}\bigl(V \otimes_{\IQ[\IZ/p]} \Q\bigr).$$
By~\eqref{r_m_in_terms_of_orbifold} it sends $\Q$, $\Q[\Z/p]$, 
and $\bigl[\Lambda^m(\IQ[\zeta]^k)\bigr]$ to 1, 1,
and $r_m$ respectively.  Hence we conclude
from~\eqref{sum_j_ge_0_Lambda[zeta]k_p_is_odd}
\begin{eqnarray}
  \sum_{m \ge 0} r_m & = & \frac{2^{(p-1)k} -1}{p}+1 \quad \text{for $p$ odd;}
  \label{r_all_explicit_p_odd}
\end{eqnarray}
\begin{eqnarray}
  \sum_{m \ge 0} r_m & = & 2^{k-1} \quad \text{for $p=2$.}
  \label{r_all_explicit_p_is_2}
\end{eqnarray}

If $X$ is a finite $\Z/p$-CW-complex with orbit space $\overline X$, then the
Riemann-Hurwitz formula states that
$$
\chi(\overline X) = \frac{1}{p} \chi(X) + \frac{p-1}{p}\chi(X^{\Z/p}).
$$
One derives this formula by verifying it for both fixed and freely permuted
cells.  Applying Proposition~\ref{prop:quotient_iso}, the Riemann-Hurwitz
formula, and Lemma~\ref
{lem:preliminaries_about_Gamma_and_Zn_rho}~\ref{lem:preliminaries_about_Gamma_and_Zn_rho:fixed_set}
to the $\Z/p$-action on the torus $T^n$, one sees
\begin{eqnarray}
  \sum_{m \ge 0} (-1)^m r_m = \chi((\Z/p)\backslash T^m) =0+ (p-1)p^{k-1}.
  \label{r_ev-r_odd}
\end{eqnarray}

We conclude from~\eqref{r_all_explicit_p_odd} and~\eqref{r_ev-r_odd} if $p$ is
odd
\begin{eqnarray}
  \sum_{\substack{m \ge 0\\m \; \text{even}}} r_m
  & = &
  \frac{2^{(p-1)k} + p -1}{2p} + \frac{(p-1) \cdot p^{k-1}}{2};
  \label{r_ev_p_odd}
  \\
  \sum_{\substack{m \ge 0\\m \; \text{odd }}} r_m 
  & = & 
  \frac{2^{(p-1)k} + p -1}{2p} -  \frac{(p-1) \cdot p^{k-1}}{2}.
  \label{r_odd_p_odd}
\end{eqnarray}
If $p = 2$, we obtain from~\eqref{r_all_explicit_p_is_2} and~\eqref{r_ev-r_odd}
since $n = k \cdot (p-1)$
\begin{eqnarray}
  \sum_{\substack{m \ge 0\\m \; \text{even}}} r_m
  & = & 2^{n-1};
  \label{r_ev_p_is_2}
  \\
  \sum_{\substack{m \ge 0\\m \; \text{odd }}} r_m 
  & = & 
  0.
  \label{r_odd_p_is_2}
\end{eqnarray}
\\[1mm]~%
\ref{lem:LambdajZ(zeta)_in_R_Q(Z/p):k_is_1} The first formula follows from~\eqref{r_m_in_terms_of_orbifold}
and applying the homomorphism $\Phi$ to \eqref{LambdalQ(zeta)}.  The rest
of~\ref{lem:LambdajZ(zeta)_in_R_Q(Z/p):k_is_1} is clear from the definitions.
\end{proof}


\typeout{------------ Section 2: Group homology ------------}

\section{Group homology}
\label{sec:Group_homology}

Next we determine the group homology of the group $\Gamma$.  Recall that for a
$\Z[G]$-module $M$, the \emph{coinvariants} are $M_G = M \otimes_{\Z[G]} \Z$.

\begin{theorem}[Homology of $B\Gamma$ and $\bub{\Gamma}$]\
  \label{the:Homology_of_BGamma_and_bub(Gamma)}
  \begin{enumerate}

  \item \label{the:Homology_of_BGamma_and_bub(Gamma):Hm(BGamma)} For $m \ge 0$,
$$H_m(\Gamma) \cong 
\begin{cases}
  \IZ^{r_m} \oplus (\IZ/p)^{s_{m+1}} & m \;\text{odd;}
  \\
  \IZ^{r_m} & m \;\text{even.}
\end{cases}
$$

\item \label{the:Homology_of_BGamma_and_bub(Gamma):coinvariants} For $m \geq 0$,
  the inclusion map $\Z^n \to \Gamma$ induces an isomorphism
$$
H_{2m}(\Z^n_{\rho})_{\Z/p} \xrightarrow{\cong} H_{2m}(\Gamma).
$$

\item \label{the:Homology_of_BGamma_and_bub(Gamma):restriction_to-p-subgr}

The map induced
by the various inclusions
$$
 \varphi_{m} \colon \bigoplus_{(P) \in \calp} H_{m}(P) \to {H}_{m}(\Gamma)
$$
is bijective for $m > n$.

\item
  \label{the:Homology_of_BGamma_and_bub(Gamma):bub(Gamma)}
  For $m \ge 0$,
$$H_m(\bub{\Gamma}) \cong 
\begin{cases}
  \IZ^{r_m} & m \;\text{odd;}
  \\
  \IZ^{r_m} \oplus (\IZ/p)^{p^k-s_{m+1}} & m \;\text{even}, m \ge 2;
  \\
  \IZ & m = 0.
\end{cases}
$$

\item \label{the:Homology_of_BGamma_and_bub(Gamma):exact_sequence} For $m \ge 1$
  the long exact homology sequence associated to the
  pushout~\eqref{pushout_for_BGamma_to_bunderbar(Gamma)}
  \begin{multline*}
    0 \to H_{2m}(\Gamma) \to H_{2m}(\bub{\Gamma}) \to \bigoplus_{(P) \in \calp}
    H_{2m-1}(P)
    \\
    \to H_{2m-1}(\Gamma) \to H_{2m-1}(\bub{\Gamma}) \to 0
  \end{multline*}
  can be identified with
  \begin{multline*}
    0 \to \IZ^{r_{2m}} \to \IZ^{r_{2m}} \oplus (\IZ/p)^{p^k - s_{2m+1}} \to
    (\IZ/p)^{p^k}
    \\
    \to \IZ^{r_{2m-1}} \oplus (\IZ/p)^{s_{2m}} \to \IZ^{r_{2m-1}} \to 0.
  \end{multline*}

\end{enumerate}
\end{theorem}

\begin{proof}\ref{the:Homology_of_BGamma_and_bub(Gamma):Hm(BGamma)}%
~\ref{the:Cohomology_of_BGamma_and_bub(Gamma):restriction_to-p-subgr}%
~\ref{the:Homology_of_BGamma_and_bub(Gamma):bub(Gamma)}
and~\ref{the:Homology_of_BGamma_and_bub(Gamma):exact_sequence}
  Recall there is a exact sequence
  \begin{multline}
    0 \to \Ext^1_{\IZ}(H^{n+1}(X),\IZ) \to H_n(X) \to \hom_{\IZ}(H^n(X),\IZ) \to
    0
    \label{universal_coeff_theorem_for_cohomology}
  \end{multline}
  for every $CW$-complex $X$ with finite skeleta, natural in $X$.  This,
  Theorem~\ref{the:Cohomology_of_BGamma_and_bub(Gamma)} and
  Corollary~\ref{cor:identifying_long_exact_sequence_for_H(Gamma)}
  imply~\ref{the:Homology_of_BGamma_and_bub(Gamma):Hm(BGamma)},%
~\ref{the:Homology_of_BGamma_and_bub(Gamma):bub(Gamma)},
  and~\ref{the:Homology_of_BGamma_and_bub(Gamma):exact_sequence}.  
\\[1mm]~\ref{the:Homology_of_BGamma_and_bub(Gamma):coinvariants} Here again we use the
  Hochschild-Serre spectral sequence
$$
E^2_{i,j} = H_i(\Z/p; H_j(\Z^n_{\rho})) \Longrightarrow H_{i+j}(\Gamma).
$$
Then the Universal Coefficient Theorem, Lemma~\ref{Tate_duality}, and
Lemma~\ref{lem:Hochschild-Serre_ss_forKast(BGamma)}~\ref{lem:Hochschild-Serre_ss_forKast(BGamma):Tate_cohomology}
imply that for $i+j$ even,
$$
\widehat{H}^{i+1}(\Z/p; H_j(\Z^n_\rho)) \cong \widehat{H}^{i+1}(\Z/p;
H^j(\Z^n_\rho)^*) \cong \widehat{H}^{-i-1}(\Z/p; H^j(\Z^n_\rho)) = 0.
$$
Hence $E^2_{i,j} = 0$ when $i+j$ is even and $i > 0$.  Since
$\widehat{H}^{-1}(\Z/p; H_{2m}(\Z^n_\rho)) = 0$, the norm map
$$
H_{2m}(\Z^n_\rho)_{\Z/p} \to H_{2m}(\Z^n_\rho)^{\Z/p}
$$
is injective.  Thus $E^2_{0,2m} = H_{2m}(\Z^n_\rho)_{\Z/p}$ is torsion-free.
Since for $i >0$, $E^2_{i,j}$ is torsion,
$$
H_{2m}(\Z^n_\rho)_{\Z/p} = E^2_{0,2m} = E^\infty_{0,2m} \xrightarrow{\cong}
H_{2m}(\Gamma).
$$
\end{proof}


\typeout{------------   Section 3: K-cohomology ------------}

\section{\texorpdfstring{$K$}{K}-cohomology}
\label{sec:K-cohomology}

Next we analyze the values of complex $K$-theory $K^*$ on $B\Gamma$ and
$\bub{\Gamma}$.  Recall that by Bott periodicity $K^*$ is $2$-periodic,
$K^0(\pt) = \IZ$, and $K^1(\pt) = 0$.

A rational computation of $K^*(BG) \otimes \IQ$ has been given for groups $G$
with a cocompact $G$-$CW$-model for $\eub{G}$
in~\cite[Theorem~0.1]{Lueck(2007)}, namely
\begin{multline*}
  K^m(BG) \otimes \IQ \xrightarrow{\cong}
  \\
  \left(\prod_{l \in \IZ} H^{2l+m}(BG;\IQ)\right) \times \left(\prod_{q
      \;\text{prime}} \prod_{(g) \in \con_q(G)} \prod_{l \in \IZ}
    H^{2l+m}(BC_G\langle g \rangle;\IQ\widehat{_q})\right),
\end{multline*}
where $\con_q(G)$ is the set of conjugacy classes (g) of elements $g \in G$ of
order $q^d$ for some integer $d \ge 1$ and $C_G\langle g \rangle$ is the
centralizer of the cyclic subgroup $\langle g \rangle$.

It gives in particular for $G = \Gamma$ because of
Theorem~\ref{the:Cohomology_of_BGamma_and_bub(Gamma)}~%
\ref{lem:preliminaries_about_Gamma_and_Zn_rho:finite_subgroups}
and~\ref{the:Cohomology_of_BGamma_and_bub(Gamma):Hm(BGamma)} and
Lemma~\ref{lem:preliminaries_about_Gamma_and_Zn_rho}
\begin{eqnarray}
  K^0(B\Gamma)  \otimes \IQ 
  & \cong & 
  \IQ^{\sum_{l \in \IZ} r_{2l}} \oplus (\IQ\widehat{_p})^{(p-1)p^k};
  \label{K0(BGamma)_rationally}
  \\
  K^1(B\Gamma)  \otimes \IQ 
  & \cong & 
  \IQ^{\sum_{l \in \IZ} r_{2l+1}}.
  \label{K1(BGamma)_rationally}
\end{eqnarray}

Recall that we have computed $\sum_{l \in \IZ} r_{2l}$ and $\sum_{l \in \IZ}
r_{2l+1}$ in Lemma~\ref{lem:LambdajZ(zeta)_in_R_Q(Z/p)}~%
\ref{lem:LambdajZ(zeta)_in_R_Q(Z/p):r_odd_and_r_ev}.

We are interested in determining the integral structure, namely, we want to show

\begin{theorem}[$K$-cohomology of $B\Gamma$ and $\bub{\Gamma}$]\
  \label{the:K-cohomology_of_BGamma_and_bub(Gamma)}
  \begin{enumerate}

  \item \label{the:K-cohomology_of_BGamma_and_bub(Gamma):K(BGamma)} For $m \in \IZ$,
$$
K^m(B\Gamma) \cong
\begin{cases}
  \IZ^{\sum_{l \in \IZ} r_{2l}} \oplus (\IZ\widehat{_p})^{(p-1)p^k} & m \: \text{even;}\\
  \IZ^{\sum_{l \in \IZ} r_{2l+1}} & m \; \text{odd;}
\end{cases}
$$
Here $\IZ\widehat{_p}$ is the $p$-adic integers.

\item \label{the:K-cohomology_of_BGamma_and_bub(Gamma):K0(BGamma)} There is a
  split exact sequence of abelian groups
$$0 \to  (\IZ\widehat{_p})^{(p-1)p^k} \to K^0(B\Gamma) 
\to K^0(B\IZ^n_{\rho})^{\IZ/p} \to 0$$ and $K^0(B\IZ^n_{\rho})^{\IZ/p} \cong
\IZ^{\sum_{l \in \IZ} r_{2l}}$.

\item \label{the:K-cohomology_of_BGamma_and_bub(Gamma):K1(BGamma)} Restricting
  to the subgroup $\Z^n$ of $\Gamma$ induces an isomorphism
  \begin{eqnarray*}
    K^1(B\Gamma) & \xrightarrow{\cong} & K^1(B\IZ^n_{\rho})^{\IZ/p}
  \end{eqnarray*}
  and $K^1(B\IZ^n_{\rho})^{\IZ/p} \cong \IZ^{\sum_{l \in \IZ} r_{2l+1}}$.

\item \label{the:K-cohomology_of_BGamma_and_bub(Gamma):K0(bub(Gamma)} We
  have $$K^0(\bub{\Gamma}) \cong \IZ^{\sum_{l \in \IZ} r_{2l}}.$$

\item \label{the:K-cohomology_of_BGamma_and_bub(Gamma):K1(bub(Gamma)} We have
$$K^1(\bub{\Gamma}) \cong \IZ^{\sum_{l \in \IZ} r_{2l+1}} \oplus T^1$$
for a finite abelian $p$-group $T^1$ for which there exists a filtration
$$T^1 = T^1_1 \supset T^1_2 \supset \cdots \supset T^1_{[(n/2)+1]} = 0$$
such that
$$T^1_i/T^1_{i+1} = (\IZ/p)^{t_i} \;\text{for}\; i = 1,2, \ldots , [(n/2)+1]$$
for integers $t_i$ which satisfy $0 \le t_i \le p^k - s_{2i+1}$.

\item \label{the:K-cohomology_of_BGamma_and_bub(Gamma):K1(bub(Gamma)_to_K1(BGamma)}
  The map $K^1(\bub{\Gamma}) \to K^1(B\Gamma)$ induces an isomorphism
$$K^1(\bub{\Gamma})/p\text{-} \tors \xrightarrow{\cong} K^1(B\Gamma)$$
Its kernel is isomorphic to $T^1$ and is isomorphic to the cokernel of the map
$$K^0(B\Gamma) 
\xrightarrow{\varphi^{0}} \bigoplus_{(P) \in \calp} \widetilde{K}^0(BP).$$
\end{enumerate}
\end{theorem}

The proof of Theorem~\ref{the:K-cohomology_of_BGamma_and_bub(Gamma)} needs some
preparation.  We will use two spectral sequences.  The \emph{Atiyah-Hirzebruch
  spectral sequence} (see~\cite[Chapter~15]{Switzer(1975)}) for topological
$K$-theory
$$
E^{i,j}_2 = H^{i}(\bub \Gamma ;K^j(\pt)) \Rightarrow K^{i+j}(\bub \Gamma)
$$
converges since $\bub \Gamma$ has a model which is a finite dimensional
$CW$-complex.  We also use the \emph{Leray-Serre spectral sequence}
(see~\cite[Chapter~15]{Switzer(1975)}) of the fibration $B\IZ^n \to B\Gamma \to B\IZ/p $.  
Recall that its $E_2$-term is $E^{i,j}_2 = H^i(\IZ/p;K^j(B\IZ^n_\rho))$ 
and it converges to $K^{i+j}(B\Gamma)$.  The Leray-Serre
spectral sequence converges (with no $\text{lim}^1$-term)
by~\cite[Theorem~6.5]{Lueck-Oliver(2001b)}.

\begin{lemma}

  \label{lem:differentials_Atiyah-Hirzebruch_ss_for_K(BGamma)}

  In the Atiyah-Hirzebruch spectral sequence converging to $K^*(\bub{\Gamma})$,
$$E^{i,j}_{\infty} \cong
\begin{cases}
  \IZ^{r_i} & i \; \text{even}, j\; \text{even};
  \\
  \IZ^{r_i} \oplus (\IZ/p)^{t_i'} & i \; \text{odd}, i \ge 3, j\; \text{even};
  \\
  0 & i=1, j \text{even};
  \\
  0 & j\; \text{odd}.
\end{cases}
$$
where $0 \le t_i' \le p^k-s_i$.
\end{lemma}

\begin{proof}~%
  Since $\bub{\Gamma}$ has a finite $CW$-model, all differentials in the
  Atiyah-Hirzebruch spectral sequence converging to $K^*(\bub{\Gamma})$ are
  rationally trivial and there exists an $N$ so that for all $i,j$, $E^{i,j}_N =
  E^{i,j}_{\infty}$.  The $E_2$-term of the Atiyah-Hirzebruch spectral sequence
  converging to $K^*(\bub{\Gamma})$ is given by
  Theorem~\ref{the:Cohomology_of_BGamma_and_bub(Gamma)}~%
~\ref{the:Cohomology_of_BGamma_and_bub(Gamma):Hm(BGamma)}
$$E^{i,j}_2 = H^i(\bub{\Gamma};K^j(\pt)) \cong
\begin{cases}
  \IZ^{r_i} & i \; \text{even}, j\; \text{even};
  \\
  \IZ^{r_i} \oplus (\IZ/p)^{p^k-s_i} & i \; \text{odd}, i \ge 3, j\;
  \text{even};
  \\
  0 & i=1, j \; \text{even};
  \\
  0 & j\; \text{odd}.
\end{cases}
$$
A map with a torsion free abelian group as target is already trivial, if it
vanishes rationally. Now consider $(i,j)$ such that it is not true that $i$ is
odd and $j$ is even. Then one shows by induction over $r\ge 2$ that
$E^{i,j}_{r}$ is zero for odd $j$ and $\IZ^{r_i}$ for even $j$, the differential
ending at $(i,j)$ in the $E_r$-term is trivial and the image of the differential
starting at $(i,j)$ is finite, and $E^{i,j}_{r}$ is an abelian subgroup of
$E^{i,j}_{r+1}$ of finite index.  Next consider $(i,j)$ such that $i$ is odd and
$j$ is even.  Then one shows by induction over $r\ge 2$ that the image of the
differential ending at $(i,j)$ in the $E^r$-term lies in the torsion subgroup of
$E^{i,j}_{r+1}$, the differential starting at $(i,j)$ is trivial, the rank of
$E^{i,j}_{r+1}$ is $r_i$ and its torsion subgroup is isomorphic to $\IZ/p^t$ for
some $t$ with $t \le p^k - s_i$.

This finishes the proof of
Lemma~\ref{lem:differentials_Atiyah-Hirzebruch_ss_for_K(BGamma)}.
\end{proof}

\begin{lemma}
  \label{lem:Leray-Serre_ss_for_K(BGamma)}\
  \begin{enumerate}
  \item \label{lem:Leray-Serre_ss_for_K(BGamma):KOast(BZn)_over_Z[Z/p]} For
    every $m \in \IZ$, there is an isomorphism of $\IZ[\IZ/p]$-modules
$$K^m(B\IZ^n_{\rho}) \cong \bigoplus_{l} H^{m+2l}(\IZ^n_{\rho});$$
In particular we get
$$K^m(B\IZ^n_{\rho})^{\IZ/p}  \cong \IZ^{\sum_{l} r_{m+2l}}.$$ 

\item \label{lem:Leray-Serre_ss_for_K(BGamma):Tate_cohomology}
$$\widehat{H}^i(\IZ/p;K^j(B\IZ^n_{\rho})) \cong
\bigoplus_{l \in \IZ} \widehat{H}^i(\IZ/p;H^{j+2l}(\IZ^n_{\rho})) \cong
\begin{cases} (\IZ/p)^{\sum_{l \in \IZ} a_{j+2l}} & i+j \;\text{even}, \\
  0 & $i+j$\;\text{odd}.
\end{cases}
$$

\item \label{lem:Leray-Serre_ss_for_K(BGamma):trivial_differentials} All
  differentials in the Leray-Serre spectral sequence are trivial.

\end{enumerate}
\end{lemma}
\begin{proof}~\ref{lem:Leray-Serre_ss_for_K(BGamma):KOast(BZn)_over_Z[Z/p]} Since $K^*(\pt)$
  is torsion free, Lemma~\ref{lem:chern_integral} below shows that the Chern
  character gives an isomorphism
$$
ch^m \colon K^m(T^n) \xrightarrow{\cong} \bigoplus_{i+j = m} H^i(T^n; K^j(\pt)) =
\bigoplus_l H^{m+2l}(T^n)
$$
Since $T^n$ is a model for the $\Z/p$-space $B\Z^n_\rho$ and $ch^m$ is natural
with respect to self-maps of the torus, $ch^m$ is an isomorphism of
$\Z[\Z/p]$-modules.

Since $H^{m+2l}(\IZ^n_{\rho})^{\IZ/p} \cong \IZ^{r_{m + 2l}}$ by
Theorem~\ref{the:Cohomology_of_BGamma_and_bub(Gamma)}%
~\ref{the:Cohomology_of_BGamma_and_bub(Gamma):Gamma_to_Zn}
and~\ref{the:Cohomology_of_BGamma_and_bub(Gamma):Hm(BGamma)},
assertion~\ref{lem:Leray-Serre_ss_for_K(BGamma):KOast(BZn)_over_Z[Z/p]} follows.
\\[1mm]%
\ref{lem:Leray-Serre_ss_for_K(BGamma):Tate_cohomology} This follows from
Lemma~\ref{lem:Hochschild-Serre_ss_forKast(BGamma)}~%
\ref{lem:Hochschild-Serre_ss_forKast(BGamma):Tate_cohomology} and
assertion~\ref{lem:Leray-Serre_ss_for_K(BGamma):KOast(BZn)_over_Z[Z/p]}.
\\[1mm]%
\ref{lem:Leray-Serre_ss_for_K(BGamma):trivial_differentials} Next we want to
show that the differentials $d^{i,j}_r$ are zero for all $r \ge 2$ and $i,j$. By
the checkerboard pattern of the $E_2$-term it suffices to show for $r \ge 2$
that the differentials $d_r^{0,j}$ are trivial for $r \ge 2$ and all odd $j \ge
1$. This is equivalent to showing that for every odd $j \ge 1$ the edge
homomorphism (see Proposition~\ref{prop:edge_homomorphisms})
$$\iota^j \colon K^j(B\Gamma) \to K^j(B\IZ^n_{\rho})^{\IZ/p} = E^{0,j}_2$$
is surjective.  To show this we use the transfer, whose properties are reviewed
in Appendix~\ref{tate_and_transfer}.  For $j$ odd,
$\widehat{H^0}(\IZ/p,K^j(\IZ^n_{\rho})) = 0$ by
assertion~\ref{lem:Leray-Serre_ss_for_K(BGamma):Tate_cohomology}.  Thus the norm
map $N = \iota^j \circ \trf^j$ is surjective, and so $\iota^j$ is surjective as
desired.
\end{proof}

Let $\calh_*$ be a generalized homology theory and $\calh^*$ a generalized
cohomology theory.  Dold defined (see~\cite{Dold(1962)} and~\cite[Example 4.1]
{Lueck(2002b)}) Chern characters
\begin{align*}
  &ch_m \colon \bigoplus_{i+j = m}H_i(X,Y; \calh_j(\pt)) \to   \calh_m(X,Y) \otimes \Q \\
  & ch^m \colon \calh^m(X,Y) \to \bigoplus_{i+j = m}H^i(X,Y; \calh^j(\pt)) \otimes
  \Q.
\end{align*}
The homological Chern character is a natural transformations of homology
theories defined on the category of $CW$-pairs.  When $X = \pt$, then $ch_m =
i_{\Q} \colon \calh_m(\pt) = \calh_m(\pt) \otimes \Z \to \calh_m(\pt) \otimes \Q$,
after the obvious identification of the targets.  Hence these Chern characters
are rational isomorphisms for any $CW$-pair.  In cohomology there are parallel
statements after restricting oneself to the category of finite $CW$-pairs.
(If the disjoint union axiom is fulfilled, finite dimensional suffices).

A $CW$-pair $(X,Y)$ is \emph{$\calh_*$-Chern integral} if for all $m$,
$$
i_{\Q} \colon \calh_m(X,Y) \to \calh_m(X,Y) \otimes \Q
$$
is a monomorphism, and $ch_m$ gives an isomorphism onto the image of $i_{\Q}$.
There is a similar definition of \emph{$\calh^*$-Chern integral} for finite
$CW$-pairs. 
 
\begin{lemma} [Chern character] \label{lem:chern_integral}\ 
  \begin{enumerate}

  \item \label{lem:chern_integral:point} A point is $\calh_*$-Chern integral if
    and only if $\calh_*(\pt)$ is torsion free.

  \item \label{lem:chern_integral:product_with_circle} If $X$ is $\calh_*$-Chern
    integral, then so is $X \times S^1$.

\end{enumerate}
Similar statements hold in cohomology.
\end{lemma}
 
\begin{proof}~\ref{lem:chern_integral:point} If a point is $\calh_*(\pt)$-Chern
  integral, then $\calh_*(\pt) \to \calh_*(\pt) \otimes \Q$ is injective, hence
  $\calh_*(\pt)$ is torsion free.  If $\calh_*(\pt)$ is torsion free, then
  $i_{\Q}$ is injective.  Since $ch_m = i_{\Q}$, a point is $\calh_*(\pt)$-Chern
  integral.

  \ref{lem:chern_integral:product_with_circle} Consider the following
  commutative diagram with split exact columns.  

$$
\xymatrix{0 \ar[d] & 0 \ar[d] & 0 \ar[d] 
\\
  \bigoplus\limits_{i+j=m} H_i(X \times D^1; \calh_j(\pt))    
\ar[d] \ar[r]^-{ch_m} 
& \calh_m(X \times D^1)  \otimes \Q  \ar[d] 
&  \calh_m(X \times D^1) \ar[d] \ar[l]_-{i_{\Q}} 
\\
  \bigoplus\limits_{i+j=m} H_i(X \times S^1; \calh_j(\pt))   
\ar@<-1ex>@{-->}[u] \ar[d] \ar[r]^-{ch_m}   
&  \calh_m(X \times S^1)  \otimes \Q \ar[d] \ar@<-1ex>@{-->}[u]  
& \calh_m(X \times S^1) \ar[d] \ar[l]_-{i_{\Q}}  \ar@<-1ex>@{-->}[u] 
\\
\bigoplus\limits_{i+j=m} H_i(X \times (S^1,D^1); \calh_j(\pt))  \ar[d] \ar[r]^-{ch_m}  
&  \calh_m(X \times (S^1,D^1)) \otimes \Q \ar[d] 
&   \calh_m(X \times (S^1,D^1))  \ar[d] \ar[l]_-{i_{\Q}}  
\\
0 & 0 & 0 }
$$

The columns come from the long exact sequence of a pair where $D^1$ is included
in $S^1$ as the upper semicircle. The splitting maps are given by a constant map
$S^1 \to D^1$.  It is elementary to see that the bottom row is isomorphic to
$$
\bigoplus_{i+j=m-1} H_i(X; \calh_j(\pt) ) \xrightarrow{ch_{m-1}} \calh_{m-1}(X)
\otimes \Q \xleftarrow{i_{\Q}} \calh_{m-1}(X).
$$
Since $X$ is $\calh_*$-Chern integral, so are $X \times (D^1,S^1)$ and $X \times
D^1$.  It follows that $X \times S^1$ is $\calh_*$-Chern integral as desired.

One may also argue by using the fact that stably $X \times S^1$ agrees with $X \vee S^1 \vee \Sigma X$ and
the property $\calh_*$-Chern integral is inherited by suspensions and wedges.
\end{proof}

\begin{proof}[Proof of Theorem~\ref{the:K-cohomology_of_BGamma_and_bub(Gamma)}]%
~\ref{the:K-cohomology_of_BGamma_and_bub(Gamma):K0(bub(Gamma)}%
~\ref{the:K-cohomology_of_BGamma_and_bub(Gamma):K1(bub(Gamma)} These
assertions follow from the Atiyah-Hirzebruch spectral sequence converging to $K^*(\bub{\Gamma})$ using
Lemma~\ref{lem:differentials_Atiyah-Hirzebruch_ss_for_K(BGamma)}.  
\\[1mm]~\ref{the:K-cohomology_of_BGamma_and_bub(Gamma):K0(BGamma)}%
~\ref{the:K-cohomology_of_BGamma_and_bub(Gamma):K1(BGamma)}%
~\ref{the:K-cohomology_of_BGamma_and_bub(Gamma):K1(bub(Gamma)_to_K1(BGamma)}
We first claim that for all $m \in \Z$, the inclusion 
$\iota \colon\Z^n \to \Gamma$ induces an epimorphism
$$\iota^m \colon K^m(B\Gamma) \to K^m(B\IZ^n_{\rho})^{\IZ/p}$$
and $K^m(B\IZ^n_{\rho})^{\IZ/p} \cong \Z^{\sum_{l \in \Z}r_m + 2l}$.  We will
also show that for $m$ odd, the map $\iota^m$ is an isomorphism.  By
Lemma~\ref{lem:Leray-Serre_ss_for_K(BGamma)}~\ref{lem:Leray-Serre_ss_for_K(BGamma):trivial_differentials},
the Leray-Serre spectral sequence collapses, so $E^{0,m}_2 = E^{0,m}_{\infty}$.
Hence the edge homomorphism $\iota^m$ is onto (see
Proposition~\ref{prop:edge_homomorphisms}).  The computation of
$K^m(B\Z^n_\rho)$ is given in
Lemma~\ref{lem:Leray-Serre_ss_for_K(BGamma)}%
~\ref{lem:Leray-Serre_ss_for_K(BGamma):KOast(BZn)_over_Z[Z/p]}.
Now assume $m$ is odd.  For any $i > 0$, $E_2^{i,m-i} = 0$ by
Lemma~\ref{lem:Leray-Serre_ss_for_K(BGamma)}%
~\ref{lem:Leray-Serre_ss_for_K(BGamma):Tate_cohomology}.
Hence $H^m(B\Gamma) = E_{\infty}^{0,m}$, so the edge homomorphism is injective.
We have now proved
assertion~\ref{the:K-cohomology_of_BGamma_and_bub(Gamma):K1(BGamma)} of our
theorem.

We next note that for all $m \in \Z$, the kernel and cokernel of the composite
$$K^m(\bub{\Gamma}) \to K^m(B\Gamma) \to K^m(B\IZ^n_{\rho})^{\IZ/p}$$
are finitely generated abelian $p$-groups.  This follows from
Proposition~\ref{prop:quotient_iso} and the following commutative diagram
$$
\xymatrix{
  B\Z^n_\rho = T^n \times S^\infty  \ar[r]^-{\simeq} \ar[d]^-{\pi}  & T^n = \R^n/\Z^n  \ar[d]^-{\underline{\pi}}   \\
  B\Gamma = T^n \times_{\Z/p} S^\infty \ar[r] & \bub\Gamma = \R^n/\Gamma   \\
}
$$

By
Lemma~\ref{lem:preliminaries_about_Gamma_and_Zn_rho}%
~\ref{lem:preliminaries_about_Gamma_and_Zn_rho:order_of_calp},
the number of conjugacy classes of order $p$ subgroups of $\Gamma$ is $p^k$.  By
the Atiyah-Segal Completion Theorem (see~\cite{Atiyah-Segal(1969)})
\begin{eqnarray}
  \widetilde{K}^m(B\IZ/p) & \cong &
  \begin{cases} 
    \II_{\IC}(\IZ/p) \otimes \IZ\widehat{_p} \cong (\IZ\widehat{_p})^{p-1} &
    \text{if} \;m \; \text{even;}
    \\
    0 & \text{if} \;m \; \text{odd}.
  \end{cases}
  \label{widetilde(Km)(BZ/p)}
\end{eqnarray}
where $\II_{\IC}(\IZ/p) \subset R_{\IC}(\IZ/p)$ is the augmentation ideal.
Hence
$$\bigoplus_{(P) \in \calp} \widetilde{K}^0(BP) \cong (\IZ\widehat{_p})^{(p-1)p^k}.$$

We are now in a position to analyze the long exact sequence
 
\begin{equation}
  0 \to K^{0}(\bub{\Gamma}) 
  \xrightarrow{\overline{f}^0} 
  K^{0}(B\Gamma) 
  \xrightarrow{\varphi^{0}} 
  \bigoplus_{(P) \in \calp} \widetilde{K}^{0}(BP)
  \xrightarrow{\delta^{0}}
  K^{1}(\bub{\Gamma}) 
  \xrightarrow{\overline{f}^1} 
  K^{1}(B\Gamma) 
  \to 0
  \label{long_exact_K-cohomology_sequences_for_bub(Gamma)_BGamma)}
\end{equation}
associated to the cellular
pushout~\eqref{pushout_for_BGamma_to_bunderbar(Gamma)}.  We will work from right
to left.
 
Since $K^1(B\Gamma) \cong K^1(B\Z^n_\rho)^{\Z/p}$ is torsion free, it follows
that the kernel of $\overline{f}^1$ equals $T^1$, the $p$-torsion subgroup of
$K^1(\bub\Gamma)$.  By exactness of
\eqref{long_exact_K-cohomology_sequences_for_bub(Gamma)_BGamma)}, $T^1$ also
equals the cokernel of $\varphi^{0}$.  This completes the proof of
assertion~\ref{the:K-cohomology_of_BGamma_and_bub(Gamma):K1(bub(Gamma)_to_K1(BGamma)}.

We showed above that $\ker \overline{f}^1 = \text{im }\delta^0$ is a finite
abelian $p$-group.  It follows that $\ker \delta^0 = \text{im } \varphi^{0}$ is
also isomorphic to $(\IZ\widehat{_p})^{(p-1)p^k}$ since any finite abelian
$p$-group $A$ is $p$-adically complete, and hence a $\Z\widehat{_p}$-module, a
$\Z$-homomorphism from $(\IZ\widehat{_p})^{(p-1)p^k} \to A$ is automatically a
$\Z\widehat{_p}$-homomorphism and $\Z\widehat{_p}$ is a principal ideal domain.

Consider the commutative diagram with exact rows
$$
\xymatrix{0 \ar[r] & K^0(\bub \Gamma) \ar[r] \ar[d]^-{\underline{\iota}^0}   
& K^0(B\Gamma) \ar[r] \ar[d]^-{\iota^0} &  \text{im }\varphi^{0}  \ar[r] \ar[d] & 0\\
  0 \ar[r] & K^0(B\Z^n_\rho)^{\Z/p} \ar[r] & K^0(B\Z^n_\rho)^{\Z/p} \ar[r] & 0 \ar[r] & 0. \\
}
$$
We have already seen that the middle vertical map is surjective with free
abelian target, hence split surjective.  Let $K$ be the kernel of $\iota^0$.
Then by the Snake Lemma, there is a short exact sequence
$$
0 \to K \to \text{im }\varphi^0 \to \text{coker}(\underline{\iota}^0) \to 0.
$$
As we noted above, $\text{im }\varphi^0 \cong (\IZ\widehat{_p})^{(p-1)p^k}$ and
$\text{coker }(\underline{\iota}^0)$ is a finite abelian $p$-group.  Thus $K$ is
also isomorphic to $(\IZ\widehat{_p})^{(p-1)p^k}$.  This completes the proof of
assertion~\ref{the:K-cohomology_of_BGamma_and_bub(Gamma):K0(BGamma)}.  \\[1mm]~%
\ref{the:K-cohomology_of_BGamma_and_bub(Gamma):K(BGamma)} This follows from
assertions~\ref{the:K-cohomology_of_BGamma_and_bub(Gamma):K0(BGamma)}
and~\ref{the:K-cohomology_of_BGamma_and_bub(Gamma):K1(BGamma)}.

This finishes the proof of
Theorem~\ref{the:K-cohomology_of_BGamma_and_bub(Gamma)}.
\end{proof}


\typeout{------------   Section 4: K-homology  ------------}

\section{\texorpdfstring{$K$}{K}-homology}
\label{sec:K-homology}

In this section we compute complex $K$-homology of $B\Gamma$ and $\bub\Gamma$.
Rationally this can be done using the Chern character of Dold~\cite{Dold(1962)}
which gives for every $CW$-complex a natural isomorphism
$$\bigoplus_{l \in \IZ} H_{m+2l}(X) \otimes \IQ \xrightarrow{\cong} 
KO_m(X) \otimes \IQ.$$ In particular we get from
Theorem~\ref{the:Homology_of_BGamma_and_bub(Gamma)}~%
\ref{the:Homology_of_BGamma_and_bub(Gamma):Hm(BGamma)}
and~\ref{the:Homology_of_BGamma_and_bub(Gamma):bub(Gamma)}
\begin{eqnarray}
  K_m(B\Gamma) \otimes \IQ & \cong & \IQ^{\sum_{l \in \IZ} r_{m+2l}};
  \label{K_m(BGamma)_rationally}
  \\
  K_m(\bub{\Gamma}) \otimes \IQ & \cong & \IQ^{\sum_{l \in \IZ} r_{m+2l}}
  \label{K_m(bub(Gamma))_rationally}
\end{eqnarray}

We are interested in determining the integral structure, namely, we want to show

\begin{theorem}[$K$-homology of $B\Gamma$ and
  $\bub{\Gamma}$]\label{the:K-homology_of_BGamma_and_bub(Gamma)}\

  \begin{enumerate}

  \item \label{the:K-homology_of_BGamma_and_bub(Gamma):K(BGamma)}

    For $m \in \Z$,
$$
K_m(B\Gamma) \cong
\begin{cases}
  \IZ^{\sum_{l \in \IZ} r_{2l}}   & m \: \text{even;}\\
  \IZ^{\sum_{l \in \IZ} r_{2l+1}} \oplus (\IZ/p^\infty)^{(p-1)p^k} & m \;
  \text{odd;}
\end{cases}
$$
Here $\Z/p^\infty = \colim_{n \to \infty} \Z/p^n \cong \Z[1/p]/\Z$.

\item \label{the:K-homology_of_BGamma_and_bub(Gamma):K0(BGamma)} The inclusion
  map $\Z^n \to\Gamma$ induces an isomorphism
$$K_0(B\IZ^n_{\rho})_{\IZ/p} \xrightarrow{\cong} K_0(B\Gamma)$$
and $K_0(B\IZ^n_{\rho})_{\IZ/p} \cong \IZ^{\sum_{l \in \IZ} r_{2l}} $.

\item \label{the:K-homology_of_BGamma_and_bub(Gamma):K1(BGamma)}

  There is a split short exact sequence of abelian groups
$$0 \to (\IZ/p^{\infty})^{(p-1)p^k} \to K_1(B\Gamma) \to K_{1}(\bub{\Gamma}) \to 0.$$

\item \label{the:K-homology_of_BGamma_and_bub(Gamma):K0(bub(Gamma)} We have
$$K_0(\bub{\Gamma}) \cong \IZ^{\sum_{l \in \IZ} r_{2l}} \oplus T^1,$$
where $T^1$ is the finite abelian $p$-group appearing in
Theorem~\ref{the:K-cohomology_of_BGamma_and_bub(Gamma)}~%
\ref{the:K-cohomology_of_BGamma_and_bub(Gamma):K1(bub(Gamma)}.

\item \label{the:K-homology_of_BGamma_and_bub(Gamma):K1(bub(Gamma)} We have
$$K_1(\bub{\Gamma}) \cong \IZ^{\sum_{l \in \IZ} r_{2l+1}}.$$

\item \label{the:K-homology_of_BGamma_and_bub(Gamma):K_0(BGamma)_to_K_0(bub(Gamma)}
  The group $T^1$ is isomorphic to a subgroup of the kernel of
$$\bigoplus_{(P) \in \calp} K_{1}(BP) 
\to K_{1}(B\Gamma).$$

\end{enumerate}
\end{theorem}

Its proof needs some preparation.

\begin{theorem}[Universal Coefficient Theorem for $K$-theory]
  \label{the:Universal_coefficient_theorem_for_K-theory}
  For any $CW$-complex $X$ there is a short exact sequence
$$0 \to \Ext_{\IZ}(K_{*-1}(X),\IZ) \to K^*(X) \to \hom_{\IZ}(K_*(X),\IZ) \to 0.$$
If $X$ is a finite $CW$-complex, there is also the $K$-homological version
$$0 \to \Ext_{\IZ}(K^{*+1}(X),\IZ) \to K_*(X) \to \hom_{\IZ}(K^*(X),\IZ) \to 0.$$
\end{theorem}
\begin{proof}
  A proof for the first short exact sequence can be found
  in~\cite{Anderson(1964)} and~\cite[(3.1)]{Yosimura(1975)}, the second sequence
  follows then from~\cite[Note~9 and~15]{Adams(1969b)}. 
\end{proof}

\begin{proof}[Proof of Theorem~\ref{the:K-homology_of_BGamma_and_bub(Gamma)}]%
~\ref{the:K-homology_of_BGamma_and_bub(Gamma):K0(bub(Gamma)}%
~\ref{the:K-homology_of_BGamma_and_bub(Gamma):K1(bub(Gamma)} These assertions
  follow from Theorem~\ref{the:K-cohomology_of_BGamma_and_bub(Gamma)}%
~\ref{the:K-cohomology_of_BGamma_and_bub(Gamma):K0(bub(Gamma)}
  and~\ref{the:K-cohomology_of_BGamma_and_bub(Gamma):K1(bub(Gamma)} and
  Theorem~\ref{the:Universal_coefficient_theorem_for_K-theory} since there is a
  finite $CW$-model for $\bub{\Gamma}$, namely $\Gamma \backslash \R^n$.
  \\[1mm]~\ref{the:K-homology_of_BGamma_and_bub(Gamma):K1(BGamma)} 
  We will use
  Pontryagin duality for locally compact abelian groups.  For such a group $G$,
  the \emph{Pontryagin dual} $\widehat G$ is $\hom(G,S^1)$, given the
  compact-open topology. A reference for the basic properties is
  \cite{Hewitt-Ross(1979)}.  These include: $\widehat G$ is also a locally
  compact abelian group.  The natural map from $G$ to its double dual is a
  isomorphism.  $G$ is discrete if and only if $\widehat G$ is compact.  If 
  $0   \to A \to B \to C \to 0$ is exact, then so is 
  $0 \to \widehat C \to \widehat B \to \widehat A \to 0$.  Our primary example of duality is
 $$
 \widehat{\Z/{p^\infty}} \cong \Z\widehat{_p}.
$$
Here $\Z/{p^\infty}$ is given the discrete topology and the $p$-adic integers
$\Z\widehat{_p}$ are given the $p$-adic topology.  This statement is included in
\cite[paragraph 25.2]{Hewitt-Ross(1979)}, but also follows from the following
assertion proved in \cite[20.8]{Lefschetz(1942)} if $H_1 \to H_2 \to H_3 \to
\cdots$ is a sequence of maps of locally compact abelian groups, then
$$
\widehat{\colim_{n \to \infty} H_n} \cong \lim_{n \to \infty} \widehat{H_n}
$$








We will now give the computation of $K_*(B \Z/p)$.  
The Atiyah-Hirzebruch Spectral Sequence shows that $\widetilde K_0(B \Z/p)= 0$.  
Vick \cite{Vick(1970)} shows that $K_1(BG)$ is the Pontryagin dual  
of $\widetilde{K}_0(BG)$ for any finite group $G$.  Applying 
these facts to $G = \Z/p$ we get (see also Knapp~\cite[Proposition~2.11]{Knapp(1978)})
\begin{eqnarray}
  K_m(B\IZ/p) \cong
  \begin{cases}
    (\IZ/p^{\infty})^{p-1} & \text{if} \; m \; \text{is odd};
    \\
    \IZ & \text{if} \; m \; \text{is even}.
  \end{cases}
  \label{K_ast(BZ/p)}
\end{eqnarray}

Thus the long exact $K$-homology sequence associated to the cellular
pushout~\eqref{pushout_for_BGamma_to_bunderbar(Gamma)} reduces to the exact
sequence
\begin{eqnarray}
  & &
  0 \to K_{0}(B\Gamma) \xrightarrow{\overline{f}_0} K_{0}(\bub{\Gamma}) \xrightarrow{\partial_0}
  \bigoplus_{(P) \in \calp} K_{1}(BP) 
  \xrightarrow{\varphi_0} K_{1}(B\Gamma) \xrightarrow{\overline{f}_1} K_{1}(\bub{\Gamma}) \to 0.
  \label{long_exact_complex_K-homology_sequence}
\end{eqnarray}
Note that $\im \partial_0$ is a finite abelian $p$-group, since it is a finitely
generated subgroup of the $p$-torsion group
$$\bigoplus_{(P) \in \calp} K_{1}(BP) \cong
(\IZ/p^{\infty})^{(p-1)p^k}.$$

Dualizing the exact sequence
$$
0 \to \im \partial_0 \to (\IZ/p^{\infty})^{(p-1)p^k} \to \im \varphi_0 \to 0,
$$
we see that $\widehat{\im \varphi_0}$ has finite $p$-power index in
$(\IZ\widehat{_p})^{(p-1)p^k}$ hence is itself isomorphic to
$(\IZ\widehat{_p})^{(p-1)p^k}$ (compare the proof of
Theorem~\ref{the:K-cohomology_of_BGamma_and_bub(Gamma)}~%
\ref{the:K-cohomology_of_BGamma_and_bub(Gamma):K0(bub(Gamma)}
and~\ref{the:K-cohomology_of_BGamma_and_bub(Gamma):K1(bub(Gamma)}).  Dualizing
again, we see $\im \varphi_0 \cong (\IZ/p^{\infty})^{(p-1)p^k}$.

The map $\overline{f}_1$ is split surjective since its target is free abelian by
assertion~\ref{the:K-homology_of_BGamma_and_bub(Gamma):K1(bub(Gamma)}.
\\[1mm]~%
\ref{the:K-homology_of_BGamma_and_bub(Gamma):K0(BGamma)} The Universal
Coefficient Theorem in $K$-theory shows that $K^0(B\Z^n_\rho) \cong
K_0(B\Z^n_\rho)^*$.  In Lemma~\ref{lem:Leray-Serre_ss_for_K(BGamma)} we showed
there is an isomorphism of $\Z[\Z/p]$-modules $K^0(B\Z^n_\rho) \cong
\oplus_{\ell} H^{2\ell}(\Z^n_\rho)$.  Now we proceed exactly as in the proof of
Theorem~\ref{the:Homology_of_BGamma_and_bub(Gamma)}~%
\ref{the:Homology_of_BGamma_and_bub(Gamma):coinvariants}, using the Leray-Serre
spectral sequence
$$
E^2_{i,j} = H_i(\Z/p; K_j(B\Z^n_{\rho})) \Longrightarrow K_{i+j}(B\Gamma).
$$
One shows $E^2_{0,2m} = K_{2m}(B\Z^n_\rho)_{\Z/p}$ is torsion-free, and for $i
>0$, $E^2_{i,j}$ has exponent $p$ and vanishes if $i + j$ is even.  Thus
$$
K_{2m}(B\Z^n_\rho)_{\Z/p} = E^2_{0,2m} = E^\infty_{0,2m} \xrightarrow{\cong}
K_{2m}(B\Gamma).
$$
By Remark~\ref{invariants_and_coinvariants_are_dual} and the Universal
Coefficient Theorem, $ (K_{2m}(B\Z^n_\rho)_{\Z/p})^* \cong \linebreak
K^{2m}(B\Z^n_\rho)^{\Z/p}$ which is isomorphic to $\Z^{\sum_{l\in\Z}r_{2l}}$ by
Lemma~\ref{lem:Leray-Serre_ss_for_K(BGamma)}~\ref{lem:Leray-Serre_ss_for_K(BGamma):KOast(BZn)_over_Z[Z/p]}.
\\[1mm]~%
\ref{the:K-homology_of_BGamma_and_bub(Gamma):K(BGamma)} This follows from
assertions~\ref{the:K-homology_of_BGamma_and_bub(Gamma):K0(BGamma)},%
~\ref{the:K-homology_of_BGamma_and_bub(Gamma):K1(BGamma)},
and~\ref{the:K-homology_of_BGamma_and_bub(Gamma):K1(bub(Gamma)}.  \\[1mm]~%
\ref{the:K-homology_of_BGamma_and_bub(Gamma):K_0(BGamma)_to_K_0(bub(Gamma)} This
follows from
assertion~\ref{the:K-homology_of_BGamma_and_bub(Gamma):K0(bub(Gamma)} and the
exact sequence~\eqref{long_exact_complex_K-homology_sequence}. This finishes the
proof of Theorem~\ref{the:K-homology_of_BGamma_and_bub(Gamma)}.
\end{proof}


\typeout{------------   Section 5: KO-cohomology  ------------}

\section{\texorpdfstring{$KO$}{KO}-cohomology}
\label{sec:KO-cohomology}

In this section we compute real $K$-cohomology $KO^*$ of $B\Gamma$.

Recall that by Bott periodicity $KO^*$ is $8$-periodic, i.e., there is a natural
isomorphism $KO^m(X) \cong KO^{m+8}(X)$ for every $m \in \IZ$ and $CW$-complex
$X$, and $KO^{-m}(\pt)$ is given for $m = 0,1,2 ,\ldots 7$ by $\IZ$, $\IZ/2$,
$\IZ/2$, $0$, $\IZ$, $0$, $0$, $0$.  We will assume from now on that $p$ is odd
in order to avoid the extra difficulties arising from the fact that $KO^m(\pt)
\cong \IZ/2$ for $m = 1,2$.

\begin{theorem}[$KO$-cohomology of $B\Gamma$ and $\bub{\Gamma}$]
  \label{the:KO-cohomology_of_BGamma_and_bub(Gamma)}
  Let $p$ be an odd prime and let $m$ be any integer.

  \begin{enumerate}

  \item \label{the:KO-cohomology_of_BGamma_and_bub(Gamma):K(BGamma)}
$$
KO^m(B\Gamma) \cong
\begin{cases}
  \left(\oplus_{l \in \IZ}  KO^{m-l}(\pt)^{r_l} \right)  \oplus (\IZ\widehat{_p})^{p^k(p-1)/2} & m \: \text{even;}\\
  \oplus_{l \in \IZ} KO^{m-l}(\pt)^{r_l} & m \; \text{odd.}
\end{cases}
$$

\item \label{the:KO-cohomology_of_BGamma_and_bub(Gamma):KO2m(BGamma)} There is a
  split exact sequence of abelian groups
$$ 0 \to (\IZ\widehat{_p})^{p^k(p-1)/2} \to KO^{2m}(B\Gamma)
\to KO^{2m}(B\IZ^n_{\rho})^{\IZ/p} \to 0,$$ and $KO^{2m}(B\IZ^n_{\rho})^{\IZ/p}
\cong \bigoplus_{l \in \IZ} KO^{2m-l}(\pt)^{r_l}$.

\item \label{the:KO-cohomology_of_BGamma_and_bub(Gamma):KO2m_plus_1(BGamma)}
  Restricting to the subgroup $\Z^n$ of $ \Gamma$ induces an isomorphism
$$
KO^{2m+1}(B\Gamma) \xrightarrow{\cong} KO^{2m+1}(B\IZ^n_{\rho})^{\IZ/p}.
$$
and $KO^{2m+1}(B\IZ^n_{\rho})^{\IZ/p} \cong \bigoplus_{l \in \IZ}
KO^{2m+1-l}(\pt)^{r_l}$.

\item \label{the:KO-cohomology_of_BGamma_and_bub(Gamma):KO2m(bub(Gamma)} We have
$$KO^{2m}(\bub{\Gamma}) \cong 
\oplus_{l \in \IZ} KO^{2m-l}(\pt)^{r_l}.$$

\item \label{the:KO-cohomology_of_BGamma_and_bub(Gamma):K12m_plus_1(bub(Gamma)}
  We have
$$KO^{2m+1}(\bub{\Gamma}) \cong  
\left(\bigoplus_{l \in \IZ} KO^{2m+1-l}(\pt)^{r_l}\right) \oplus TO^{2m+1},$$
where $TO^{2m+1}$ is a finite abelian $p$-group for which there exists a
filtration
$$TO^{2m+1} = TO^{2m+1}_1 \supset TO^{2m+1}_2 
\supset \cdots \supset TO^{2m+1}_{[(n+4-(-1)^m)/4]} = 0$$ such that
$TO^{2m+1}_i/TO^{2m+1}_{i+1} = (\IZ/p)^{to_i}$ holds for integers $to_i$ which
satisfy $0 \le to_i \le p^k - s_{4i+(-1)^m}$.

\item \label{the:KO-cohomology_of_BGamma_and_bub(Gamma):KO_odd(BGamma)_to_KO_odd(bub(Gamma))}
  The map $KO^{2m+1}(\bub{\Gamma}) \to KO^{2m+1}(B\Gamma)$ induces an
  isomorphism
$$KO^{2m+1}(\bub{\Gamma})/\ptors \xrightarrow{\cong} KO^{2m+1}(B\Gamma).$$
Its kernel is isomorphic to $TO^{2m+1}$ and is isomorphic to the cokernel of the
map
$$KO^{2m}(B\Gamma) 
\to \bigoplus_{(P) \in \calp} \widetilde{KO}^{2m}(BP).$$
\end{enumerate}
\end{theorem}

\begin{lemma}
  \label{lem:differentials_Atiyah-Hirzebruch_ss_for_KO-cohomology_localized_at_p}
  Let $p$ be an odd prime.  In the Atiyah-Hirzebruch spectral sequence
  converging to $K^*(\bub{\Gamma})$ after localizing at $p$
$$\left(E^{i,j}_{\infty}\right)_{(p)}  \cong
\begin{cases}
  \IZ_{(p)}^{r_i} & i \; \text{even}, j\equiv 0 \pmod 4;
  \\
  \IZ_{(p)}^{r_i} \oplus (\IZ/p)^{t_i'} & i \; \text{odd}, i \ge 3, j\equiv 0
  \pmod 4;
  \\
  0 & i=1, j \equiv 0 \pmod 4;
  \\
  0 & j \not \equiv 0 \pmod 4,
\end{cases}
$$
where $0 \le t_i' \le p^k-s_i$.

\end{lemma}

\begin{proof}~%
  Because of Theorem~\ref{the:Cohomology_of_BGamma_and_bub(Gamma)}%
~\ref{the:Cohomology_of_BGamma_and_bub(Gamma):Hm(BGamma)} the $E_2$-term of the
  spectral sequence converging to $K^*(B\Gamma)_{(p)}$ is given after
  localization at $p$ by
  \begin{multline*}
    \bigl(E^{i,j}_2\bigr)_{(p)}   = H^i(\bub{\Gamma};KO^j(\pt))_{(p)} \\
    \cong
    \begin{cases} \IZ_{(p)}^{r_i} & i \; \text{even, } j\equiv 0 \pmod 4;
      \\
      \IZ_{(p)}^{r_i} \oplus (\IZ/p)^{p^k -s_i} & i \; \text{odd, } i \ge 3,
      j\equiv 0 \pmod 4;
      \\
      0 & i=1, j \equiv 0 \pmod 4;
      \\
      0 & j\not\equiv 0 \pmod 4.
    \end{cases}
  \end{multline*}

  The rest of the proof is analogous to the proof of
  Lemma~\ref{lem:differentials_Atiyah-Hirzebruch_ss_for_K(BGamma)}.
\end{proof}

\begin{lemma}
  \label{lem_KO(BZn)_over_Z[Z/p]}
  Let $p$ be an odd prime.  For every $m \in \Z$, there are isomorphisms of
  $\IZ[\IZ/p]$-modules
  \begin{align*}
    KO^m(B\IZ^n_{\rho}) \otimes \Z[1/2] & \cong
    \bigoplus_{i \in \IZ} H^i(B\IZ^n_{\rho}) \otimes KO^{m-i}(\pt) \otimes \Z[1/2] \\
    KO^m(B\IZ^n_{\rho}) \otimes \Z[1/p] & \cong \bigoplus_{i \in \IZ}
    H^i(B\IZ^n_{\rho}) \otimes KO^{m-i}(\pt) \otimes \Z[1/p]
  \end{align*}
\end{lemma}
\begin{proof}

  Since $KO^*(X) \otimes \Z[1/2]$ is a generalized cohomology theory with
  torsion-free coefficients, the Chern character and
  Lemma~\ref{lem:chern_integral} give the first isomorphism.

  One proves that there are isomorphisms of abelian groups
$$
KO^m(B\IZ^n_{\rho}) \cong \bigoplus_{i \in \IZ} H^i(B\IZ^n_{\rho}) \otimes
KO^{m-i}(\pt)
$$
by induction on $n$ using excision and the fact that $B\Z^n = S^1 \times
B\Z^{n-1}$.  It follows that the Atiyah-Hirzebruch spectral sequence $E^{i,j}_2
= H^i(B\Z^n;K^j(\pt)[1/p]) \Rightarrow KO^{i+j}(B\IZ^n)[1/p] $ collapses.
This spectral sequence is natural with respect to automorphisms of
$\IZ^n$. Hence we obtain a descending filtration by $\IZ[1/p] [\IZ/p]$-modules
$$KO^m(B\IZ^n_{\rho})[1/p]  = F^{0,m} \supset F^{1,m-1} \supset F^{2,m-2} \supset \cdots 
\supset F^{m,0} \supset F^{m+1,-1} = 0$$ and exact sequences
$$
0 \to F^{i+1,m-i-1} \to F^{i,m-i} \xrightarrow{\pi} H^i(\IZ^n_{\rho}) \otimes
K^{m-i}(\pt) \otimes {\Z[1/p]} \to 0.
$$
It thus suffices to show that these exact sequences split over
$\IZ[1/p][\IZ/p]$ for all $i$.  If $m-i \equiv 3,5,6,7 \pmod 8$, this follows
from the fact that $KO^{m-i}(\pt) = 0$.  If $m-i \equiv 0,4 \pmod 8$, then
$K^{m-i}(\pt) \cong \Z$ and $H^i(\IZ^n_{\rho}) \otimes K^{m-i}(\pt) \otimes
{\Z[1/p]}$ is a finitely generated $\Z[1/p]$-torsion free module over the ring
$\Z[1/p][\Z/p] \xrightarrow{\cong} \Z[1/p] \times \Z[1/p][\zeta]$, hence is
projective.  Finally, suppose $m-i \equiv 1,2 \pmod 8$.  Since the
Atiyah-Hirzebruch spectral sequence collapses, there is a homomorphism of
abelian groups $s\colon H^i(\IZ^n_{\rho}) \otimes K^{m-i}(\pt) \otimes {\Z[1/p]} \to
F^{i,m-i}$ so that $\pi \circ s = \id$.  Define
$$
\tilde s \colon H^i(\IZ^n_{\rho}) \otimes K^{m-i}(\pt) \otimes {\Z[1/p]} \to
F^{i,m-i}, \qquad x \mapsto \sum_{g \in \Z/p} g \cdot s(g^{-1}x).
$$
Then $\tilde s$ is a homomorphism of $\Z[\Z/p]$-modules and $\pi \circ \tilde s$
is multiplication by $p$ and hence is the identity since $K^{m-i}(\pt) \cong \Z/2$.
\end{proof}

\begin{lemma}
  \label{lem:Leray-Serre_ss_for_KO(BGamma)}\ Let $p$ be an odd prime.
  \begin{enumerate}

  \item \label{lem:Leray-Serre_ss_for_KO(BGamma):fixed_set} 
  For every $m \in \Z$, there is an isomorphism of abelian groups
$$
KO^m(B\Z^n_\rho)^{\Z/p} \cong \bigoplus_{l \in \Z} KO^{m-l}(\pt)^{r_l}.
$$

\item \label{lem:Leray-Serre_ss_for_KO(BGamma):Tate_cohomology}
$$\widehat{H}^i(\IZ/p;KO^j(B\IZ^n_{\rho})) \cong
\bigoplus_{l \in \IZ} \widehat{H}^i(\IZ/p;H^{j+4l}(\IZ^n_{\rho})) \cong
\begin{cases} (\IZ/p)^{\sum_{l \in \IZ} a_{j+4l}} & i+j \;\text{even}, \\
  0 & $i+j$\;\text{odd.}
\end{cases}
$$

\item \label{lem:Leray-Serre_ss_for_KO(BGamma):differentials} All differentials
  in the Leray-Serre spectral sequence associated to the
  extension~\eqref{Gamma_as_extension} converging to $KO^*(B\Gamma)$ vanish.

\end{enumerate}

\end{lemma}

\begin{proof}~\ref{lem:Leray-Serre_ss_for_KO(BGamma):fixed_set} 
  It suffices to show the
  isomorphism exists after inverting 2 and after localizing at 2.  Furthermore,
  if $M$ is a $\Z[\Z/p]$-module, then $M^{\Z/p} \otimes \Z[1/2] \cong (M \otimes
  \Z[1/2])^{\Z/p}$ and $M^{\Z/p} \otimes \Z_{(2)} \cong (M \otimes \Z_{(2)}
  )^{\Z/p}$ since localization is an exact functor.  The assertion then follows
  from Lemma~\ref{lem_KO(BZn)_over_Z[Z/p]} and the definition of the numbers
  $r_l$.  
\\[1mm]~\ref{lem:Leray-Serre_ss_for_KO(BGamma):Tate_cohomology} 
  Since $\Z[1/2] \subset
  \Z_{(p)}$, Lemma~\ref{lem_KO(BZn)_over_Z[Z/p]} implies that
$$KO^j(B\Z^n_\rho) \otimes \Z_{(p)} \cong \bigoplus_{l \in \IZ} H^{j+4l}(B\IZ^n_{\rho}) \otimes \Z_{(p)}.$$
The first isomorphism in
assertion~\ref{lem:Leray-Serre_ss_for_KO(BGamma):Tate_cohomology} then follows
since localization is an exact functor and the Tate cohomology groups are
$p$-torsion.  The second isomorphism follows from
Lemma~\ref{lem:Hochschild-Serre_ss_forKast(BGamma)}%
~\ref{lem:Hochschild-Serre_ss_forKast(BGamma):Tate_cohomology}.
\\[1mm]~%
\ref{lem:Leray-Serre_ss_for_KO(BGamma):differentials} First note that the
Leray-Serre spectral sequence converges with no $\text{lim}^1$-term, see
\cite[Theorem~6.5]{Lueck-Oliver(2001b)}.

It suffices to prove the differentials vanish after inverting p and after
localizing at $p$.  If we invert $p$, the claim follows from
$$E^{i,j}_2[1/p] = H^i(\IZ/p;KO^j(B\IZ^n_{\rho}))[1/p] = 0 \quad \text{for}\; i \ge 1.$$ 
If we localize at $p$, the proof that the differentials vanish is identical to
the proof of
Lemma~\ref{lem:Leray-Serre_ss_for_K(BGamma)}~\ref{lem:Leray-Serre_ss_for_K(BGamma):trivial_differentials}.
\end{proof}

\begin{proof}[Proof of
  Theorem~\ref{the:KO-cohomology_of_BGamma_and_bub(Gamma)}]
\ref{the:KO-cohomology_of_BGamma_and_bub(Gamma):KO2m(bub(Gamma)} 
We first note that Proposition~\ref{prop:quotient_iso} and 
Lemma~\ref{lem:Leray-Serre_ss_for_KO(BGamma)}~\ref{lem:Leray-Serre_ss_for_KO(BGamma):fixed_set} 
imply that for all $m \in \Z$, the kernel and cokernel of the composite
  \begin{equation} \label{isomorphism_away_from_p} KO^m(\bub\Gamma) \to
    KO^m(B\Gamma) \to KO^m(B\Z^n_\rho)^{\Z/p} \cong \oplus_{l\in \Z}
    KO^{m-l}(\pt)^{r_l}
  \end{equation}
  are finitely generated $p$-groups.  This implies that the desired isomorphism
  holds after inverting $p$.  It holds at $p$ by
  Lemma~\ref{lem:differentials_Atiyah-Hirzebruch_ss_for_KO-cohomology_localized_at_p}.
 \\~\ref{the:KO-cohomology_of_BGamma_and_bub(Gamma):KO2m_plus_1(BGamma)} 
As in the proof of Theorem~\ref{the:K-cohomology_of_BGamma_and_bub(Gamma)}, 
one shows that the map
$$
\iota^m \colon KO^m(B\Gamma) \to KO^m(B\Z^n_\rho)^{\Z/p}
$$
is an isomorphism for $m$ odd and an epimorphism for $m$ even.
\\~\ref{the:KO-cohomology_of_BGamma_and_bub(Gamma):K12m_plus_1(bub(Gamma)}%
~\ref{the:KO-cohomology_of_BGamma_and_bub(Gamma):KO_odd(BGamma)_to_KO_odd(bub(Gamma))} 
Since $p$ is odd, every non-trivial irreducible $\IZ/p$-representation is of complex type. 
Hence we get from~\cite[Remark on page 133 after Proposition~2.2]{Segal(1968c)} that 
$KO^m_{\IZ/p}(\pt) \cong KO^m(\pt) \oplus K^m(\pt) \otimes \II_{\IR}(\IZ/p)$. The Atiyah-Segal Completion Theorem for 
$KO^*$ (see~\cite{Atiyah-Segal(1969)}) implies
\begin{eqnarray}
  \widetilde{KO}^m(B\IZ/p) \cong
  \begin{cases}
    \II_{\IR}(\IZ/p) \otimes \IZ\widehat{_p} \cong (\IZ\widehat{_p})^{(p-1)/2} &
    m \;\text{even};
    \\
    0 & \text{otherwise}.
  \end{cases}
  \label{reduced_KO-cohomology_of_BP}
\end{eqnarray}
The cellular pushout~\eqref{pushout_for_BGamma_to_bunderbar(Gamma)} yields for
$m \in \IZ$ a long exact sequence
\begin{multline}
  0 \to KO^{2m}(\bub{\Gamma}) \xrightarrow{\overline{f}^{2m}} KO^{2m}(B\Gamma)
  \xrightarrow{\varphi^{2m}} \bigoplus_{(P) \in \calp} \widetilde{KO}^{2m}(BP)
  \\
  \xrightarrow{\delta^{2m}} KO^{2m+1}(\bub{\Gamma})
  \xrightarrow{\overline{f}^{2m+1}} KO^{2m+1}(B\Gamma) \to 0.
  \label{long_exact_KO-cohomology_sequences_for_bub(Gamma)_BGamma)}
\end{multline}
Define $TO^{2m+1}$ to be the kernel of the surjection $\overline{f}^{2m+1}$.
Since $\overline{f}^{2m+1}$ is an isomorphism after inverting $p$
by~\eqref{isomorphism_away_from_p} and
assertion~\ref{the:KO-cohomology_of_BGamma_and_bub(Gamma):KO2m_plus_1(BGamma)},
$TO^{2m+1}$ is $p$-torsion.  We next claim $\overline{f}^{2m+1}$ is split.  We
only need verify this after localizing at $p$ in which case it follows since
$K^{2m+1}(B\Gamma) \otimes \Z_{(p)}$ is free over $\Z_{(p)}$ by
assertion~\ref{the:KO-cohomology_of_BGamma_and_bub(Gamma):KO2m_plus_1(BGamma)}
and
Lemma~\ref{lem:Leray-Serre_ss_for_KO(BGamma)}~\ref{lem:Leray-Serre_ss_for_KO(BGamma):fixed_set}.
Finally, the stated filtration of $TO^{2m+1}$ is a consequence of
Lemma~\ref{lem:differentials_Atiyah-Hirzebruch_ss_for_KO-cohomology_localized_at_p}.
The completes the proof of
assertion~\ref{the:KO-cohomology_of_BGamma_and_bub(Gamma):K12m_plus_1(bub(Gamma)}.
Assertion~\ref{the:KO-cohomology_of_BGamma_and_bub(Gamma):KO_odd(BGamma)_to_KO_odd(bub(Gamma))}
is a consequence.  \\[1mm]~%
\ref{the:KO-cohomology_of_BGamma_and_bub(Gamma):KO2m(BGamma)} The proof of this
is identical to that of
Theorem~\ref{the:K-cohomology_of_BGamma_and_bub(Gamma)}%
~\ref{the:K-cohomology_of_BGamma_and_bub(Gamma):K0(BGamma)};
the only missing part is to show the epimorphism
$$
\iota^{2m} \colon KO^{2m}(B\Gamma) \to KO^{2m}(B\Z^n_\rho)^{\Z/p}
$$
is split.  At $p$, this follows since $KO^{2m}(B\Z^n_\rho)^{\Z/p} \otimes
\Z_{(p)}$ is free over $\Z_{(p)}$.  After inverting $p$, the splitting is
provided by composing the inverse of the composite
\eqref{isomorphism_away_from_p} with the map $KO^{2m}(\bub\Gamma)[1/p] \to
KO^{2m}(B\Gamma)[1/p]$.
\end{proof}


\typeout{------------   Section 6: KO-homology ------------}

\section{\texorpdfstring{$KO$}{KO}-homology}
\label{sec:KO-homology}

\

In this section we want to compute the real $K$-homology $KO_*$ of $B\Gamma$ and
$\bub{\Gamma}$. Rationally this can be done using the Chern character of
Dold~\cite{Dold(1962)}: for every $CW$-complex there is a natural isomorphism
$$\bigoplus_{l \in \IZ} H_{m+4l}(X) \otimes \IQ \xrightarrow{\cong} 
KO_m(X) \otimes \IQ.$$ In particular we get from
Theorem~\ref{the:Homology_of_BGamma_and_bub(Gamma)}~%
\ref{the:Homology_of_BGamma_and_bub(Gamma):Hm(BGamma)}
and~\ref{the:Homology_of_BGamma_and_bub(Gamma):bub(Gamma)}
\begin{eqnarray}
  KO_m(B\Gamma) \otimes \IQ & \cong & \IQ^{\sum_{l \in \IZ} r_{m+4l}};
  \label{KO_m(BGamma)_rationally}
  \\
  KO_m(\bub{\Gamma}) \otimes \IQ & \cong & \IQ^{\sum_{l \in \IZ} r_{m+4l}}
  \label{KO_m(bub(Gamma))_rationally}
\end{eqnarray}

We are interested in determining the integral structure.

\begin{theorem}[$KO$-homology of $B\Gamma$ and $\bub{\Gamma}$]
  \label{the:KO-homology_of_BGamma_and_bub(Gamma)}
  Let $p$ be an odd prime and $m$ be any integer.

  \begin{enumerate}

  \item \label{the:KO-homology_of_BGamma_and_bub(Gamma):KO(BGamma)}
$$
KO_m(B\Gamma) \cong
\begin{cases}
  \IZ^{\sum_{l \in \IZ} r_{2l}}   & m \: \text{even;}\\
  \IZ^{\sum_{l \in \IZ} r_{2l+1}} \oplus (\IZ/p^\infty)^{p^k(p-1)/2} & m \;
  \text{odd;}
\end{cases}
$$

\item \label{the:KO-homology_of_BGamma_and_bub(Gamma):KO_2m(BGamma)} The
  inclusion map $\Z^n \to\Gamma$ induces an isomorphism
$$KO_{2m}(B\IZ^n_{\rho})_{\Z/p}  \xrightarrow{\cong} KO_{2m}(B\Gamma)$$
and $KO_{2m}(B\IZ^n_{\rho})_{\Z/p} \cong \bigoplus_{l \in \Z}
KO_{2m-l}(\pt)^{r_{l}}.  $

\item \label{the:KO-homology_of_BGamma_and_bub(Gamma):KO_2m_plus_1(BGamma)}
  There is a split short exact sequence of abelian groups
$$0 \to (\IZ/p^{\infty})^{p^k(p-1)/2} \to KO_{2m+1}(B\Gamma)
\to KO_{2m+1}(\bub{\Gamma}) \to 0.$$

\item \label{the:KO-homology_of_BGamma_and_bub(Gamma):KO2m(bub(Gamma)} We have
$$KO_{2m}(\bub{\Gamma}) \cong 
\left(\bigoplus_{l\in\Z} KO_{2m-l}(\pt)^{r_{l}}\right) \oplus TO^{2m+5}$$ where
$TO^{2m+5} $ is the finite abelian $p$-group appearing in
Theorem~\ref{the:KO-cohomology_of_BGamma_and_bub(Gamma)}~%
\ref{the:KO-cohomology_of_BGamma_and_bub(Gamma):K12m_plus_1(bub(Gamma)}.

\item \label{the:KO-homology_of_BGamma_and_bub(Gamma):KO2m_plus_1(bub(Gamma)} We
  have
$$KO_{2m+1}(\bub{\Gamma}) \cong 
\bigoplus_{l\in\Z} KO_{2m+1-l}(\pt)^{r_{l}}.
$$

\item\label{the:KO-homology_of_BGamma_and_bub(Gamma):KO(bub(Gamma):K_ev(BGamma)_to_K_ev(bub(Gamma)}
  The group $TO^{2m+5}$ is isomorphic to a subgroup of the kernel of
$$\bigoplus_{(P) \in \calp} KO_{2m+1}(BP) 
\to KO_{2m+1}(B\Gamma).$$

\end{enumerate}

\end{theorem}

\begin{theorem}[Universal Coefficient Theorem for $KO$-theory]
  \label{lem:Universal_Coefficient_Theorem_for_KO-theory}
  For any $CW$-complex $X$ there is a short exact sequence
$$0 \to \Ext_{\IZ}(KO_{n+3}(X),\IZ) \to KO^n(X) \to \hom(KO_{n+4}(X),\IZ) \to 0.$$
If $X$ is a finite $CW$-complex, there is a short exact sequence
$$0 \to \Ext_{\IZ}(KO^{n+5}(X),\IZ) \to KO_n(X) \to \hom_{\IZ}(KO^{n+4}(X),\IZ) \to 0.$$
\end{theorem}
\begin{proof}
  A proof for the first short exact sequence can be found
  in~\cite{Anderson(1964)} and~\cite[(3.1)]{Yosimura(1975)}, the second sequence
  follows then from~\cite[Note~9 and~15]{Adams(1969b)}.  

\end{proof}

\begin{proof}[Proof of Theorem~\ref{the:KO-homology_of_BGamma_and_bub(Gamma)}]%
~\ref{the:KO-homology_of_BGamma_and_bub(Gamma):KO2m(bub(Gamma)}%
~\ref{the:KO-homology_of_BGamma_and_bub(Gamma):KO2m_plus_1(bub(Gamma)}
  These assertions follow from
  Theorem~\ref{the:KO-cohomology_of_BGamma_and_bub(Gamma)}%
~\ref{the:KO-cohomology_of_BGamma_and_bub(Gamma):KO2m(bub(Gamma)}
  and~\ref{the:KO-cohomology_of_BGamma_and_bub(Gamma):K12m_plus_1(bub(Gamma)}
  and the Universal Coefficient Theorem for
  $KO$-theory~\ref{lem:Universal_Coefficient_Theorem_for_KO-theory}.  
  \\[1mm]~\ref{the:KO-homology_of_BGamma_and_bub(Gamma):KO_2m(BGamma)} 
  There are natural
  transformations of cohomology theories $i^* \colon KO^* \to K^*$ and $r^*\colon K^* \to
  KO^*$, induced by sending a real representation $V$ to its complexification
  $\C \otimes_{\R} V$ and a complex representation to its restriction as a real
  representation.  The composite $r^* \circ i^*\colon KO^* \to KO^*$ is
  multiplication by two.  Since the map
$$K_0(B\IZ^n_{\rho})_{\IZ/p}  \xrightarrow{\cong} K_0(B\Gamma).$$
is bijective by Theorem~\ref{the:K-homology_of_BGamma_and_bub(Gamma)}~%
\ref{the:K-homology_of_BGamma_and_bub(Gamma):K0(BGamma)}, the map
$$KO_{2m}(B\IZ^n_{\rho})_{\IZ/p}  \xrightarrow{\cong} KO_{2m}(B\Gamma)$$
is bijective after inverting $2$. In order to show that it is itself bijective,
it remains to show that it is bijective after inverting $p$.  This follows from
Proposition~\ref{prop:quotient_iso}.

Since we are dealing with $KO$-homology, the Atiyah-Hirzebruch spectral sequence
converges also for the infinite-dimensional $CW$-complex $B\Gamma$. Because of
the existence of Dold's Chern character, all its differentials vanish
rationally. For $m \in \IZ$ we have $H_{2m}(B\Gamma) \cong \IZ^{r_{2m}}$ by
Theorem~\ref{the:Homology_of_BGamma_and_bub(Gamma)}.  Hence we get for an odd
prime $p$ since $KO_m(\pt)_{(p)}$ is $\IZ_{(p)}$ for $m \equiv 0 \pmod 4$ and
$0$ otherwise
$$
KO_{2m}(B\Gamma)_{(p)} \cong (\IZ_{(p)})^{\sum_{l \in \IZ} r_{2m+4l}};
$$
We conclude that
$$KO_{2m}(B\Gamma) \cong 
\bigoplus_{l\in\Z} KO_{2m-l}(\pt)^{r_{l}};
$$
holds after localizing at $p$. It remains to show that it holds after inverting
$p$. This follows from Proposition~\ref{prop:quotient_iso} and
assertion~\ref{the:KO-homology_of_BGamma_and_bub(Gamma):KO2m(bub(Gamma)}.
\\[1mm]~%
\ref{the:KO-homology_of_BGamma_and_bub(Gamma):KO_2m_plus_1(BGamma)} The
Atiyah-Hirzebruch spectral sequence shows that $\widetilde{KO}_{2m}(B\Z/p) = 0$
for all $m \in \Z$.  The methods of \cite{Vick(1970)} together with the
Universal Coefficient Theorem for KO-theory show that
$\widetilde{KO}_{2m+3}(BG)$ is the Pontryagin dual of $\widetilde{KO}^{2m}(BG)$
for any finite group $G$.  Applying these facts to $G= \Z/p$ for an odd prime
$p$, we see
$$
\widetilde{KO}_m(B\Z/p) =
\begin{cases}
  (\Z/p^\infty)^{(p-1)/2} & m \; \text{odd;}\\
  0 & m \;\text{even.}
\end{cases}
$$
Thus the long exact $KO$-homology
sequence associated to the cellular
pushout~\eqref{pushout_for_BGamma_to_bunderbar(Gamma)} reduces to the exact
sequence
\begin{multline}
  \label{KO_homology_long_exact_sequence}
  0 \to KO_{2m}(B\Gamma) \xrightarrow{\overline{f}_{2m}} KO_{2m}(\bub{\Gamma})
  \xrightarrow{\partial_{2m}}
  \bigoplus_{(P) \in \calp} KO_{2m-1}(BP) \\
  \xrightarrow{\varphi_{2m-1}} KO_{2m-1}(B\Gamma)
  \xrightarrow{\overline{f}_{2m-1}} KO_{2m-1}(\bub{\Gamma}) \to 0.
\end{multline}
Note $\text{im }\partial_{2m}$ is a finite abelian $p$-group, since it is a
finitely generated subgroup of the $p$-torsion group
$$
\bigoplus_{(P)\in \calp} KO_{2m-1}(BP) \cong (\Z/p^{\infty})^{(p-1)p^k/2}.
$$
Thus $\text{im }\varphi_{2m-1} \cong (\Z/p^{\infty})^{(p-1)p^k/2}$ (compare with
the proof of
Theorem~\ref{the:K-cohomology_of_BGamma_and_bub(Gamma)}%
~\ref{the:K-cohomology_of_BGamma_and_bub(Gamma):K1(BGamma)}).
It remains to see that $\overline{f}_{2m-1}$ splits, which we verify at $p$ and
away from $p$.  The target of $\overline{f}_{2m-1}$ is free after localizing at
$p$ by
assertion~\ref{the:KO-homology_of_BGamma_and_bub(Gamma):KO2m_plus_1(bub(Gamma)},
so it splits.  After inverting $p$, the exact
sequence~\ref{KO_homology_long_exact_sequence} shows that
$\overline{f}_{2m-1}[1/p]$ is an isomorphism.  \\[1mm]~%
\ref{the:KO-homology_of_BGamma_and_bub(Gamma):KO(BGamma)} This follows from
assertions~\ref{the:KO-homology_of_BGamma_and_bub(Gamma):KO_2m(BGamma)},%
~\ref{the:KO-homology_of_BGamma_and_bub(Gamma):KO_2m_plus_1(BGamma)}
and~\ref{the:KO-homology_of_BGamma_and_bub(Gamma):KO2m_plus_1(bub(Gamma)}.
\\[1mm]~%
\ref{the:KO-homology_of_BGamma_and_bub(Gamma):KO(bub(Gamma):K_ev(BGamma)_to_K_ev(bub(Gamma)}
This follows from
assertions~\ref{the:KO-homology_of_BGamma_and_bub(Gamma):KO_2m(BGamma)}
and~\ref{the:KO-homology_of_BGamma_and_bub(Gamma):KO2m(bub(Gamma)} and the long
exact sequence~\eqref{KO_homology_long_exact_sequence}.  This finishes the proof
of Theorem~\ref{the:KO-homology_of_BGamma_and_bub(Gamma)}.
\end{proof}


\typeout{------------   Section 7: Equivariant K-cohomology ------------}

\section{Equivariant \texorpdfstring{$K$}{K}-cohomology}
\label{sec:Equivariant_K-cohomology}

In the sequel an equivariant cohomology theory is to be understood in the sense
of~\cite[Section~1]{Lueck(2005c)}.  Equivariant topological complex $K$-theory
$K_?^*$ is an example as shown in~\cite[Example~1.6]{Lueck(2005c)} based
on~\cite{Lueck-Oliver(2001b)}.  This applies also to equivariant topological
real $K$-theory $KO_?^*$.

Rationally one obtains
\begin{eqnarray*}
K^0_{\Gamma}(\eub{\Gamma}) \otimes \IQ & \cong &
\IQ^{(p-1)p^k + \sum_{l \in \IZ} r_{2l}};
\\
K^1_{\Gamma}(\eub{\Gamma}) \otimes \IQ & \cong &
\IQ^{\sum_{l \in \IZ} r_{2l+1}},
\end{eqnarray*}
from~\cite[Theorem~5.5 and Lemma~5.6]{Lueck-Oliver(2001b)}
using Theorem~\ref{the:Cohomology_of_BGamma_and_bub(Gamma)}~%
\ref{the:Cohomology_of_BGamma_and_bub(Gamma):bub(Gamma)}
and Lemma~\ref{lem:preliminaries_about_Gamma_and_Zn_rho}.
We want to get an integral computation. Recall that we have
computed $\sum_{l \in \IZ} r_{2l}$ and
$\sum_{l \in \IZ} r_{2l+1}$ in Lemma~\ref{lem:LambdajZ(zeta)_in_R_Q(Z/p)}%
~\ref{lem:LambdajZ(zeta)_in_R_Q(Z/p):r_odd_and_r_ev}.

\begin{theorem}[Equivariant $K$-cohomology of $\eub{\Gamma}$]
\label{the:equivariant_K-cohomology_of_eub(Gamma)}
\ \begin{enumerate}

\item \label{the:equivariant_K-cohomology_of_eub(Gamma):Kast_explicite} 
$$K^m_{\Gamma}(\eub{\Gamma}) \cong
\begin{cases}
\IZ^{(p-1)p^k + \sum_{l \in \IZ} r_{2l}} & m \; \text{even};
\\
\IZ^{\sum_{l \in \IZ} r_{2l+1}} & m \; \text{odd}.
\end{cases}
$$

\item \label{the:equivariant_K-cohomology_of_eub(Gamma):T1} 
There is an exact sequence
$$0 \to K^0(\bub{\Gamma}) \to K^0_{\Gamma}(\eub{\Gamma}) \to
\bigoplus_{(P) \in \calp} \II_{\IC}(P)
\to T^1 \to 0,$$
where $T^1$ is the finite abelian $p$-group appearing in
Theorem~\ref{the:K-cohomology_of_BGamma_and_bub(Gamma)}~%
\ref{the:K-cohomology_of_BGamma_and_bub(Gamma):K1(bub(Gamma)}.

\item \label{the:equivariant_K-cohomology_of_eub(Gamma):K1_isos}
The canonical maps
\begin{eqnarray*}
K^1_{\Gamma}(\eub{\Gamma) 
&\xrightarrow{\cong} &
K^1(B\Gamma});
\\
K^1(B\Gamma)   
& \xrightarrow{\cong} &  
K^1(B\IZ^n_{\rho})^{\IZ/p},
\end{eqnarray*}
are isomorphisms.

\end{enumerate}

\end{theorem}

In the sequel we will often use the following lemma.

\begin{lemma} \label{lem:long_exact_sequence_for_equivariant_(co-)homology}\
\begin{enumerate}

\item \label{lem:long_exact_sequence_for_equivariant_(co-)homology:cohomology}
Let $\calh_?^*$ be an equivariant cohomology theory in the sense
of~\cite[Section~1]{Lueck(2005c)}. Then there is a long exact sequence
\begin{multline*}
\cdots \to \calh^m(\bub{\Gamma}) 
\xrightarrow{\ind_{\Gamma \to 1}} \calh^m_{\Gamma}(\eub{\Gamma}) 
\xrightarrow{\varphi^m}
\bigoplus_{(P) \in \calp} \overline{\calh}^m_{P}(\pt) 
\\
\to  \calh^{m+1}(\bub{\Gamma}) 
\xrightarrow{\ind_{\Gamma \to 1}} \calh^{m+1}_{\Gamma}(\eub{\Gamma}) 
\to 
\cdots
\end{multline*}
where $\overline{\calh}^m_{P}(\pt)$ is the cokernel of the induction map
$\ind_{P \to 1} \colon \calh^m(\pt) \to \calh^m_{P}(\pt)$
and the map $\varphi^m$ is induced by the various
inclusions $P \to \Gamma$.

The map 
$$\ind_{\Gamma \to 1}[1/p] \colon  \calh^m(\bub{\Gamma})[1/p] 
\to \calh^m_{\Gamma}(\eub{\Gamma})[1/p]$$
is split injective.

\item \label{lem:long_exact_sequence_for_equivariant_(co-)homology:homology}
Let $\calh^?_*$ be an equivariant homology theory in the sense of~\cite[Section~1]{Lueck(2002b)}.
Then there is a long exact sequence
\begin{multline*}
\cdots \to \calh_{m+1}^{\Gamma}(\eub{\Gamma})  
\xrightarrow{\ind_{\Gamma \to 1}}
\calh_{m+1}(\bub{\Gamma})
\to 
\bigoplus_{(P) \in \calp} \widetilde{\calh}_m^{P}(\pt) 
\\
\xrightarrow{\varphi_m}
\calh_{m}^{\Gamma}(\eub{\Gamma})  
\xrightarrow{\ind_{\Gamma \to 1}}
\calh_{m}(\bub{\Gamma})
\to 
\cdots
\end{multline*}
where $\widetilde{\calh}_m^{P}(\pt)$ is the kernel of the induction map
$\ind_{P \to 1} \colon \calh_m^{P}(\pt) \to \calh_m(\pt)$
and the map $\varphi_m$ is induced by the various
inclusions $P \to \Gamma$.

The map 
$$\ind_{\Gamma \to 1}[1/p] \colon  \calh^{\Gamma}_m(\eub{\Gamma})[1/p]
\to \calh_m(\bub{\Gamma})[1/p]$$
is split surjective.

\end{enumerate}
\end{lemma}
\begin{proof}~%
\ref{lem:long_exact_sequence_for_equivariant_(co-)homology:cohomology}
From the cellular $\Gamma$-pushout~\eqref{G-pushoutfor_EGamma_to_eunderbar_Gamma}
we obtain a long exact sequence
\begin{multline}
\cdots \to \calh^m_{\Gamma}(\eub{\Gamma}) 
\to
\calh^m_{\Gamma}(E\Gamma) \oplus \bigoplus_{(P) \in \calp} \calh^m_{\Gamma}(\Gamma/P)  
\\
\to \bigoplus_{(P) \in \calp} \calh^m_{\Gamma}(\Gamma \times_P EP) 
\to \calh^{m+1}_{\Gamma}(\eub{\Gamma}) 
\to \calh^{m+1}_{\Gamma}(E\Gamma) \oplus \bigoplus_{(P) \in \calp} \calh^{m+1}_{\Gamma}(\Gamma/P) 
 \to \cdots.
\label{K_Gamma-cohomology_Mayer_Vietoris_sequence}
\end{multline}
From the cellular pushout~\eqref{pushout_for_BGamma_to_bunderbar(Gamma)} 
we obtain the long 
exact sequence
\begin{multline}
\cdots \to \calh^m(\bub{\Gamma}) 
\to 
\calh^m(B\Gamma) \oplus \bigoplus_{(P) \in \calp} \calh^m(\pt)  
\\
\to \bigoplus_{(P) \in \calp} \calh^{m}(BP)
\to \calh^{m+1}(\bub\Gamma)  
\to \calh^{m+1}(B\Gamma)  
\oplus \bigoplus_{(P) \in \calp} \calh^{m+1}(\pt)
\to \cdots.
\label{K_cohomology_Mayer_Vietoris_sequence}
\end{multline}
Induction with the group homomorphism $ \Gamma \to 1$ yields
a map from the long exact sequence~\eqref{K_cohomology_Mayer_Vietoris_sequence}
to the long exact sequence~\eqref{K_Gamma-cohomology_Mayer_Vietoris_sequence}.
Recall that the induction homomorphism 
$\calh^m(\Gamma\backslash X) \to \calh^m_{\Gamma}(X)$
is an isomorphism if $\Gamma$ acts freely on the proper
$\Gamma$-$CW$-complex $X$. Therefore the maps
\begin{eqnarray*}
\bigoplus_{(P) \in \calp} \calh^{m}(BP) 
& \xrightarrow{\cong} &
\bigoplus_{(P) \in \calp} \calh^{m}_{\Gamma}(\Gamma \times_P EP);
\\
\calh^m(B\Gamma) 
& \xrightarrow{\cong} &
\calh^m_{\Gamma}(E\Gamma),
\end{eqnarray*}
are bijective. 
Hence one can splice the
long exact sequences~\eqref{K_Gamma-cohomology_Mayer_Vietoris_sequence}
and~\eqref{K_cohomology_Mayer_Vietoris_sequence} together
to obtain the desired long exact sequence, after noting the commutative diagram
$$
\xymatrix{\calh^m_\Gamma(\Gamma/P)  & \calh^m(\pt) \ar[l]_-{\ind_{\Gamma \to 1}}  \\
\calh^m_P(\pt)  \ar[u]^{\ind_{P \to \Gamma}}_{\cong}  & \calh^m(\pt) \ar[u]_{=}  \ar[l]_-{\ind_{P \to 1}}
}
$$

We have the following commutative diagram, where the vertical arrow
are given by induction with the group homomorphism $\Gamma \to 1$
$$
\xymatrix{
\calh^m(\bub{\Gamma})
\ar@{>->}[r] \ar[d] &
{\calh^m(B\IZ^n)} \ar[d]^{\cong}\\
{\calh^m_{\Gamma}(\eub{\Gamma})}
\ar[r]  &
{\calh^m_{\Gamma}(\Gamma \times_{\IZ^n} E\IZ^n)}
}
$$

The upper horizontal arrow is split injective after inverting $p$ by
Proposition~\ref{prop:quotient_iso}. 
The right vertical arrow is bijective since
$\Gamma$ acts freely on $\Gamma \times_{\IZ^n} E\IZ^n$. Hence 
$\calh^m(\bub{\Gamma}) \to \calh^m_{\Gamma}(\eub{\Gamma})$ is injective after inverting
$p$. 
\\[1mm]~%
\ref{lem:long_exact_sequence_for_equivariant_(co-)homology:homology}
The proof is analogous to the one of 
assertion~\ref{lem:long_exact_sequence_for_equivariant_(co-)homology:cohomology}.
This finishes the proof of Lemma~\ref{lem:long_exact_sequence_for_equivariant_(co-)homology}.
\end{proof}

\begin{proof}[Proof of Theorem~\ref{the:equivariant_K-cohomology_of_eub(Gamma)}]
 Recall that $K^0_{\Gamma}(\Gamma/P) \cong R_{\IC}(P)$ and
$K^1_{\Gamma}(\Gamma/P) \cong 0$.
Hence we obtain from Lemma~\ref{lem:long_exact_sequence_for_equivariant_(co-)homology}~%
\ref{lem:long_exact_sequence_for_equivariant_(co-)homology:cohomology}
the long exact sequence
\begin{multline}
0 \to K^0(\bub{\Gamma}) \to K^0_{\Gamma}(\eub{\Gamma}) \to
\bigoplus_{(P) \in \calp} \overline{R}_{\IC}(P)
\to K^1(\bub{\Gamma}) \to K^1_{\Gamma}(\eub{\Gamma}) \to 0,
\label{spliced_K_cohomology_sequence}
\end{multline}
where $\overline{R}_{\IC}(P)$ is the cokernel of the homomorphism
$R_{\IC}(1) \to R_{\IC}(P)$ given by restriction with $P \to 1$. 
Notice that the composite of the augmentation ideal $\II_{\IC}(P) \to
R_{\IC}(P)$ with the projection $R_{\IC}(P) \to \overline{R}_{\IC}(P)$
is an isomorphism  of finitely generated free abelian groups
\begin{eqnarray}
\II_{\IC}(P) &\xrightarrow{\cong} & \overline{R}_{\IC}(P)
\label{II_C_and_overlineR_C}
\end{eqnarray}
and that
$\II_{\IC}(P)$ is isomorphic to $\IZ^{p-1}$.


\ref{the:equivariant_K-cohomology_of_eub(Gamma):K1_isos}
We have already shown in Theorem~\ref{the:K-cohomology_of_BGamma_and_bub(Gamma)}~%
\ref{the:K-cohomology_of_BGamma_and_bub(Gamma):K1(BGamma)} 
that the map $K^1(B\Gamma)  \xrightarrow{\cong}  K^1(B\IZ^n_{\rho})^{\IZ/p}$
is bijective and that 
$K^1(B\Gamma) \cong \IZ^{\sum_{l \in \IZ} r_{2l+1}}$. Hence it remains to prove that
the composite
$$K^1_{\Gamma}(\eub{\Gamma}) \to K^1_{\Gamma}(E\Gamma)  \xrightarrow{\cong} K^{1}(B\Gamma)$$
is bijective.
We obtain from~\eqref{long_exact_K-cohomology_sequences_for_bub(Gamma)_BGamma)} 
and~\eqref{spliced_K_cohomology_sequence} the following commutative
diagram with exact rows
$$
\xymatrix{
\bigoplus_{(P) \in \calp} \overline{R}_{\IC}(P) \ar[r] \ar[d]
&
K^1(\bub{\Gamma})  \ar[r] \ar[d]_{\id}
&
K^1_{\Gamma}(\eub{\Gamma}) \ar[r] \ar[d]
& 
0
\\
\bigoplus_{(P) \in \calp} \widetilde{K}^{0}(BP) \ar[r]
&
K^{1}(\bub{\Gamma}) \ar[r]
&
K^{1}(B\Gamma) \ar[r]
&
0
}
$$
By the five lemma it suffices to show that the map
$$
\ker\left(K^{1}(\bub{\Gamma}) \to K^{1}_{\Gamma}(\eub{\Gamma})\right)
\to 
\ker\left(K^{1}(\bub{\Gamma}) \to K^{1}(B\Gamma)\right)
$$
is surjective. 
We conclude from Theorem~\ref{the:K-cohomology_of_BGamma_and_bub(Gamma)}~%
\ref{the:K-cohomology_of_BGamma_and_bub(Gamma):K1(bub(Gamma)_to_K1(BGamma)}
that the  kernel of
$K^{1}(\bub{\Gamma}) \to K^{1}(B\Gamma)$ is the  finite abelian $p$-group
$T^1$ appearing in Theorem~\ref{the:K-cohomology_of_BGamma_and_bub(Gamma)}~%
\ref{the:K-cohomology_of_BGamma_and_bub(Gamma):K1(bub(Gamma)}.
Hence it remains to show for every  integer $l > 0$
that the obvious composite
$$\bigoplus_{(P) \in \calp} R_{\IC}(P)
\to \bigoplus_{(P) \in \calp} K^{0}(BP) \to 
\left(\bigoplus_{(P) \in \calp} K^{0}(BP)\right)
\left/ 
p^l\cdot\left(\bigoplus_{(P) \in \calp} K^{0}(BP)\right)\right.$$
is surjective.
By the Atiyah-Segal Completion Theorem the map
$R_{\IC}(P) \to K^{0}(BP)$ can be identified with the map
$$\id \oplus i \colon \IZ \oplus I(\IZ/p) 
\to \IZ \oplus \bigl(I(\IZ/p) \otimes \IZ\widehat{_p}\bigr)$$
Hence it suffices to show that the composite
$$\IZ \to \IZ\widehat{_p} \to \IZ\widehat{_p}/p^l   \IZ\widehat{_p}$$
is surjective. This is true since the latter map can be identified with the canonical
epimorphism $\IZ \to \IZ/p^l$.
\\[1mm]~%
\ref{the:equivariant_K-cohomology_of_eub(Gamma):T1}
This follows from 
 Theorem~\ref{the:K-cohomology_of_BGamma_and_bub(Gamma)}~%
\ref{the:K-cohomology_of_BGamma_and_bub(Gamma):K1(bub(Gamma)_to_K1(BGamma)},
the long exact sequence~\eqref{spliced_K_cohomology_sequence}, the isomorphism~\eqref{II_C_and_overlineR_C}
and assertion~\ref{the:equivariant_K-cohomology_of_eub(Gamma):K1_isos}.
\\[1mm]~%
\ref{the:equivariant_K-cohomology_of_eub(Gamma):Kast_explicite}
We have shown  $K^0(\bub{\Gamma}) \cong \IZ^{\sum_{l \in \IZ} r_{2l}}$ in 
 Theorem~\ref{the:K-cohomology_of_BGamma_and_bub(Gamma)}%
~\ref{lem:preliminaries_about_Gamma_and_Zn_rho:order_of_calp}.
We have $\II(\IZ/p) \cong \IZ^{(p-1)/2}$. The order of $\calp$ is $p^k$ by
Lemma~\ref{lem:preliminaries_about_Gamma_and_Zn_rho}%
~\ref{lem:preliminaries_about_Gamma_and_Zn_rho:order_of_calp}.
Hence we conclude from assertion~\ref{the:equivariant_K-cohomology_of_eub(Gamma):T1}
$$K^0_{\Gamma}(\eub{\Gamma}) \cong \IZ^{(p-1)p^k + \sum_{l \in \IZ} r_{2l}}.$$
The computation of $K^1_{\Gamma}(\eub{\Gamma})$ follows
from Theorem~\ref{the:K-cohomology_of_BGamma_and_bub(Gamma)}~%
\ref{the:K-cohomology_of_BGamma_and_bub(Gamma):K1(BGamma)} 
and assertion~\ref{the:equivariant_K-cohomology_of_eub(Gamma):K1_isos}.
\end{proof}

\begin{remark}[Geometric interpretation of $T^1$]
\label{rem:geometric_interpretation_of_TO1}
The exact sequence appearing in Theorem~\ref{the:equivariant_K-cohomology_of_eub(Gamma)}~%
\ref{the:equivariant_K-cohomology_of_eub(Gamma):T1} 
has the following interpretation in terms of equivariant vector bundles.
Since $\Gamma$ is a crystallographic group, $\Gamma$ acts
properly on $\IR^n$ such that this action reduced to $\IZ^n$ is the free standard action
and $\IR^n$ is a model for $\eub{\Gamma}$.
Hence the quotient of $\IZ^n\backslash \IR^n$ is the standard $n$-torus $T^n$
together with a $\IZ/p$-action.  There is a bijection
$$\calp \xrightarrow{\cong} (T^n)^{\IZ/p}$$
coming from the fact that $(\IR^n)^P$ consists of exactly one point
for $(P) \in \calp$. In particular $(T^n)^{\IZ/p}$
consists of $p^k$ points (see Lemma~\ref{lem:preliminaries_about_Gamma_and_Zn_rho}%
~\ref{lem:preliminaries_about_Gamma_and_Zn_rho:fixed_set}.)
Hence for any complex $\IZ/p$-vector bundle $\xi$ 
we get a collection of complex $\IZ/p$-representations $\{\xi_x \mid x \in (T^n)^{\IZ/p}\}$
satisfying $\dim_{\IC}(\xi_x) = \dim_{\IC}(\xi_y) = \dim(\xi)$ for $x,y \in (T^n)^{\IZ/p}$. 
This yields a map
$$\beta \colon K_{\IZ/p}^0(T^n) \to \bigoplus_{P \in (P)} I_{\IC}(P).$$
sending the class of a $\IZ/p$-vector bundle $\xi$ to the collection
$\{[\xi_x] - \dim(\xi) \cdot [\IC] \mid x \in (T^n)^{\IZ/p}\}$.
Let 
$$\alpha \colon K^0\bigl((\IZ/p)\backslash T^n\bigr)  \to K_{\IZ/p}^0(T^n)$$
be the homomorphism coming from the pullback construction associated
to the projection $T^n \to (\IZ/p)\backslash T^n$.
We obtain
the exact sequence
$$0 \to K^0\bigl((\IZ/p)\backslash T^n\bigr) \xrightarrow{\alpha} 
K_{\IZ/p}^0(T^n) \xrightarrow{\beta}
\bigoplus_{(P) \in \calp} I_{\IC}(P)\to  T^1\to 0$$
which can be identified with exact sequence of 
Theorem~\ref{the:equivariant_K-cohomology_of_eub(Gamma)}~%
\ref{the:equivariant_K-cohomology_of_eub(Gamma):T1}. 

Thus the group $T^1$ is related to (stable version of) the question when a collection of
$\IZ/p$-representations $\{V_x \mid x \in (T^n)^{\IZ/p}\}$
with $\dim_{\IC}(V_x) = \dim_{\IC}(V_y)$ for $x,y \in (T^n)^{\IZ/p}$
can be realized as the fibers of a $\IZ/p$-vector bundle $\xi$ over $T^n$
at the points in $(T^n)^{\IZ/p}$. 

Moreover, a $\IZ/p$-vector bundle over $T^n$ is stably isomorphic to the pullback of
a vector bundle over $(\IZ/p)\backslash T^n$ if and only if
for every $x \in (T^n)^{\IZ/p}$ the $\IZ/p$-representation $\xi_x$ has trivial
$\IZ/p$-action.
\end{remark}


\typeout{------------   Section 8: Equivariant K-homology  ------------}

\section{Equivariant \texorpdfstring{$K$}{K}-homology}
\label{sec:Equivariant_K-homology}

In the sequel equivariant homology theory is to be understood in the sense
of~\cite[Section~1]{Lueck(2002b)}. Equivariant topological complex $K$-homology $K^?_*$ is an
example (see~\cite{Davis-Lueck(1998)},~\cite[Section~6]{Lueck-Reich(2005)}).
The construction there yields the same 
for proper $G$-$CW$-complexes as the construction
due to Kasparov~\cite{Kasparov(1988)}. It is two-periodic. 
For finite groups $G$ the group $K_m^G(\pt)$ is $R_{\IC}(G)$
for even $m$ and trivial for odd $m$.

We obtain from~\cite[Theorem~0.7]{Lueck(2002d)} using
Lemma~\ref{lem:preliminaries_about_Gamma_and_Zn_rho} an isomorphism
\begin{eqnarray*}
  K_m(\bub{\Gamma})\left[\frac{1}{p}\right] \oplus 
  \bigoplus_{(P)  \in \calp} 
  K_m(\pt) \otimes I_{\C}(P)\left[\frac{1}{p}\right]
  & \cong &
  K_m^{\Gamma}(\eub{\Gamma})\left[\frac{1}{p}\right]
\end{eqnarray*} 
and hence from Theorem~\ref{the:K-homology_of_BGamma_and_bub(Gamma)}
\begin{eqnarray}
  K_0^{\Gamma}(\eub{\Gamma})\left[\frac{1}{p}\right]
  & \cong &
  \bigl(\IZ[1/p]\bigr)^{(p-1)p^k + \sum_{l} r_{2l}}
  \\
  K_1^{\Gamma}(\eub{\Gamma})\left[\frac{1}{p}\right]
  & \cong &
  \bigl(\IZ[1/p]\bigr)^{\sum_{l} r_{2l+1}}.
\end{eqnarray} 
We want to get an integral computation.  

\begin{theorem}[Equivariant $K$-homology of $\eub{\Gamma}$]\
  \label{the:equivariant_K-homology_of_eub(Gamma)}
  \begin{enumerate}

  \item \label{the:equivariant_K-homology_of_eub(Gamma):explicit} We have
$$K_m^{\Gamma}(\eub{\Gamma}) \cong
\begin{cases}
  \IZ^{(p-1)p^k + \sum_{l \in \IZ} r_{2l}} & m \; \text{even};
  \\
  \IZ^{\sum_{l \in \IZ} r_{2l+1}} & m \; \text{odd}.
\end{cases}
$$

\item \label{the:equivariant_K-homology_of_eub(Gamma):EXT} There is a natural
  isomorphism
$$K^{\Gamma}_m(\eub{\Gamma})
\xrightarrow{\cong} \hom_{\IZ}\bigl(K_{\Gamma}^m(\eub{\Gamma}),\IZ).$$

\item \label{the:equivariant_K-homology_of_eub(Gamma):K_1} The map
  $K_{1}^{\Gamma}(\eub{\Gamma}) \to K_{1}(\bub{\Gamma})$ is an isomorphism.

  There is an exact sequence
$$0 \to \bigoplus_{(P) \in \calp} \widetilde{R}_{\IC}(P) 
\to K_{0}^{\Gamma}(\eub{\Gamma}) \to K_{0}(\bub{\Gamma}) \to 0,$$ where
$\widetilde{R}_{\IC}(P)$ is the kernel of the map $R_{\IC}(P) \to R_{\IC}(1)$
sending $[V]$ to $[\IC \otimes_{\IC P} V]$.  It splits after inverting $p$.

\end{enumerate}
\end{theorem}
Its proof needs some preparation.

\begin{lemma}
  \label{lem:ext_RC(G)_versus_Ext_Z}
  Let $G$ be a finite group.  Then there is an isomorphism of
  $R_{\IC}(G)$-modules
$$R_{\IC}(G) \xrightarrow{\cong} \hom_{\IZ}(R_{\IC}(G),\IZ)$$
which sends $[V]$ to the homomorphism $R_{\IC}(G) \to \IZ, \quad [W] \mapsto
\dim_{\IC}\bigl(\hom_{\IC G}(V,W)\bigr)$. Here $R_{\IC}(G)$ acts on 
$\hom_{\IZ}(R_{\IC}(G);\IZ)$ by $([V] \cdot
\phi)([W]) := \phi([V^*] \cdot [W])$.

In particular we get for any $R_{\IC}(G)$-module $M$ a natural isomorphism of
$R_{\IC}(G)$-modules
$$\Ext^i_{R_{\IC}(G)}(M,R_{\IC}(G)) \xrightarrow{\cong} \Ext^1_{\IZ}(M,\IZ)
\quad \text{for}\; i \ge 0.$$
\end{lemma}
\begin{proof}
  See~\cite[2.5 and~2.10]{Madsen(1986)}.  
\end{proof}

\begin{theorem}[Universal coefficient theorem for equivariant K-theory]
  \label{the:Universal_coefficient_theorem_for_equivariant_K-theory}
  Let $G$ be a finite group and $X$ be a finite $G$-$CW$-complex. Then there are
  for $n \in \IZ$ natural exact sequences of $R_{\IC}(G)$-modules
 $$0 \to \Ext_{R_{\IC}(G)}\bigl(K^G_{n-1}(X),R_{\IC}(G)\bigr) \to
K_G^n(X) \to \hom_{R_{\IC}(G)}\bigl(K_n^G(X),R_{\IC}(G)\bigr) \to 0.$$ and
$$0 \to \Ext_{R_{\IC}(G)}\bigl(K_G^{n+1}(X),R_{\IC}(G)\bigr) \to
K_n^G(X) \to \hom_{R_{\IC}(G)}\bigl(K_G^n(X),R_{\IC}(G)\bigr) \to 0.$$
\end{theorem}
\begin{proof}
  The first sequence is proved in~\cite{Boekstedt(1981)}.  The second sequence
  follows from the first by equivariant S-duality (see~\cite{Madsen(1986)}, \cite{Wirthmueller(1975)}).
\end{proof}

\begin{proof}[Proof of Theorem~\ref{the:equivariant_K-homology_of_eub(Gamma)}]%
~\ref{the:equivariant_K-homology_of_eub(Gamma):EXT} Since $\IZ^n$ acts freely
  on $\eub{\Gamma}$, induction with $\Gamma \to \IZ/p$ induces isomorphisms
  \begin{eqnarray*}
    K^{\Gamma}_n(\eub{\Gamma}) 
    & \xrightarrow{\cong} & 
    K_n^{\IZ/p}(\IZ^n\backslash \eub{\Gamma});
    \\
    K_{\IZ/p}^n(\IZ^n\backslash \eub{\Gamma}) 
    & \xrightarrow{\cong} & 
    K_{\Gamma}^n(\eub{\Gamma}).
  \end{eqnarray*} 
  Since $\IZ^n\backslash \eub{\Gamma}$ is a finite $\IZ/p$-$CW$-complex, we
  obtain from Lemma~\ref{lem:ext_RC(G)_versus_Ext_Z} and
  Theorem~\ref{the:Universal_coefficient_theorem_for_equivariant_K-theory} the
  exact sequence of $R_{\IC}(\IZ/p)$-modules
$$0 \to \Ext_{\IZ}^1\bigl(K_{\IZ/p}^{n+1}(\IZ^n\backslash \eub{\Gamma}),\IZ\bigr) 
\to K_n^{\IZ/p}(\IZ^n\backslash \eub{\Gamma}) \to
\hom_{\IZ}\bigl(K_{\IZ/p}^n(\IZ^n\backslash \eub{\Gamma}),\IZ) \to 0.$$ 
(Another construction of the sequence above will be given in~\cite{Joachim-Lueck(2010)}.)
Hence we get an exact sequence of $R_{\IC}(\IZ/p)$-modules (see
also~\cite[Proposition~2.8]{Madsen(1986)})
$$0 \to \Ext_{\IZ}^1\bigl(K_{\Gamma}^{n+1}(\eub{\Gamma}),\IZ\bigr) 
\to K^{\Gamma}_n(\eub{\Gamma}) \to
\hom_{\IZ}\bigl(K_{\Gamma}^n(\eub{\Gamma}),\IZ) \to 0.$$ Since
$K_{\Gamma}^{n+1}(\eub{\Gamma})$ is a finitely generated free abelian group for
all $n \in \IZ$ by Theorem~\ref{the:equivariant_K-cohomology_of_eub(Gamma)}, we
obtain for $n \in \IZ$ an isomorphism of $R_{\IC}(\IZ/p)$-modules
$$K^{\Gamma}_n(\eub{\Gamma})
\xrightarrow{\cong} \hom_{\IZ}\bigl(K_{\Gamma}^n(\eub{\Gamma}),\IZ)$$%
\ref{the:equivariant_K-homology_of_eub(Gamma):explicit} Apply
Theorem~\ref{the:equivariant_K-cohomology_of_eub(Gamma)}~%
\ref{the:equivariant_K-cohomology_of_eub(Gamma):Kast_explicite} and
assertion~\ref{the:equivariant_K-homology_of_eub(Gamma):EXT} to get the concrete
identification of $K^{\Gamma}_n(\eub{\Gamma})$.  \\[1mm]~%
\ref{the:equivariant_K-homology_of_eub(Gamma):K_1} From
Lemma~\ref{lem:long_exact_sequence_for_equivariant_(co-)homology}~%
\ref{lem:long_exact_sequence_for_equivariant_(co-)homology:homology} we obtain a
long exact sequence
\begin{multline}
  0 \to K_{1}^{\Gamma}(\eub{\Gamma}) \to K_{1}(\bub{\Gamma}) \to \bigoplus_{(P)
    \in \calp} \widetilde{K}_{0}^{\IZ/p}(\pt) \to K_{0}^{\Gamma}(\eub{\Gamma})
  \to K_{0}(\bub{\Gamma}) \to 0.
  \label{spliced_K_homology_sequence}
\end{multline}
where $\widetilde{K}_{0}^{\IZ/p}(\pt)$ is the kernel of the map
$K_{0}^{\IZ/p}(\pt) \to K_{0}(\pt)$ coming from induction with $\IZ/p \to 1$.
Since $K_{1}^{\Gamma}(\eub{\Gamma})$ and $K_{1}(\bub{\Gamma})$ are finitely
generated free abelian groups of the same rank by
assertion~\ref{the:equivariant_K-cohomology_of_eub(Gamma):Kast_explicite} and
Theorem~\ref{the:K-homology_of_BGamma_and_bub(Gamma)}~%
\ref{the:K-homology_of_BGamma_and_bub(Gamma):K1(bub(Gamma)} and 
$\bigoplus_{(P)  \in \calp} \widetilde{K}_{0}^{\IZ/p}(\pt)$ is torsion free, the map
$K_{1}^{\Gamma}(\eub{\Gamma}) \to K_{1}(\bub{\Gamma})$ is bijective and we get a
short exact sequence
$$0 \to \bigoplus_{(P) \in \calp} \widetilde{K}_{0}^{\IZ/p}(\pt) 
\to K_{0}^{\Gamma}(\eub{\Gamma}) \to K_{0}(\bub{\Gamma}) \to 0.$$

\end{proof}


\typeout{------------   Section 9: Equivariant KO-cohomology  ------------}

\section{Equivariant \texorpdfstring{$KO$}{}-cohomology}
\label{sec:Equivariant_KO-cohomology}

Recall that equivariant topological real $KO$-theory $KO_?^*$ is an
equivariant cohomology theory in the sense 
of~\cite[Section~1]{Lueck(2005c)}. It is $8$-periodic.
Recall also that equivariant topological real $K$-homology $KO^?_*$ is an
equivariant homology theory in the sense 
of~\cite[Section~1]{Lueck(2002b)}. It is $8$-periodic.
  
We first give some information about $KO^G_m(\pt)$ and
$KO_G^m(\pt)$ for finite $G$. We have $KO_m^G(\pt) = KO^{-m}_G(\pt)$ 
If $G$ is a finite group, then we get for $m \in \IZ$
$$KO^{-m}_G(\pt) \cong KO_m^G(\pt) \cong K_m(\IR G)$$
where $K_m(\IR G)$ is the topological $K$-theory of the real group $C^*$-algebra
$\IR G$. Let $\{V_i \mid i = 0,1,2, \ldots , r\}$ be a complete set of representatives 
for the $\IR G$-isomorphism classes of irreducible real $G$-representations.
By Schur's Lemma the endomorphism ring
$D_i = \Endo_{\IR G}(V_i)$ is a skewfield over $\IR$ and hence
isomorphic to $\IR$, $\IC$ or $\IH$. There are positive integers $k_i$ for
$i \in \{0,1,\ldots , r\}$ such that we obtain a splitting
$$\IR G \cong \prod_{i = 0}^r M_{k_i}(D_i).$$
Since topological $K$-theory is compatible with products, by Morita equivalence
we obtain for $m \in \IZ$ an isomorphism
\begin{eqnarray}
K_m(\IR G) & \cong & \prod_{i = 1}^r K_m(D_i)
\label{computation_K_m(RG)}
\end{eqnarray}
The real $K$-theory of the building blocks are given by $KO_m(\R) = KO_m(\pt)$, $KO_m(\C) \linebreak
= K_m(\pt)$, and $KO_m(\IH) = KO_{m+4}(\pt)$.
If $G = \IZ/p$ for an odd prime $p$ and we take for $V_0$ the trivial
real $\IZ/p$-representation $\IR$, 
then $r = (p-1)/2$, $D_0 = \IR$ and $D_i = \IC$ for $i \in \{1,2,  \ldots (p-1)/2\}$.
This implies
\begin{eqnarray}
KO_m^{\IZ/p}(\pt) & \cong & KO_m(\pt) \oplus K_m(\pt)^{(p-1)/2}.
\label{KO_mZ/p(pt)}
\\
KO^m_{\IZ/p}(\pt) & \cong & KO_{-m}(\pt) \oplus K_{-m}(\pt)^{(p-1)/2}.
\label{KOn_Z/p(pt)}
\end{eqnarray}
Let $\widetilde{KO}_m^{\IZ/p}(\pt)$ be the kernel of the map
$KO_m^{\IZ/p}(\pt) \to KO_m(\pt)$ given by induction with $\IZ/p \to 1$.
This corresponds under the isomorphism~\eqref{KO_mZ/p(pt)} to the obvious projection
of  $KO_{m}(\pt) \oplus K_{m}(\pt)^{(p-1)/2}$ onto $KO_{m}(\pt)$. 
Let $\overline{KO}^m_{\IZ/p}(\pt)$ be the cokernel of the map
$KO^m(\pt) \to KO^m_{\IZ/p}(\pt)$ given by induction with $\IZ/p \to 1$.
This corresponds under the isomorphism~\eqref{KOn_Z/p(pt)} to the obvious inclusion
of $KO_{-m}(\pt)$ into $KO_{-m}(\pt) \oplus KO_{-m}(\pt)^{(p-1)/2}$. 
Hence we get
\begin{eqnarray}
\widetilde{KO}_m^{\IZ/p}(\pt) & \cong & K_m(\pt)^{(p-1)/2};
\label{widetildeKO_homology_Z/p(pt)}
\\
\overline{KO}^m_{\IZ/p}(\pt) & \cong & K_{-m}(\pt)^{(p-1)/2};
\label{widetildeKO_cohomology_Z/p(pt)}
\end{eqnarray}
This implies
\begin{eqnarray}
&\widetilde{KO}_m^{\IZ/p}(\pt) \cong \overline{KO}^m_{\IZ/p}(\pt)
\cong 
\begin{cases}
\IZ^{(p-1)/2} & m\; \text{even};
\\
0 & m \; \text{odd}.
\end{cases}
\label{widetildeKOZ/p(pt)}
\end{eqnarray}

We conclude from~\cite[Theorem~5.2]{Lueck(2005c)} using
Lemma~\ref{lem:preliminaries_about_Gamma_and_Zn_rho}~%
\ref{lem:preliminaries_about_Gamma_and_Zn_rho:ideals}
for $m \in \IZ$ 
\begin{eqnarray*}
KO^{2m}_{\Gamma}(\eub{\Gamma}) \otimes \IQ & \cong &
\IQ^{p^k(p-1)/2 + \sum_{l \in \IZ} r_{2m+4l}};
\\
KO^{2m+1}_{\Gamma}(\eub{\Gamma}) \otimes \IQ & \cong &
\IQ^{\sum_{l \in \IZ} r_{2m+1+4l}}.
\end{eqnarray*}
Again we seek an integral computation.

\begin{theorem}[Equivariant KO-cohomology]
\label{the:equivariant_KO-cohomology}
Let $p$ be an odd prime and let $m$ be any integer.
\begin{enumerate}

\item \label{the:equivariant_KO-cohomology:explicite}
$$
KO^{m}_{\Gamma}(\eub{\Gamma}) \cong
\begin{cases}
\IZ^{p^k(p-1)/2} \oplus \bigoplus_{i \in \IZ} KO^{m-i}(\pt)^{r_i} & m \: \text{even} \\
\bigoplus_{i \in \IZ} KO^{m-i}(\pt)^{r_i} & m \: \text{odd}.
\end{cases}
$$

\item \label{the:equivariant_KO-cohomology_some_isos}
If $TO^{2m+1}$ is the finite abelian $p$-group appearing in
Theorem~\ref{the:KO-cohomology_of_BGamma_and_bub(Gamma)}~%
\ref{the:KO-cohomology_of_BGamma_and_bub(Gamma):K12m_plus_1(bub(Gamma)},
then there is an exact sequence
$$0 \to KO^{2m}(\bub\Gamma) \to KO^{2m}_{\Gamma}(\eub{\Gamma}) \to
\bigoplus_{(P) \in \calp}  \overline{KO}^{2m}_{\IZ/p}(\pt) \to TO^{2m+1} \to 0.$$

\item \label{the:equivariant_KO-cohomology:some_isos}
The canonical maps
\begin{eqnarray*}
KO^{2m+1}_{\Gamma}(\eub{\Gamma}) & \xrightarrow{\cong} & KO^{2m+1}(B\Gamma);
\\
KO^{2m+1}(B\Gamma)  & \xrightarrow{\cong} & KO^{2m+1}(B\IZ^n_{\rho})^{\IZ/p},
\end{eqnarray*}
are isomorphisms.

\end{enumerate}
\end{theorem}
\begin{proof}~%
\ref{the:equivariant_KO-cohomology:some_isos}
Lemma~\ref{lem:long_exact_sequence_for_equivariant_(co-)homology}~%
\ref{lem:long_exact_sequence_for_equivariant_(co-)homology:cohomology} 
together with~\eqref{widetildeKOZ/p(pt)} implies
that there is a  long exact sequence
\begin{multline}
0 \to KO^{2m}(\bub{\Gamma}) \to KO^{2m}_{\Gamma}(\eub{\Gamma}) \to
\bigoplus_{(P) \in \calp}  \overline{KO}^{2m}_{\IZ/p}(\pt)
\\
\to KO^{2m+1}(\bub{\Gamma}) \to KO^{2m+1}_{\Gamma}(\eub{\Gamma}) \to 0,
\label{spliced_KO_cohomology_sequence}
\end{multline}
and that the kernel of the epimorphism
$KO^{2m+1}(\bub{\Gamma}) \to KO^{2m+1}_{\Gamma}(\eub{\Gamma})$ is a finite abelian
$p$-group.

For $m \in \IZ$ the composite
$$KO^{2m+1}(\bub{\Gamma}) \xrightarrow{\alpha} KO^{2m+1}_{\Gamma}(\eub{\Gamma}) \xrightarrow{\beta} KO^{2m+1}(B\Gamma)$$
is surjective and has a finite abelian $p$-group as kernel by
Theorem~\ref{the:KO-cohomology_of_BGamma_and_bub(Gamma)}~%
\ref{the:KO-cohomology_of_BGamma_and_bub(Gamma):KO_odd(BGamma)_to_KO_odd(bub(Gamma))}.
Hence the map
$\beta$
is surjective for all $m \in \IZ$. 
Since $\alpha$ is surjective
by~\eqref{spliced_KO_cohomology_sequence}, 
the map $\text{ker }(\beta \circ \alpha) \to \text{ker }(\beta)$ is surjective and hence the kernel of $\beta$ is a
finite abelian $p$-group.

The following diagram commutes
$$
\xymatrix{KO^{2m+1}_{\Gamma}(\eub{\Gamma}) \ar[r] \ar[d] \ar@/^5mm/[rr]^-{2 \cdot \id}&
K^{2m+1}_{\Gamma}(\eub{\Gamma}) \ar[r] \ar[d]^-{\cong} &
KO^{2m+1}_{\Gamma}(\eub{\Gamma})  \ar[d] 
\\
KO^{2m+1}(B\Gamma) \ar[r] \ar@/_5mm/[rr]_-{2 \cdot \id} &
K^{2m+1}(B\Gamma) \ar[r]  &
KO^{2m+1}(B\Gamma) 
}
$$
where the left horizontal maps are given by
induction with $\IR \to \IC$, the right horizontal maps by restriction with
$\IR \to \IC$ and the middle vertical arrow is an isomorphism by
Theorem~\ref{the:equivariant_K-cohomology_of_eub(Gamma)}.
Hence the kernel of the epimorphism
$KO^{2m+1}_{\Gamma}(\eub{\Gamma}) \to KO^{2m+1}(B\Gamma)$
is an abelian group of exponent $2$. 
We have already shown that its kernel is a finite abelian
$p$-group. Since $p$ is odd, we conclude that
$$KO^{2m+1}_{\Gamma}(\eub{\Gamma}) \xrightarrow{\cong} KO^{2m+1}(B\Gamma)$$
is an isomorphism. 

The bijectivity of 
$KO^{2m+1}(B\Gamma)  \xrightarrow{\cong} KO^{2m+1}(B\IZ^n_{\rho})^{\IZ/p}$
has already been proved in
Theorem~\ref{the:KO-cohomology_of_BGamma_and_bub(Gamma)}~%
\ref{the:KO-cohomology_of_BGamma_and_bub(Gamma):KO2m_plus_1(BGamma)}.
\\[1mm]~%
\ref{the:equivariant_KO-cohomology:explicite}
Since kernel of the epimorphism
$KO^{2m+1}(\bub{\Gamma}) \to KO^{2m+1}_{\Gamma}(\eub{\Gamma})$ is a finite abelian
$p$-group and  $\bigoplus_{(P) \in \calp}  \overline{KO}^{2m}_{\IZ/p}(\pt)$ is 
isomorphic to $\IZ^{p^k(p-1)/2}$ by
Lemma~\ref{lem:preliminaries_about_Gamma_and_Zn_rho}~%
\ref{lem:preliminaries_about_Gamma_and_Zn_rho:order_of_calp} 
and by~\eqref{widetildeKOZ/p(pt)}, we conclude from the exact 
sequence~\eqref{spliced_KO_cohomology_sequence} that
$$KO^{2m}_{\Gamma}(\eub{\Gamma}) \cong 
KO^{2m}(\bub{\Gamma}) \oplus \IZ^{p^k(p-1)/2}.$$
Since we have already computed $KO^{2m}(\bub{\Gamma})$ and $KO^{2m+1}(B\Gamma)$ in
Theorem~\ref{the:KO-cohomology_of_BGamma_and_bub(Gamma)}, 
assertion~\ref{the:equivariant_KO-cohomology:explicite} follows using
assertion~\ref{the:equivariant_KO-cohomology:some_isos}.
\\[1mm]~%
\ref{the:equivariant_KO-cohomology_some_isos}
The kernel of the epimorphism $KO^{2m+1}(\bub{\Gamma}) \to KO^{2m+1}(B\Gamma)$
is isomorphic to $TO^{2m+1}$ by 
Theorem~\ref{the:KO-cohomology_of_BGamma_and_bub(Gamma)}~%
\ref{the:KO-cohomology_of_BGamma_and_bub(Gamma):K12m_plus_1(bub(Gamma)} 
and~\ref{the:KO-cohomology_of_BGamma_and_bub(Gamma):KO_odd(BGamma)_to_KO_odd(bub(Gamma))}.
Since $KO^{2m+1}_{\Gamma}(\eub{\Gamma}) \xrightarrow{\cong} KO^{2m+1}(B\Gamma)$
is bijective by assertion~\ref{the:equivariant_KO-cohomology:some_isos},
the claim follows from the long exact  
sequence~\eqref{spliced_KO_cohomology_sequence}.
\end{proof}


\typeout{------------   Section 10: Equivariant KO-homology  ------------}

\section{Equivariant \texorpdfstring{$KO$}{KO}-homology}
\label{sec:Equivariant_KO-homology}

We obtain from~\cite[Theorem~0.7]{Lueck(2002d)} using
Lemma~\ref{lem:preliminaries_about_Gamma_and_Zn_rho} isomorphisms
\begin{eqnarray*}
  KO_{2m}^{\Gamma}(\eub{\Gamma}) \otimes \IQ & \cong &
  \IQ^{p^k(p-1)/2 + \sum_{l \in \IZ} r_{4l+2m}};
  \\
  KO_{2m+1}^{\Gamma}(\eub{\Gamma}) \otimes \IQ & \cong &
  \IQ^{\sum_{l \in \IZ} r_{4l+2m+1}}.
\end{eqnarray*}

We want to get an integral computation.

\begin{theorem}[Equivariant $KO$-homology]
  \label{the:Equivariant_KO-homology}
  Let $p$ be an odd prime and $m$ be any integer.
  \begin{enumerate}

  \item \label{the:Equivariant_KO-homology:KO}
$$
KO_m^\Gamma(\eub{\Gamma}) \cong
\begin{cases}
  \IZ^{p^k(p-1)/2} \oplus
  \left(\bigoplus_{i=0}^n  KO_{m-i}(\pt)^{r_{i}}\right) & m \; \text{even;}\\
  \bigoplus_{i=0}^n KO_{m-i}(\pt)^{r_{i}}. & m \: \text{odd.}
\end{cases}
$$

\item \label{the:Equivariant_KO-homology:odd} For $m \in \IZ$ the map $
  KO_{2m+1}^{\Gamma}(\eub{\Gamma}) \xrightarrow{\cong} KO_{2m+1}(\bub{\Gamma}) $
  is an isomorphism.

\item \label{the:Equivariant_KO-homology:even} There is a short exact sequence
  \begin{eqnarray*}
    & 0 \to \bigoplus_{(P) \in \calp} \widetilde{KO}_{2m}^{\IZ/p}(\pt) 
    \to KO_{2m}^{\Gamma}(\eub{\Gamma}) \to KO_{2m}(\bub{\Gamma}) \to 0,
  \end{eqnarray*}
  where $\widetilde{KO}^{\Z/p}_{2m}(\pt)$ is the kernel of the map
  ${KO}^{\Z/p}_{2m}(\pt) \to KO_{2m}(\pt)$ coming from induction with $\IZ/p  \to 1$.  It splits after inverting $p$.
\end{enumerate}

\end{theorem}

\begin{proof}
  Lemma~\ref{lem:long_exact_sequence_for_equivariant_(co-)homology}%
~\ref{lem:long_exact_sequence_for_equivariant_(co-)homology:homology} implies
  that there is an long exact sequence
  \begin{multline}
    0 \to KO_{2m+1}^{\Gamma}(\eub{\Gamma}) \to KO_{2m+1}(\bub{\Gamma}) \to
    \bigoplus_{(P) \in \calp} \widetilde{KO}_{2m}^{\IZ/p}(\pt)
    \\
    \to KO_{2m}^{\Gamma}(\eub{\Gamma}) \to KO_{2m}(\bub{\Gamma}) \to 0.
    \label{spliced_KO_homology_sequence}
  \end{multline}
  and that the map
$$KO_{i}^{\Gamma}(\eub{\Gamma})[1/p] \to KO_{i}(\bub{\Gamma})[1/p]$$
is split surjective for $i \in \IZ$. In particular the cokernel of
$KO_{2m+1}^{\Gamma}(\eub{\Gamma}) \to KO_{2m+1}(\bub{\Gamma})$ is a finite
abelian $p$-group.  Since $\widetilde{KO}_{2m}^{\IZ/p}(\pt)$ is a finitely
generated free abelian group by~\eqref{widetildeKOZ/p(pt)}, the long exact
sequence~\eqref{spliced_KO_homology_sequence} reduces to an isomorphism
\begin{eqnarray*}
  KO_{2m+1}^{\Gamma}(\eub{\Gamma}) & \xrightarrow{\cong} &  KO_{2m+1}(\bub{\Gamma})
\end{eqnarray*}
and a short exact sequence
\begin{eqnarray}
  & 0 \to \bigoplus_{(P) \in \calp} \widetilde{KO}_{2m}^{\IZ/p}(\pt) 
  \to KO_{2m}^{\Gamma}(\eub{\Gamma}) \to KO_{2m}(\bub{\Gamma}) \to 0,
  \label{short_exact_sequence_KO_split_1/p}
\end{eqnarray}
which splits after inverting $p$.  We have proven
assertions~\ref{the:Equivariant_KO-homology:odd}
and~\ref{the:Equivariant_KO-homology:even}.

Since the composite
$$KO_{i}^{\Gamma}(\eub{\Gamma}) \to K_{i}^{\Gamma}(\eub{\Gamma}) 
\to KO_{i}^{\Gamma}(\eub{\Gamma})$$ is multiplication with $2$ and
$K_{i}^{\Gamma}(\eub{\Gamma})$ is a finitely generated free abelian group by
Theorem~\ref{the:equivariant_K-homology_of_eub(Gamma)}, the torsion subgroup of
the finitely generated abelian group $KO_{i}^{\Gamma}(\eub{\Gamma})$ is
annihilated by $2$ for $i \in \IZ$.  Since by
Theorem~\ref{the:KO-homology_of_BGamma_and_bub(Gamma)}~%
\ref{the:KO-homology_of_BGamma_and_bub(Gamma):KO2m(bub(Gamma)}
\begin{eqnarray*}
  \bigoplus_{(P) \in \calp} \widetilde{KO}_{2m}^{\IZ/p}(\pt) 
  & \cong &
  \IZ^{p^k(p-1)/2};
  \\
  KO_{2m}(\bub{\Gamma}) 
  & \cong & 
  \left(\bigoplus_{i=0}^n  KO_{2m-i}(\pt)^{r_{i}}\right)
  \oplus TO^{2m+5}
\end{eqnarray*}
for a finite abelian $p$-group $TO^{2m+5}$ and the torsion in $\bigoplus_{i=0}^n
KO_{m-i}(\pt)^{r_{i}}$ is annihilated by multiplication with $2$, we get
from~\eqref{short_exact_sequence_KO_split_1/p} an isomorphism of abelian groups
$$KO_{2m}^{\Gamma}(\eub{\Gamma}) \cong \IZ^{p^k(p-1)/2} \oplus 
\left(\bigoplus_{i=0}^n KO_{2m-i}(\pt)^{r_{i}}\right).$$ This is the even case
of assertion~\ref{the:Equivariant_KO-homology:KO}.  The odd case of
assertion~\ref{the:Equivariant_KO-homology:KO} follows from
assertion~\ref{the:Equivariant_KO-homology:odd} and
Theorem~\ref{the:KO-homology_of_BGamma_and_bub(Gamma)}%
~\ref{the:KO-homology_of_BGamma_and_bub(Gamma):KO2m_plus_1(bub(Gamma)}.
\end{proof}


\typeout{------------   Section 11: Topological K-theory of the group C^*-algebra ------------}

\section{Topological \texorpdfstring{$K$}{K}-theory of the group \texorpdfstring{$C^*$}{C*}-algebra}

\label{sec:Topological_K-theory_of_the_group_Cast-algebra}

In this section we compute the topological $K$-theory $K_n(C^*_r(\Gamma))$ of
the complex reduced group $C^*$-algebra $C^*_r(\Gamma)$ and the topological
$K$-theory $KO_n(C^*_r(\Gamma;\IR)) := K_n(C^*_r(\Gamma;\IR))$ of the real
reduced group $C^*$-algebra $C^*_r(\Gamma;\IR)$.

The Baum-Connes Conjecture (see~\cite[Conjecture 3.15 on page
254]{Baum-Connes-Higson(1994)}) predicts for a group $G$ that the complex and
the real assembly maps
\begin{eqnarray}
  K^{G}_n(\eub{G}) & \xrightarrow{\cong} & K_n(C^*_r(G));
  \label{BCC_complex}
  \\
  KO^{G}_n(\eub{G}) & \xrightarrow{\cong} & KO_n(C^*_r(G;\IR)),
  \label{BCCreal}
\end{eqnarray}
are bijective for $n \in \IZ$. It has been proved for $G = \Gamma$ (and many
more groups) in~\cite{Higson-Kasparov(2001)}.


\subsection{The complex case}
\label{subsec:The_complex_case}
We begin with the complex case.

\begin{proof}[Proof of Theorem~\ref{the:Topological_K-theory_of_the_complex_group_Cast-algebra}]
  Because of the isomorphism~\eqref{BCC_complex} all claims follow from
  Lemma~\ref{lem:preliminaries_about_Gamma_and_Zn_rho}~%
\ref{lem:preliminaries_about_Gamma_and_Zn_rho:ideals},
  Lemma~\ref{lem:LambdajZ(zeta)_in_R_Q(Z/p)}%
~\ref{lem:LambdajZ(zeta)_in_R_Q(Z/p):r_odd_and_r_ev} and
  Theorem~\ref{the:equivariant_K-homology_of_eub(Gamma)} except the statement
  that
$$K_1(C^*_r(\Gamma)) \xrightarrow{\cong} K_1(C^*_r(\IZ^n_\rho))^{\IZ/p}$$
is bijective. Induction with $\iota \colon \IZ^n \to \Gamma$ yields a
homomorphism
$$ K_1(C^*_r(\IZ^n)) \to K_1(C^*_r(\Gamma))$$
and restriction with $\iota$ yields a homomorphism
$$ K_1(C^*_r(\Gamma)) \to K_1(C^*_r(\IZ^n)).$$
Since an inner automorphism of $\Gamma$ induces the identity on
$K_1(C^*_r(\Gamma))$, these homomorphisms induce homomorphisms
\begin{eqnarray*}
  \iota_* \colon K_1(C^*_r(\IZ^n_\rho)_{\IZ/p}  & \to & K_1(C^*_r(\Gamma));
  \\
  \iota^* \colon K_1(C^*_r(\Gamma )) & \to & K_1(C^*_r(\IZ^n_\rho))^{\IZ/p}.
\end{eqnarray*}
By the double coset formula the composite $\iota^* \circ \iota_*$ is the norm
map
$$N \colon K_1(C^*_r(\IZ^n_\rho))_{\IZ/p}  \to K_1(C^*_r(\IZ^n_\rho))^{\IZ/p}.$$ 
The cokernel of the norm map is
$\widehat{H}^0(\IZ/p; K_1(C^*_r(\IZ^n_\rho))$.  Note
\begin{align*}
  \widehat{H}^0(\IZ/p; K_1(C^*_r(\IZ^n_\rho)) 
& \cong \widehat{H}^0(\IZ/p; K_1(B\IZ^n_\rho)) 
& \text{(the BC Conjecture for $\Z^n$)}
\\
  & \cong \widehat{H}^0(\IZ/p; K^1(B\IZ^n_\rho)^*) 
& \text{(the UCT for $K$-theory~\ref{the:Universal_coefficient_theorem_for_K-theory})}
\\
  & \cong \widehat{H}^{-1}(\IZ/p; K^1(B\IZ^n_\rho)) 
&
  \text{(Lemma~\ref{Tate_duality} proven below)}
  \\
  & = 0 & \text{(Lemma~\ref{lem:Leray-Serre_ss_for_K(BGamma)}%
~\ref{lem:Leray-Serre_ss_for_K(BGamma):Tate_cohomology}).}
\end{align*}

This implies that the norm map $N$ and hence $\iota^* \colon K_1(C^*_r(\Gamma))
\xrightarrow{\cong} K_1(C^*_r(\IZ^n_\rho))^{\IZ/p}$ are surjective.  Since
source and target of $\iota^*$ are finitely generated free abelian groups of the
same rank by
assertion~\ref{the:Topological_K-theory_of_the_complex_group_Cast-algebra:explicite}
and Lemma~\ref{lem:Leray-Serre_ss_for_K(BGamma)}~%
\ref{lem:Leray-Serre_ss_for_K(BGamma):KOast(BZn)_over_Z[Z/p]}, $\iota^*$ is an
isomorphism.
\end{proof}


\subsection{The real case}
\label{subsec:The_real_case}

Next we treat the real case.

\begin{proof}[Proof of Theorem~\ref{the:Topological_K-theory_of_the_real_group_Cast-algebra}]
  Because of the isomorphisms~\eqref{widetildeKOZ/p(pt)} and~\eqref{BCCreal} all
  claims follow from Theorem~\ref{the:Equivariant_KO-homology} except the claim
  that
$$KO_{2m+1}(C^*_r(\Gamma;\IR)) \xrightarrow{\cong} KO_{2m+1}(C^*_r(\IZ^n_\rho;\IR))^{\IZ/p}$$
is bijective. As we have natural transformations of cohomology theories
$i^* \colon KO_* \to K_*$ and $r^* \colon K_* \to KO_*$ with $r^* \circ i^* = 2
\cdot\id$,
Theorem~\ref{the:Topological_K-theory_of_the_complex_group_Cast-algebra}~%
\ref{the:Topological_K-theory_of_the_complex_group_Cast-algebra:K1_abstract}
implies that the map is bijective after inverting $2$. Since $p$ is odd, it
remains to show that it is bijective after inverting $p$.  Because of the
bijectivity of $KO_{2m+1}(C^*_r(\Gamma;\IR)) \xrightarrow{\cong}
KO_{2m+1}(\bub{\Gamma})$, the fact that $KO_{2m+1}(B\IZ^n_{\rho})_{\IZ/p} \to
KO_{2m+1}(\bub{\Gamma})$ is bijective after inverting $p$ (use
Proposition~\ref{prop:quotient_iso}), the fact that norm map is always bijective
after inverting $p$, and the the isomorphism~\eqref{BCCreal} for $\Z^n$, the
claim holds.
\end{proof}


\typeout{------------   Section 12: The group 
Gamma satisfies the (unstable) Gromov-Lawson-Rosenberg Conjecture  ------------}

\section{The group   \texorpdfstring{$\Gamma$}{Gamma} 
satisfies the (unstable) Gromov-Lawson-Rosenberg Conjecture}

\label{sec:The_(unstable)_Gromov-Lawson-Rosenberg_Conjecture_holds_for_Gamma}

In this section we give the proof of
Theorem~\ref{the:The_(unstable)_Gromov-Lawson-Rosenberg_Conjecture_holds_for_Gamma},
after first giving some background.


\subsection{The Gromov-Lawson-Rosenberg Conjecture}
\label{subsec:The_Gromov-Lawson-Rosenberg_Conjecture}

For a closed, spin manifold $M$ of dimension $m$ with fundamental group $G$, one
can define an invariant
\begin{eqnarray}
  \alpha (M)  & \in & KO_m(C^*_r(G); \R), 
  \label{alpha(M)}
\end{eqnarray}
which vanishes if $M$ admits a metric of positive scalar 
curvature (see~\cite{Rosenberg(1986b)}).  The \emph{(unstable) Gromov-Lawson-Rosenberg
  Conjecture} for a group $G$ states that if $\alpha(M) = 0$ and $\dim M \geq
5$, then $M$ admits a metric of positive scalar curvature.  The (unstable)
Gromov-Lawson-Rosenberg Conjecture is known to be valid for some fundamental
groups, for example, the trivial group (see~\cite{Stolz(1992)}), for finite groups with periodic cohomology
(see~\cite{Botvinnik-Gilkey-Stolz(1997)} and~\cite{Kwasik-Schultz(1990)}),
some torsion-free infinite groups, for example,
when $G$ is a fundamental group of a complete Riemannian manifold of non-positive
sectional curvature (see~\cite{Rosenberg(1986b)}), and some infinite groups with
torsion, for example, cocompact Fuchsian groups (see~\cite{Davis-Pearson(2003)}), but
not in general -- there is a counterexample when $G = \IZ^4 \times \IZ/3$ due to
Schick~\cite{Schick(1998e)}.

There is a weaker version of the conjecture which may be valid for all groups.
Let $B^8$ be a ``Bott manifold,'' a simply-connected spin $8$-manifold with
$\widehat{A}$-genus equal to one.  We say that a manifold $M$ \emph{stably
  admits a metric of positive scalar curvature} if $M \times (B^8)^j$ admits a
metric of positive scalar curvature for some $j \geq 0$.  The \emph{stable
  Gromov-Lawson-Rosenberg Conjecture} formulated by
Rosenberg-Stolz~\cite{Rosenberg-Stolz(1995)} states that for a closed spin
manifold $M$ with fundamental group $G$, then $M$ stably admits a metric of
positive scalar curvature if and only if $\alpha(M) = 0$.  Since the Baum-Connes
Conjecture implies the stable Gromov-Lawson-Rosenberg Conjecture
(see~\cite[Theorem~3.10]{Stolz(2002)} for an outline of the proof) and $\Gamma$
satisfies the Baum-Connes Conjecture, we know already that $\Gamma$ satisfies
the stable Gromov-Lawson-Rosenberg Conjecture.

There are two definitions of the invariant $\alpha$, one topological and one
analytic.  Let $\bfKO$ be the periodic spectrum underlying real $K$-theory, and
let $\bfp\colon \bfko \to \bfKO$ be the $0$-connective cover, that is, it
induces an isomorphism on $\pi_i$ for $i \geq 0$ and $\pi_i(\bfko)= 0$ for $i$
negative.  Then the topological definition of $\alpha(M)$ is the image of the
class $[f_M\colon M \to BG]$ where $f_M$ induces the identity on the fundamental
group under the composite
\begin{eqnarray}
  \Omega^{\text{Spin}}_m(BG) \xrightarrow{D} ko_m(BG) \xrightarrow{p_{BG}} KO_m(BG)
  \xrightarrow{A} KO_m(C^*_r(G))
  \label{ApD}
\end{eqnarray}
where $D$ is the $ko$-orientation of spin bordism, $p_{BG}$ is the canonical map from connective to the periodic $K$-theory,
and $A$ is the assembly map. The analytic definition of $\alpha(M)$ is the index of the Dirac operator.  These two definitions agree
(see~\cite{Rosenberg(1986b)}).
Furthermore if $M$ has positive scalar curvature, then the
Bochner-Lichnerowicz-Weitzenb\"ock formula shows that the index is zero so that
$\alpha(M) =0$.

Finally, we mention one more result in our quick review, and that is the
generalization of the Gromov-Lawson surgery theorem of due to Jung and
Stolz~\cite[3.7]{Rosenberg-Stolz(1995)}.  
 
\begin{proposition}
  \label{prop:Stolz-Jung}
  Let $M$ be a connected closed spin manifold with fundamental group $G$ and
  dimension $m \ge 5$. Let $[f\colon N \to BG] \in \Omega^{\text{Spin}}_m(BG)$.
  (Note that $N$ need not have fundamental group $G$.)  If $D[f_M :M \to BG] =
  D[f\colon N \to BG] \in ko_m(BG)$ and $N$ admits a metric of positive scalar
  curvature, then so does $M$.

\end{proposition}


\subsection{The proof of
  Theorem~\ref{the:The_(unstable)_Gromov-Lawson-Rosenberg_Conjecture_holds_for_Gamma}}
\label{subsec:Proof_of_Theorem(the:The_(unstable)_Gromov-Lawson-Rosenberg_Conjecture_holds_for_Gamma)}

The proof of
Theorem~\ref{the:The_(unstable)_Gromov-Lawson-Rosenberg_Conjecture_holds_for_Gamma}
needs some preparation.

\begin{lemma}
  \label{lem_widetildeOmega_spin_to_widetilde_ko_is-surjective}
  Let $p$ be an odd prime. Then the map
$$\widetilde{D}\colon \widetilde{\Omega}^{\text{Spin}}_m(B\IZ/p) 
\to \widetilde{ko}_m(B\IZ/p)$$ is surjective for all $m \ge 0$.
\end{lemma}
\begin{proof}
  If $M$ is a $\IZ[\IZ/p]$-module, then $H_i(\IZ/p;M)[1/p] = 0$ for $i \ge 1$
  and hence the canonical maps
$$H_i(B\IZ/p;M) \xrightarrow{\cong} H_i(B\IZ/p;M)_{(p)}
\xrightarrow{\cong} H_i(B\IZ/p;M_{(p)})$$ are bijective for $i \ge 1$.  We
conclude from the Atiyah-Hirzebruch spectral sequences that the vertical maps in
the commutative diagram
$$\xymatrix@!C=10em{
  \widetilde{\Omega}^{\text{Spin}}_m(B\IZ/p) \ar[r]^{\widetilde{D}}
  \ar[d]^{\cong} & \widetilde{ko}_m(B\IZ/p) \ar[d]^{\cong}
  \\
  \widetilde{\Omega}^{\text{Spin}}_m(B\IZ/p)_{(p)} \ar[r]^{\widetilde{D}_{(p)}}
  & \widetilde{ko}_m(B\IZ/p)_{(p)} }
$$
are bijective for $m \ge 0$.  Hence it suffices to prove the surjectivity of the
lower horizontal map.  Since $p$ is odd, $\Omega^{\text{Spin}}_j(\pt)_{(p)}$ is
zero for $j \not\equiv 0 \pmod 4$ and $\Omega^{\text{Spin}}_j(\pt)_{(p)}$ is a
finitely generated free $\IZ_{(p)}$-module for $j \equiv 0 \pmod 4$ (see
\cite{Anderson-Brown-Peterson(1967)}).  The same is true for $ko_j(\pt)_{(p)}$ by
Bott periodicity. Hence there are no differentials in Atiyah-Hirzebruch spectral
sequences converging to $\widetilde{\Omega}^{\text{Spin}}_{i+j}(B\IZ/p)_{(p)}$
and $\widetilde{ko}_{i+j}(B\IZ/p)_{(p)}$ and we get for the $E^{\infty}$-terms
\begin{eqnarray*}
  E^{\infty}_{i,j}\left(\widetilde{\Omega}^{\text{Spin}}_{i+j}(B\IZ/p)_{(p)}\right)
  & \cong & 
  \widetilde{H}_i(\IZ/p) \otimes \Omega^{\text{Spin}}_j(\pt)_{(p)};
  \\
  E^{\infty}_{i,j}\left(\widetilde{ko}_{i+j}(B\IZ/p)_{(p)}\right)
  & \cong  & 
  \widetilde{H}_i(\IZ/p) \otimes ko_j(\pt)_{(p)}.
\end{eqnarray*}
It suffices to show that the map on the $E^{\infty}$-terms is surjective for all
$i,j$. Hence it is enough to show that the map
$$D_{(p)} \colon \Omega^{\text{Spin}}_j(\pt)_{(p)} \to ko_j(\pt)_{(p)}$$
is surjective for all $j$. Since $ko_*(\pt)_{(p)}$ is a polynomial algebra on a
single generator in dimension 4, it suffices to prove $D_{(p)}$ is onto when $j
= 4$.  In this case both $\Omega^{\text{Spin}}_4(\pt)$ and $ko_4(\pt)$ are
infinite cyclic with the former generated by a spin manifold of signature 16,
for example the Kummer surface $K$.  The $\hat A$-genus of $K$ is 2 and the
index of the real Dirac operator is $\widehat A(K)/2$ 
(see~\cite[Theorem II.7.10]{Lawson-Michelsohn(1989)}).  Hence $D\colon \Omega^{\text{Spin}}_4(\pt) \to
ko_4(\pt)$ is an isomorphism.
\end{proof}

\begin{theorem}[ko-homology]
  \label{the:ko-homology}
  Let $p$ be an odd prime and let $m$ be any integer.

  \begin{enumerate}

  \item \label{the:ko-homology:ko(BGamma)}
$$ko_m(B\Gamma) \cong
\begin{cases}
  \bigoplus_{i=0}^n  ko_{m-i}(\pt)^{r_{i}} & m \; \text {even;}\\
  to_{m}(B\Gamma) \oplus \left(\bigoplus_{i=0}^n ko_{m-i}(\pt)^{r_{i}}\right) &
  m \; \text {odd.}
\end{cases}
$$
where $to_m(B\Gamma)$ is a finite abelian $p$-group defined for $m$ odd.

\item \label{the:ko-homology:ko_2m(BGamma)} The inclusion map $\Z^n \to \Gamma$
  induces an isomorphism
$$ko_{2m}(B\IZ^n_{\rho})_{\IZ/p}  \xrightarrow{\cong} ko_{2m}(B\Gamma)$$
and $ko_{2m}(B\IZ^n_{\rho})_{\IZ/p} \cong \bigoplus_{i=0}^n
ko_{2m-i}(\pt)^{r_{i}}.  $

\item \label{the:ko-homology:long_exact_sequence} There is a long exact sequence
  \begin{multline*}
    0 \to ko_{2m}(B\Gamma) \xrightarrow{\overline{f}_{2m}} ko_{2m}(\bub{\Gamma})
    \xrightarrow{\partial_{2m}} \bigoplus_{(P) \in \calp}
    \widetilde{ko}_{2m-1}(BP)
    \\
    \xrightarrow{\varphi_{2m-1}} ko_{2m-1}(B\Gamma)
    \xrightarrow{\overline{f}_{2m-1}} ko_{2m-1}(\bub{\Gamma}) \to 0.
  \end{multline*}
  Hence $ko_{m}(B\Gamma)[1/p] \to ko_{m}(\bub{\Gamma})[1/p]$ is an isomorphism
  for $m \in \IZ$. 

\item \label{the:ko-homology:ko_2m_plus_1(bub(Gamma))}
$$ko_{2m+1}(\bub{\Gamma}) \cong 
\bigoplus_{i=0}^{2m+1} ko_{2m+1-i}(\pt)^{r_{i}}.
$$

\item \label{the:ko-homology:other_values} Let $to_{2m}(\bub\Gamma) =
  \im \partial_{2m}$ and $to_{2m-1}(B\Gamma) = \im \varphi_{2m-1}$.  These are
  finite abelian $p$-groups.  There is an exact sequence
$$0 \to ko_{2m}(B\Gamma) \to ko_{2m}(\bub{\Gamma}) \to to_{2m}(\bub\Gamma) \to 0$$ 
and an isomorphism
$$
ko_{2m+1}(B\Gamma) \cong to_{2m+1}(B\Gamma) \oplus \bigoplus_{i=0}^n
ko_{2m+1-i}(\pt)^{r_{i}}.
$$

\end{enumerate}
\end{theorem}
\begin{proof}~\ref{the:ko-homology:long_exact_sequence} The Atiyah-Hirzebruch
  spectral sequence implies that that $\widetilde{ko}_{2m}(B\IZ/p)$ vanishes and
  that $\widetilde{ko}_{2m+1}(B\IZ/p)$ is a finite abelian $p$-group.  Now the
  claim follows from the long exact sequence associated to the cellular
  pushout~\eqref{pushout_for_BGamma_to_bunderbar(Gamma)}.  
\\[1mm]~\ref{the:ko-homology:ko_2m(BGamma)} 
  The proof is similar to that of
  Theorem~\ref{the:Homology_of_BGamma_and_bub(Gamma)}%
~\ref{the:Homology_of_BGamma_and_bub(Gamma):coinvariants}.  We analyze the
  Leray-Serre spectral sequence associated to the
  extension~\eqref{Gamma_as_extension}
$$
E^2_{i,j} = H_i(\IZ/p;ko_j(B\IZ^n_{\rho})) \Rightarrow ko_{i+j}(B\Gamma).
$$
One can show analogously to the proof of Lemma~\ref{lem_KO(BZn)_over_Z[Z/p]}
that there are isomorphisms of $\IZ[\IZ/p]$-modules
\begin{align}
  \label{the:ko-homology:ko_2m(BGamma):iso_away_2}
  ko_j(B\IZ^n_{\rho}) \otimes \Z[1/2] & \cong
  \bigoplus_{l=0}^n H_{l}(\IZ^n_{\rho}) \otimes ko_{j-l}(\pt) \otimes \Z[1/2] \\
  \label{the:ko-homology:ko_2m(BGamma):iso_at_2}
  ko_j(B\IZ^n_{\rho}) \otimes {\Z_{(2)}} & \cong \bigoplus_{l=0}^n
  H_{l}(\IZ^n_{\rho}) \otimes ko_{j-l}(\pt) \otimes {\Z_{(2)}}.
\end{align}

Since $ko_m(\pt)_{(p)}$ is $\Z_{(p)}$ when $m$ is divisible by four and vanishes
otherwise,
$$
\widehat H^{i+1}(\Z/p; ko_j(B\IZ^n_{\rho})) \cong \bigoplus_{\ell}\widehat
H^{i+1}(\Z/p; H_{j-4\ell}(\IZ^n_{\rho}) ).
$$
This fact, the Universal Coefficient Theorem, Lemma~\ref{Tate_duality}, and
Lemma~\ref{lem:Hochschild-Serre_ss_forKast(BGamma)}~%
\ref{lem:Hochschild-Serre_ss_forKast(BGamma):Tate_cohomology} imply $\widehat
H^{i+1}(\Z/p; ko_j(B\IZ^n_{\rho})) = 0$ when $i + j$ is even.
 
Thus $E^2_{0,2m} = ko_{2m}(B\Z^n_\rho)_{\Z/p}$ maps injectively to
$ko_{2m}(B\Z^n_\rho)^{\Z/p}$ and hence is $p$-torsion-free, and for $i >0$,
$E^2_{i,j}$ has exponent $p$ and vanishes if $i + j$ is even.  Thus
$$
ko_{2m}(B\Z^n_\rho)_{\Z/p} \cong E^2_{0,2m} = E^\infty_{0,2m}
\xrightarrow{\cong} ko_{2m}(B\Gamma).
$$

By\eqref{the:ko-homology:ko_2m(BGamma):iso_away_2},~\eqref{the:ko-homology:ko_2m(BGamma):iso_at_2}
and Theorem~\ref{the:Homology_of_BGamma_and_bub(Gamma)}~%
\ref{the:Homology_of_BGamma_and_bub(Gamma):Hm(BGamma)},%
~\ref{the:Homology_of_BGamma_and_bub(Gamma):coinvariants}
,
$$
ko_{2m}(B\IZ^n_{\rho})_{\Z/p} \cong \bigoplus_{l=0}^n H_{l}(\IZ^n_{\rho})_{\Z/p}
\otimes ko_{2m-l}(\pt) \cong \bigoplus_{l=0}^n ko_{2m-l}(\pt)^{r_l}
$$
\\[1mm]~%
\ref{the:ko-homology:ko_2m_plus_1(bub(Gamma))} We will compute the group
$ko_{2m+1}(\bub\Gamma)$ after localizing at $p$ and after inverting $p$.  We
will begin with localizing at $p$.  We use the Atiyah-Hirzebruch spectral
sequence
$$
E^2_{i,j} =H_i(\bub\Gamma; ko_j(\pt)_{(p)}) \Rightarrow ko_{i+j}(\bub
\Gamma)_{(p)}
$$
for the generalized homology theory $ko_m(-)_{(p)}$.  Note also that when $i$ is
odd, Theorem~\ref{the:Homology_of_BGamma_and_bub(Gamma)}~%
\ref{the:Homology_of_BGamma_and_bub(Gamma):bub(Gamma)} states that
$H_i(\bub\Gamma) \cong \Z^{r_i}$.  In particular, when $i+j$ is odd, $E^2_{i,j}$
is finitely generated free over $\Z_{(p)}$.  Since the differentials in the
Atiyah-Hirzebruch spectral sequence are rationally trivial, $E^{\infty}_{i,j}
\subset E^2_{i,j}$ and has finite $p$-power index whenever $i + j$ is odd.
Hence
$$
ko_{2m+1}(\bub\Gamma)_{(p)} \cong \bigoplus_i E^\infty_{i,2m+1-i} \cong
\bigoplus_i \left(ko_{2m+1-i}(\pt)^{r_i}\right)_{(p)}.
$$
Now we invert $p$.  For any integer $j \geq 0$,
\begin{align*} ko_j(\bub \Gamma)[1/p]
  & \xleftarrow{\cong} ko_j(B\Z^n_\rho)_{\Z/p}[1/p] & (\text{Proposition~\ref{prop:quotient_iso}})\\
  & \cong \oplus_i H_{i}(B\Z^n_\rho)_{\Z/p} \otimes ko_{j-i}(\pt)[1/p] &
  (\text{isomorphisms \eqref{the:ko-homology:ko_2m(BGamma):iso_away_2},
    \eqref{the:ko-homology:ko_2m(BGamma):iso_at_2}})
  \\
  & \cong \oplus_i \left(ko_{j-i}(\pt)^{r_i}\right)[1/p] &
  \text{Theorem~\ref{the:Homology_of_BGamma_and_bub(Gamma)}%
~\ref{the:Homology_of_BGamma_and_bub(Gamma):Hm(BGamma)},%
~\ref{the:Homology_of_BGamma_and_bub(Gamma):coinvariants}}
\end{align*}
\\[1mm]~
\ref{the:ko-homology:other_values} The group $to_{2m}(\bub\Gamma)$ is a subgroup
and the group $to_{2m-1}(B\Gamma)$ is a quotient group of the finite abelian
$p$-group $\widetilde{ko}_{2m}(B\IZ/p)$, hence are finite abelian $p$-groups
themselves.  To complete the proof of
assertion~\ref{the:ko-homology:other_values}, by
assertions~\ref{the:ko-homology:long_exact_sequence}
and~\ref{the:ko-homology:ko_2m_plus_1(bub(Gamma))} we only need prove that
$\overline{f}_{2m+1}$ is a split surjection.  This follows since
$ko_{2m+1}(\bub\Gamma)_{(p)}$ is free over $\Z_{(p)}$ and $\overline{f}_{2m+1}
\otimes \id_{\Z[1/p]}$ is an isomorphism.  \\[1mm]~%
\ref{the:ko-homology:ko(BGamma)} This follows from
assertions~\ref{the:ko-homology:ko_2m(BGamma)}
and~\ref{the:ko-homology:other_values}.
\end{proof}

Now we are ready to prove
Theorem~\ref{the:The_(unstable)_Gromov-Lawson-Rosenberg_Conjecture_holds_for_Gamma}.

\begin{proof}[Proof of 
Theorem~\ref{the:The_(unstable)_Gromov-Lawson-Rosenberg_Conjecture_holds_for_Gamma}] 
Let $M$ be a closed $m$-dimensional manifold with $m \ge 5$ and fundamental
group $\pi_1(M) \cong \Gamma$.  Suppose that $\alpha(M) = 0$. We have to show
that $M$ carries a metric with positive scalar curvature.

The following commutative diagram with exact rows is key to the proof.
$$
\xymatrix{\bigoplus_{(P) \in \calp} \widetilde{ko}_{m}(BP) \ar[r] &
  ko_{m}(B\Gamma) \ar[d]_{A \circ p_{B\Gamma}} \ar[r]^{\beta} &
  ko_{m}(\bub{\Gamma}) \ar[d]_{p_{\bub{\Gamma}}}
  \\
  & KO_{m}(C^*_r(\Gamma;\IR)) \ar[r] & KO_{m}(\bub{\Gamma}) }
$$
where the bottom map is the composite of the inverse of the Baum-Connes map
$KO_m^\Gamma (\eub\Gamma) \to KO^\Gamma_m(C^*_r(\Gamma;\R))$ (which is an
isomorphism by \cite{Higson-Kasparov(2001)}) and the  map
$KO_m^\Gamma (\eub\Gamma) \to KO_m (\bub\Gamma)$ coming from induction with $\Gamma \to 1$.  
The top row is exact by Theorem~\ref{the:ko-homology}~%
\ref{the:ko-homology:long_exact_sequence}. The square commutes since the map
$p_{\bub{\Gamma}} \circ \beta$ equals the composite
$$
ko_m(B\Gamma) \to KO_m(B\Gamma) = KO^\Gamma_m(E\Gamma) \to
KO_m^\Gamma(\eub\Gamma) \to KO_m(B\Gamma).
$$

Since by assumption $\alpha(M) = 0$, the image of $D[f_M \colon M \to B\Gamma]
\in ko_m(B\Gamma)$ under the composite $p_{\bub{\Gamma}} \circ \beta$ is zero,
where $f_M \colon M \to B\Gamma$ is the classifying map of $M$ associated to
$\pi_1(M) \cong \Gamma$.

Next we show that the map $p_{\bub{\Gamma}}[1/p]$ is injective.  Because of
Proposition~\ref{prop:quotient_iso}, it suffices to show
$ko_{m}(B\IZ^n_{\rho})_{\IZ/p}[1/p] \to KO_{m}(B\IZ^n_{\rho})_{\IZ/p}[1/p] $ is
injective.  Since $p$ divides the order of $\IZ/p$ it suffices to show that $ ko_{m}(B\IZ^n) \to KO_{m}(B\IZ^n)$
is injective. This follows from the following commutative square
$$\xymatrix{
  \bigoplus_{l=0}^n \bigl(ko_{m-l}(\pt)\bigr)^{\binom{n}{l}} \ar[r]_-{\cong}
  \ar[d]_-{\bigoplus_{l=0}^n \bigl(p_{\pt}\bigr)^{\binom{n}{l}}} &
  ko_{m}(B\IZ^n) \ar[d]_-{p_{B\IZ^n}}
  \\
  \bigoplus_{l=0}^n \bigl(KO_{m-l}(\pt)\bigr)^{\binom{n}{l}} \ar[r]_-{\cong} &
  KO_{m}(B\IZ^n) }
$$
since $p_{\pt} \colon ko_{m}(\pt) \to KO_m(\pt)$ is injective for all $m \in
\IZ$.  This finishes the proof that the kernel of the map $p_{\bub{\Gamma}}$
consists of $p$-torsion. Hence $\beta (D[f_M \colon M \to B\Gamma]) \in
ko_m(\bub\Gamma)$ is $p$-torsion.

Now we can finish the proof in the case that $m$ is even. Then the map $\beta$
is injective and its domain is a finitely generated abelian group without
$p$-torsion by
Theorem~\ref{the:ko-homology}~\ref{the:ko-homology:ko_2m(BGamma)}
and~\ref{the:ko-homology:long_exact_sequence}. Hence $D[f_M \colon M \to B\Gamma] \in
ko_m(B\Gamma)$ is trivial and we conclude from Proposition~\ref{prop:Stolz-Jung}
that $M$ carries a metric with positive scalar curvature.

Hence we will now assume that $m$ is odd. Then the target of $\beta$ is a
finitely generated abelian group without $p$-torsion by
Theorem~\ref{the:ko-homology}~\ref{the:ko-homology:ko_2m_plus_1(bub(Gamma))}.
Hence the image of $D[f_M \colon M \to B\Gamma] \in ko_m(B\Gamma)$ under $\beta$
is zero.  We conclude from Theorem~\ref{the:ko-homology}~%
\ref{the:ko-homology:long_exact_sequence} that there is an element
$$(x_P)_{(P) \in \calp} \in \bigoplus_{(P) \in \calp} \widetilde{ko}_{m}(BP)$$
which is mapped under $\bigoplus_{(P) \in \calp} \widetilde{ko}_{m}(BP) \to
ko_m(B\Gamma)$ to $D[f_M \colon M \to B\Gamma]$. Combining this with
Lemma~\ref{lem_widetildeOmega_spin_to_widetilde_ko_is-surjective} yields
elements $[N_P \to BP] \in \widetilde \Omega^{\text{Spin}}_m(B\IZ/p)$ such that
the image of $[N_P \to BP]_{(P) \in \calp}$ under the composite
$$
\bigoplus_{(P) \in \calp} \widetilde \Omega^{\text{Spin}}_m(BP) \to
\Omega^{\text{Spin}}_m(B\Gamma) \xrightarrow{D} ko_m(B\Gamma)
$$
agrees with $D[f_M \colon M \to B\Gamma]$.  By surgery we can arrange that the
map $N_P \to BP$ is $2$-connected and in particular a classifying map for
$N_P$.  Since $m$ is odd,
$\widetilde{KO}_m(C^*_r(P;\R)) = 0$ (see the beginning of
Section~\ref{sec:Equivariant_KO-cohomology}).  Hence since the
Gromov-Lawson-Rosenberg conjecture holds for manifolds whose fundamental group
is odd-order cyclic \cite{{Kwasik-Schultz(1990)}}, each $N_P$ admits a metric of
positive scalar curvature.  Recall
$$
D[f_M\colon M\to B\Gamma] = D[\left(\amalg_{P \in (\calp)} N_P\right) \to
\left(\amalg_{P \in (\calp)}BP\right) \to B\Gamma ] \in ko_m(B\Gamma).
$$
Hence, by Proposition~\ref{prop:Stolz-Jung}, $M$ admits a metric of positive
scalar curvature.  

Now we just need to show that the last sentence of
Theorem~\ref{the:The_(unstable)_Gromov-Lawson-Rosenberg_Conjecture_holds_for_Gamma}
is valid.

Let $M$ be a closed spin manifold with odd dimension $m \geq 5$ and fundamental
group $\Gamma$.  Suppose that its $p$-cover $\widehat M$ associated with the
subgroup $\iota \colon \Z^n \to \Gamma$ admits a metric of positive scalar curvature.
Then $0 = \alpha(\widehat M) = \iota^*\alpha(M) \in KO_m(C^*_r(\Z^n; \R))$.
Hence by Theorem~\ref{the:Topological_K-theory_of_the_real_group_Cast-algebra}%
~\ref{the:Topological_K-theory_of_the_real_group_Cast-algebra:K1_abstract},
$\alpha(M) = 0$.  Hence by our argument above, $M$ admits a metric of positive
scalar curvature. 
\end{proof}


\typeout{---------------------------------------- Appendix   ---------------------------------}

\begin{appendix}
  \section{Tate cohomology, duality, and transfers} \label{tate_and_transfer}

  Here we collect facts concerning duality in Tate cohomology, transfers in
  generalized (co)-homology theories, and edge homomorphisms in the Leray-Serre
  spectral sequence.

  Recall that $\widehat{H}^*(G;M)$ denotes the \emph{Tate cohomology}
  (see~\cite[VI.4]{Brown(1982)}) of a finite group $G$ with coefficients in a
  $\IZ [G]$-module $M$, that $\widehat{H}^i(G;M) = H^i(G;M)$ for $i \ge 1$, that
  $\widehat{H}^i(G;M) = H_{-i-1}(G;M)$ for $i \le -2$, and that there is an
  exact sequence
$$\label{norm_exact_sequence} 0 \to \widehat{H}^{-1}(G;M) \to M_G \xrightarrow{N} M^G \to 
\widehat{H}^{0}(G;M) \to 0.$$ Here $M^G $ are the \emph{invariants} of $M$, $
M_G = M \otimes_{\IZ G } \Z = M/\langle gm - m \rangle_{g \in G, m\in M}$ are
the \emph{coinvariants} of $M$, and $N[m] = \sum_{g \in G} gm$ is the \emph{norm
  map}.  Note $M^G = H^0(G; M)$ and $M_G= H_0(G;M)$.

For a abelian group $M$, define the \emph{dual $M^* = \hom_\Z(M,\Z)$} and the
\emph{torsion dual $M^\wedge = \hom_\Z(M,\IQ/\Z)$}.  Note that if $M$ is a
finitely generated free abelian group (respectively a finite abelian group) then
there is a non-canonical isomorphism $M \cong M^*$ (respectively $M \cong
M^\wedge$).  If $M$ is a left $\IZ G$-module, give $M^*$ and $M^\wedge$ the
structure of left $\IZ G$-modules by defining $(g\varphi)(m) :=
\varphi(g^{-1}m)$ for $g \in G$ and $m \in M$.

\begin{lemma}[Tate duality] \label{Tate_duality} Let $G$ be a finite group and
  $M$ be a finitely generated $\IZ G $-module which contains no $p$-torsion for
  all primes $p$ dividing the order of $G$.  Then for all integers $i$ there is
  an isomorphism of abelian groups
$$\widehat{H}^i(G;M) \cong \widehat H^{-i}(G;M^*).$$ 
Hence for all integers $i >0$,
$$
H^{i+1}(G;M) \cong H_i(G;M^*).
$$
\end{lemma}

\begin{proof}
  The Tate cohomology group $\widehat{H}^i(G;M)$ is a finitely generated group
  of exponent $|G|$, hence is a finite abelian group.  Thus there is a
  non-canonical isomorphism of abelian groups $\widehat H^i(G; M) \cong \widehat
  H^i(G; M)^\wedge$.  Duality in Tate cohomology shows
$$
\widehat H^i(G; M)^\wedge \cong \widehat H^{-i-1}(G; M^\wedge)
$$ (see \cite[VI.7.3]{Brown(1982)}; duality holds for any $\IZ G$-module).  
Let $FM$ be $M$ modulo its torsion subgroup.  Then $(FM)^* \to M^*$ 
and $(FM)^\wedge \otimes \Z_{(|G|)} \to M^\wedge \otimes \Z_{(|G|)}$ are isomorphisms and 
$$
0 \to \hom_\Z(FM,\Z) \to \hom_\Z(FM,\IQ) \to \hom_\Z(FM,\IQ/\Z) \to 0
$$
is a short exact sequence.  Thus
$$
\widehat H^{-i-1}(G; M^\wedge) \cong \widehat H^{-i-1}(G; (FM)^\wedge) \cong
\widehat H^{-i}(G; (FM)^*) \cong \widehat H^{-i}(G; M^*)
$$
as desired.
\end{proof}

\begin{remark}
  \label{invariants_and_coinvariants_are_dual}
  Here is a related remark.  Let $G = \langle g \rangle$ be a finite cyclic
  group and $M$ be a $\IZ G$-module.  Then by dualizing the exact sequence
$$
M \xrightarrow{g-1} M \to M_G \to 0
$$
one obtains the exact sequence
$$
0 \to (M_G)^* \to M^* \xrightarrow{g^{-1}-1} M^*.
$$
Hence $(M_G)^* \cong (M^*)^G$.
\end{remark}

Let $\pi\colon E \to B$ be a regular $G$-cover of $CW$-complexes.  Let $\calh_*$ a
generalized homology theory and $\calh^*$ a generalized cohomology theory.
There are transfer maps $\trf_*$ and $\trf^*$ switching the domain and range of
$\pi_*$ and $\pi^*$.  Their definition is given in \cite[Chapter
4]{Adams(1978a)} when $B$ is finite and in \cite[Chapter IV,
3]{Lewis-May-Steinberger(1986)} in general.  All four maps are $G$-equivariant
with respect to the induced $G$-action on $\calh_*(E)$ and the trivial
$G$-action on $\calh_*(B)$ and $\calh^*(B)$.  Hence we have maps
\begin{align*}
  \pi_*  \colon \calh_*(E)_G & \to \calh_*(B)\\
  \trf_*  \colon \calh_*(B) & \to \calh_*(E)^G\\
  \pi^*  \colon \calh^*(B) & \to \calh^*(E)^G\\
  \trf_* \colon \calh^*(E)_G & \to \calh^*(B)
\end{align*}
The basic theorem connecting the two is this special case of the double coset
formula \cite[Corollary 6.4, p.~206]{Lewis-May-Steinberger(1986)}.
\begin{theorem} \label{norm_is_composite} Both $\trf_* \circ \pi_*$ and $\pi^*
  \circ \trf^*$ are given by the norm map, i.e.~multiplication by $\sum_{g \in
    G} g$.
\end{theorem}

For ordinary (co)homology theory, $\pi_* \circ \trf_*$ and $\trf^* \circ~\pi^*$
are both multiplication by $q = |G|$.  This has the consequence that $\pi_*$ and
$\pi_*$ are isomorphisms after inverting $q$.  These last composite formulae are
no longer true for generalized (co)homology theories, but one can say something.

A generalized homology theory is $1/q$-local if $\calh_*(X) \otimes \Z \to
\calh_*(X) \otimes \Z[1/q]$ is an isomorphism for all $X$ and $m$.  For example,
for any generalized homology theory, $\calh_*(X) \otimes \Z[1/q]$ is a
$1/q$-local generalized homology theory.  There is an analogous definition and
remark for generalized cohomology theories.

\begin{proposition}\label{prop:quotient_iso} Let $G$ be a finite group of order
  $q$.  Let $\calh_*$ and $\calh^*$ be $1/q$-local (co)homology theories.  Let
  $X$ be a $G$-$CW$-complex and $\pi \colon X \to \overline{X}$ the quotient map.
  \begin{enumerate}
  \item $\pi_m \colon \calh_m(X)_G \xrightarrow{\cong} \calh_m(\overline{X})$ is an
    isomorphism for all $m \in \Z$.
  \item If $X$ is a finite $CW$-complex, then $\pi^m \colon \calh^m(\overline X)
    \xrightarrow{\cong} \calh^m(X)^G$ is an isomorphism for all $m \in \Z$.
  \end{enumerate}
\end{proposition}

\begin{proof} We give the argument only for homology, the one for cohomology is
  analogous.

  Given a $G$-$CW$-complex $X$, we obtain a natural map
$$j_* \colon \calh_*(X)_G
\to \calh_*(G\backslash X)$$ Since the functor sending a $\IZ[1/q][G]$-module
$M$ to $M_G$ is an exact functor, the assignment sending a $G$-$CW$-complex $X$
to $\calh_*(X)_G$ and to $\calh_*(G\backslash X)$ are $G$-homology theories and
$j_*$ is a natural transformation of $G$-homology theories.  One easily checks
that $j_*$ is a bijection when $X$ is $G/H$ for any subgroup $H \subset G$. A
Mayer-Vietoris argument implies that $j_*$ is a bijection for any finite
$G$-$CW$-complex, and, since homology commutes with colimits, $j_*$ is a
bijection for any $G$-$CW$-complex.
\end{proof}
Atiyah's computation of $K^0(B\Z/p)$ shows that a finiteness hypothesis is
necessary for a generalized cohomology theory.

At several places in this paper we use a property of edge homomorphisms in
spectral sequences and we review this now.  Let $\calh_*$ and $\calh^*$ be
(co)homology theories.  Let $F \to E \to B$ be a fibration. Assume that $B$ is
path-connected with fundamental group $G$.  There are Leray-Serre spectral
sequences
\begin{align*}
  E^2_{i,j} = H_i(B;\calh_j(F)) & \Rightarrow \calh_{i+j}(E) \\
  E_2^{i,j} = H^i(B;\calh^j(F)) & \Rightarrow \calh^{i+j}(E).
\end{align*}
These spectral sequences have coefficients twisted by the action of $G$ on the
(co)homology of the fiber, in particular
\begin{align*}
  E^2_{0,j} \cong H_0(G;\calh_j(F)) & = \calh_{j}(F)_G \\
  E_2^{0,j} \cong H^0(G;\calh^j(F)) & = \calh^{j}(F)^G.
\end{align*}
The spectral sequences give maps
\begin{align*}
  H_j(F)_G \cong E^2_{0,j} \twoheadrightarrow & E^\infty_{0,j} \rightarrowtail \calh_j(E) \\
  \calh^j(E) \twoheadrightarrow & E_\infty^{0,j} \rightarrowtail E_2^{0,j} \cong
  H^j(F)^G;
\end{align*}
the composites are called the \emph{edge homomorphisms}.

The proof of the proposition below follows the proof in the untwisted case
\cite[page 354]{Switzer(1975)}.

\begin{proposition}[Edge homomorphisms] \label{prop:edge_homomorphisms} The edge
  homomorphisms
  \begin{align*}
    \calh_j(F)_G & \to \calh_j(E) \\
    \calh^j(E) & \to \calh^j(F)^G \\
  \end{align*}
  equal the maps on (co)homology induced by the inclusion of the fiber $F \to
  E$.
\end{proposition}

\end{appendix}




\end{document}